\documentclass{tacNoCP}

\usepackage{enumerate}
\usepackage{amsmath,amssymb}
\usepackage{mathrsfs}
\usepackage{hhline}
\usepackage{hyperref}
\usepackage{xcolor}
\hypersetup{colorlinks=false, linkbordercolor=cyan, citecolor=green}

\title{A simplicial foundation for \\ differential and sector forms\\\vspace{3.7pt}in tangent categories}
\author{G. S. H. Cruttwell and R. B. B. Lucyshyn-Wright}
\thanks{First author supported by an NSERC Discovery Grant and second author by an AARMS Postdoctoral Fellowship.  The second author gratefully acknowledges further financial support in the form of a Mount Allison University Research Stipend.  Thanks to Robin Cockett for useful discussions.}
\address{Department of Mathematics and Computer Science, \\ 
Mount Allison University, Sackville, Canada.}
\eaddress{gcruttwell@mta.ca, rlucyshynwright@mta.ca}
\copyrightyear{2016}

\input xy
\xyoption{all}
\xyoption{2cell}
\UseAllTwocells

\numberwithin{equation}{subsection}

\newtheorem{observation}{Remark}[section]
\newtheorem{lemma}[observation]{Lemma}  %%share counter with remark
\newtheorem{theorem}[observation]{Theorem}
\newtheorem{proposition}[observation]{Proposition} 
\newtheorem{corollary}[observation]{Corollary} 
\newtheoremrm{definition}[observation]{Definition}
\newtheoremrm{remark}[observation]{Remark}
\newtheoremrm{example}[observation]{Example}
\newtheoremrm{notation}[observation]{Notation}

\newcommand{\<}{\langle}
\renewcommand{\>}{\rangle}
\newcommand{\p}{\pi}

\newcommand{\A}{\ensuremath{\mathscr A}\xspace}
\newcommand{\B}{\ensuremath{\mathscr B}\xspace}
\newcommand{\C}{\ensuremath{\mathscr C}\xspace}
\newcommand{\D}{\ensuremath{\mathscr D}\xspace}
\newcommand{\E}{\ensuremath{\mathscr E}\xspace}

\newcommand{\N}{\ensuremath{\mathbb N}\xspace}
\newcommand{\R}{\ensuremath{\mathbb R}\xspace}
\newcommand{\T}{\ensuremath{\mathbb T}\xspace}
\newcommand{\V}{\ensuremath{\mathscr V}\xspace}
\newcommand{\X}{\ensuremath{\mathscr X}\xspace}

\newcommand{\op}{{\mbox{\textnormal{\scriptsize op}}}}

\newcommand{\set}{\ensuremath{\textnormal{\textbf{set}}}\xspace}

\newcommand{\cmon}{\ensuremath{\textnormal{\textbf{cmon}}}\xspace}
\newcommand{\ab}{\ensuremath{\textnormal{\textbf{ab}}}\xspace}
\newcommand{\cochain}{\ensuremath{\textnormal{\textbf{cochain}}}\xspace}
\newcommand{\finCard}{\ensuremath{\textnormal{\textbf{finCard}}}\xspace}
\newcommand{\finOrd}{\ensuremath{\textnormal{\textbf{finOrd}}}\xspace}
\newcommand{\Cart}{\ensuremath{\textnormal{\textbf{Cart}}}
\xspace}
\newcommand{\Cat}{\ensuremath{\textnormal{\textbf{Cat}}}
\xspace}
\newcommand{\mf}{\ensuremath{\textnormal{\textbf{mf}}}\xspace}
\newcommand{\cart}{\ensuremath{\textnormal{\textbf{cart}}}
\xspace}

\newcommand{\Diff}{\ensuremath{\textnormal{\textbf{Diff}}}
\xspace}

\newcommand{\s}{\ensuremath{\mathop{;}}}

\newcommand{\Aut}{\ensuremath{\textnormal{Aut}}}
\newcommand{\ob}{\ensuremath{\mathop{\textnormal{ob}}}}
\newcommand{\alt}{\ensuremath{\textnormal{alt}}}
\newcommand{\syn}{\ensuremath{\textnormal{syn}}}
\newcommand{\slambda}{\ensuremath{\underline{\lambda}}}

\newcommand{\di}{\ensuremath{\partial}}

\newcommand*{\emptybox}{\leavevmode\hbox{}}

\makeatletter

\newdimen\w@dth

\def\setw@dth#1#2{\setbox\z@\hbox{\scriptsize $#1$}\w@dth=\wd\z@
\setbox\@ne\hbox{\scriptsize $#2$}\ifnum\w@dth<\wd\@ne \w@dth=\wd\@ne \fi
\advance\w@dth by 1.2em}

\def\t@^#1_#2{\allowbreak\def\n@one{#1}\def\n@two{#2}\mathrel
{\setw@dth{#1}{#2}
\mathop{\hbox to \w@dth{\rightarrowfill}}\limits
\ifx\n@one\empty\else ^{\box\z@}\fi
\ifx\n@two\empty\else _{\box\@ne}\fi}}
\def\t@@^#1{\@ifnextchar_ {\t@^{#1}}{\t@^{#1}_{}}}

\def\t@left^#1_#2{\def\n@one{#1}\def\n@two{#2}\mathrel{\setw@dth{#1}{#2}
\mathop{\hbox to \w@dth{\leftarrowfill}}\limits
\ifx\n@one\empty\else ^{\box\z@}\fi
\ifx\n@two\empty\else _{\box\@ne}\fi}}
\def\t@@left^#1{\@ifnextchar_ {\t@left^{#1}}{\t@left^{#1}_{}}}

\def\two@^#1_#2{\def\n@one{#1}\def\n@two{#2}\mathrel{\setw@dth{#1}{#2}
\mathop{\vcenter{\hbox to \w@dth{\rightarrowfill}\kern-1.7ex
                 \hbox to \w@dth{\rightarrowfill}}%
       }\limits
\ifx\n@one\empty\else ^{\box\z@}\fi
\ifx\n@two\empty\else _{\box\@ne}\fi}}
\def\tw@@^#1{\@ifnextchar_ {\two@^{#1}}{\two@^{#1}_{}}}

\def\tofr@^#1_#2{\def\n@one{#1}\def\n@two{#2}\mathrel{\setw@dth{#1}{#2}
\mathop{\vcenter{\hbox to \w@dth{\rightarrowfill}\kern-1.7ex
                 \hbox to \w@dth{\leftarrowfill}}%
       }\limits
\ifx\n@one\empty\else ^{\box\z@}\fi
\ifx\n@two\empty\else _{\box\@ne}\fi}}
\def\t@fr@^#1{\@ifnextchar_ {\tofr@^{#1}}{\tofr@^{#1}_{}}}

\newdimen\W@dth
\def\setW@dth#1#2{\setbox\z@\hbox{$#1$}\W@dth=\wd\z@
\setbox\@ne\hbox{$#2$}\ifnum\W@dth<\wd\@ne \W@dth=\wd\@ne \fi
\advance\W@dth by 1.2em}

\def\T@^#1_#2{\allowbreak\def\N@one{#1}\def\N@two{#2}\mathrel
{\setW@dth{#1}{#2}
\mathop{\hbox to \W@dth{\rightarrowfill}}\limits
\ifx\N@one\empty\else ^{\box\z@}\fi
\ifx\N@two\empty\else _{\box\@ne}\fi}}
\def\T@@^#1{\@ifnextchar_ {\T@^{#1}}{\T@^{#1}_{}}}

\def\T@left^#1_#2{\def\N@one{#1}\def\N@two{#2}\mathrel{\setW@dth{#1}{#2}
\mathop{\hbox to \W@dth{\leftarrowfill}}\limits
\ifx\N@one\empty\else ^{\box\z@}\fi
\ifx\N@two\empty\else _{\box\@ne}\fi}}
\def\T@@left^#1{\@ifnextchar_ {\T@left^{#1}}{\T@left^{#1}_{}}}

\def\Tofr@^#1_#2{\def\N@one{#1}\def\N@two{#2}\mathrel{\setW@dth{#1}{#2}
\mathop{\vcenter{\hbox to \W@dth{\rightarrowfill}\kern-1.7ex
                 \hbox to \W@dth{\leftarrowfill}}%
       }\limits
\ifx\N@one\empty\else ^{\box\z@}\fi
\ifx\N@two\empty\else _{\box\@ne}\fi}}
\def\T@fr@^#1{\@ifnextchar_ {\Tofr@^{#1}}{\Tofr@^{#1}_{}}}

\def\Two@^#1_#2{\def\N@one{#1}\def\N@two{#2}\mathrel{\setW@dth{#1}{#2}
\mathop{\vcenter{\hbox to \W@dth{\rightarrowfill}\kern-1.7ex
                 \hbox to \W@dth{\rightarrowfill}}%
       }\limits
\ifx\N@one\empty\else ^{\box\z@}\fi
\ifx\N@two\empty\else _{\box\@ne}\fi}}
\def\Tw@@^#1{\@ifnextchar_ {\Two@^{#1}}{\Two@^{#1}_{}}}

\def\to{\@ifnextchar^ {\t@@}{\t@@^{}}}
\def\from{\@ifnextchar^ {\t@@left}{\t@@left^{}}}
\def\tofro{\@ifnextchar^ {\t@fr@}{\t@fr@^{}}}
\def\To{\@ifnextchar^ {\T@@}{\T@@^{}}}
\def\From{\@ifnextchar^ {\T@@left}{\T@@left^{}}}
\def\Two{\@ifnextchar^ {\Tw@@}{\Tw@@^{}}}
\def\Tofro{\@ifnextchar^ {\T@fr@}{\T@fr@^{}}}

\makeatother

\input{diagxy}

\begin{document}
\maketitle

\begin{abstract}
Tangent categories provide an axiomatic framework for understanding various tangent bundles and differential operations that occur in differential geometry, algebraic geometry, abstract homotopy theory, and computer science.   Previous work has shown that one can formulate and prove a wide variety of definitions and results from differential geometry in an arbitrary tangent category, including generalizations of vector fields and their Lie bracket, vector bundles, and connections.  

In this paper we investigate differential and \textit{sector} forms in tangent categories.  We show that sector forms in any tangent category have a rich structure: they form a symmetric cosimplicial object.  This appears to be a new result in differential geometry, even for smooth manifolds.  In the category of smooth manifolds, the resulting complex of sector forms has a subcomplex isomorphic to the de Rham complex of differential forms, which may be identified with \textit{alternating} sector forms.  Further, the symmetric cosimplicial structure on sector forms arises naturally through a new equational presentation of symmetric cosimplicial objects, which we develop herein.

\end{abstract}

\tableofcontents

\section{Introduction}

Tangent categories \cite{rosicky, sman3} provide an axiomatization of one of the key structures in differential geometry: the tangent bundle.  Tangent categories are useful for a number of reasons.  First, constructions of objects like the tangent bundle appear in a variety of categories, some related to the category of smooth manifolds, others to categories in algebraic geometry, and others to categories in homotopy theory and computer science.  Thus, it is helpful to have a single axiomatization which can deal with all these examples simultaneously.  Secondly, a variety of definitions and constructions in differential geometry are closely linked to the tangent bundle.  For example, vector fields, the Lie bracket, connections, and differential forms can all be viewed as certain maps in the category of smooth manifolds which take as domain or codomain the tangent bundle (or bundles related to it).  Thus, one can hope to give definitions and prove results about these objects in an arbitrary tangent category.   

This paper is a contribution to the second aspect of this program; in particular, in this paper we are interested in determining how to define differential forms, their exterior derivative, and the resulting cochain complex of de Rham in an arbitrary tangent category.  However, to do so requires a close inspection of the nature of differential forms.  This inspection reveals an interesting structure, a simplicial object of \emph{sector forms}, of which de Rham cohomology can be seen as a simple consequence.  

There is a relatively straightforward analog of the notion of differential form in any tangent category.  Classical differential $n$-forms on a smooth manifold $M$ can be viewed as multilinear, alternating maps
	\[ T_nM \rightarrow \mathbb{R} \]
where $T_nM$ is the object of consisting of all ``$n$-tuples of tangent vectors at a common point on $M$''.  (That is, $T_nM$ is the fibre product of $n$ copies of the tangent bundle $TM \rightarrow M$ over $M$.)  These objects exist in any tangent category, and thus one can define (classical) differential forms in any tangent category as above, with $\mathbb{R}$ replaced by a suitable coefficient object.

However, a difficulty arises when attempting to define a direct analog of the exterior derivative of such forms in an arbitrary tangent category.  In the category of smooth manifolds, the exterior derivative of an $n$-form $\omega: T_nM \rightarrow \mathbb{R}$ is an $(n+1)$-form $\di \omega$, which can be defined locally on open subsets $U \cong \R^n$.  In the case where $M = \R^n$ one can define $\partial \omega$ as an alternating sum of certain maps $T_{n+1}M \rightarrow \R$ \cite[7.8]{natural}, each expressed in terms of the Jacobian derivative $T(\omega):T(T_nM) \rightarrow T\R \cong \R \times \R$ by pre-composing with a certain canonical map $\kappa:T_{n+1}M \rightarrow T(T_nM)$.  A similar definition applies in any Cartesian differential category \cite{deRhamCartDiffCats}.  However, in an arbitrary tangent category, the objects need not be manifolds, and the local definition cannot be mimicked globally for want of a suitable map $\kappa$ to mediate between the intended domain of $\partial \omega$ (namely $T_{n+1}M$) and the domain of $T(\omega)$ (namely $T(T_nM)$).

One solution to this problem can be found by considering how synthetic differential geometry (SDG) handles differential forms.  In SDG one finds categories which have \textit{representable} tangent structure; these are categories with an object $D$ for which there is a tangent functor $T$ defined by $TM := M^D$.  Various definitions and results have been transplanted from classical differential geometry to models of SDG; see, e.g., \cite{kock,lavendhomme,reyes}.  In a typical model of SDG, as in a tangent category, the objects need not be locally isomorphic to some $\mathbb{R}^n$.  Thus, for a general object in such categories, the exterior derivative also cannot be defined by mimicking the classical definition directly.  

In SDG, the solution to this problem is to look at a different type of map: instead of considering multilinear alternating maps from $T_nM \rightarrow R$, one instead considers multilinear alternating maps 
	\[ T^nM \to R, \] 
where $T^nM$ is the $n$th iterate of the tangent bundle of $M$.  Such maps were first considered in \cite{Kock:DiffFormsSDG,KRV} and referred to as \textit{singular forms} in \cite[Definition 4.1]{lavendhomme}; we shall use that name here to distinguish them from other notions of form we shall consider.  In contrast to the classical case, the Jacobian derivative $T(\omega)$ of a singular $n$-form $\omega$ does have the expected domain, namely $T^{n+1}M$.  The references above show how to define an exterior derivative for such forms; the definition involves an alternating sum of permutations of the Jacobian derivative.    Moreover, it has been shown that for a particular model of SDG which contains the category of smooth manifolds, if $M$ is a smooth manifold then singular forms are in bijective correspondence with classical differential forms, and their exterior derivatives agree \cite[IV, Proposition 3.7]{reyes}.  This shows that models of SDG have a notion of de Rham complex which generalizes the classical notion.  

Thus, for tangent categories, a natural point of investigation is to look at multilinear alternating maps from $T^nM$ to some coefficient object $E$.  We show in this paper that such maps indeed have an exterior derivative that generalizes the definition from SDG, and has the required properties (Proposition \ref{prop:complexOfSectorForms}).  Thus, this shows that tangent categories have a notion of de Rham cohomology (namely, the cohomology of the resulting complex) which generalizes the SDG notion, and hence also generalizes the classical notion.  

However, there is much more to say about maps from $T^nM$ to $E$ in a tangent category.  In particular, none of the results that need to be proved to show that such maps have an exterior derivative require that the maps from $T^nM$ be alternating; multilinearity suffices.  Thus, one is led to consider maps
	\[ \omega: T^nM \rightarrow E \]
(for a suitable coefficient object $E$) which are multilinear but not necessarily alternating.  Such maps do not appear in published accounts of SDG, but have appeared in differential geometry \cite{white, sectorFormsPhysicsArticle}.  They are known as \emph{sector forms} (for some basic examples of sector forms, see Section \ref{sec:results}).   The exterior derivative of singular forms works for sector forms, and so in addition to the complex of singular forms, tangent categories have complexes of sector forms (\ref{def:compl_sector_forms}).   

However, there is much more structure to these sector forms than a cochain complex.  We show that for each $n$, there are $n+1$ `derivative' or \textit{co-face} operations which take sector $n$-forms to sector $(n+1)$-forms (Theorem \ref{thm:sym_cosimpl_cmon_sector_forms}), there are $n-1$ \textit{symmetry} operations which take sector $n$-forms to sector $n$-forms, and there are $n-1$ \textit{co-degeneracy} operations which take sector $n$-forms to sector $(n-1)$-forms (Proposition \ref{thm:cod_symm_sector_forms})\footnote{Note that  while pre-composing with the projection or zero natural transformations gives an obvious way to define co-face and co-degeneracy operations, these are not the operations we are using: see section \ref{sec:results} for a basic overview of the definitions of these operations.}.  Taken together, these operations constitute the structure of an \textit{(augmented) symmetric cosimplicial object} \cite{Barr,Gra:Symm} of sector forms (Theorem \ref{thm:sym_cosimpl_cmon_sector_forms}); that is, there is a functor on the category of finite cardinals.  This is a remarkably rich structure, and has not previously appeared in either ordinary differential geometry or synthetic differential geometry\footnote{The first person to write on sector forms in standard differential geometry, J.E. White, describes the co-face and symmetry maps of sector forms, but does not describe the the co-degeneracy maps or any of the equations these various maps must satisfy in order to yield a symmetric cosimplicial object \cite[Definition 3.8, 3.12]{white}}.

Thus, we view the symmetric cosimplicial object of sector forms as the primary object of interest in relation to the various notions of differential forms considered above.  In particular, from this cosimplicial object one can obtain as a simple corollary the complex of sector forms and the complex of singular forms\footnote{In the context of synthetic differential geometry, there is a different way of obtaining the de Rham complex from simplicial data, using the idea of \textit{infinitesimal cochains} \cite[Section I.18]{kock} \cite{diffFormsasInfCochains}.  However, this approach is different from the approach via sector forms that we pursue here.}.  Moreover, by generalizing to maps which are not necessarily alternating, one also generalizes covariant tensors (multilinear maps with domain $T_n$) which have numerous uses throughout differential geometry \cite[Section 3.1]{white}.   In other words, sector forms generalize three important ideas in differential geometry: differential forms, covariant tensors, and singular forms.  Thus, it is important to understand the structure of sector forms, and this paper represents a substantial advance in the study of these objects in the general setting of tangent categories.  

\begin{center}
\begin{tabular}{ |l||c|c| }
\hline
& Alternating & Not alternating \\ 
\hhline{|=#=|=|}
Domain $T_n$ & differential form & covariant tensor \\
\hline
Domain $T^n$ & singular form & sector form \\
\hline
\end{tabular}
\end{center}

\vspace{0.25in}

It is also worth noting that this paper contains two other points of independent interest.  First, to establish the symmetric cosimplicial structure of sector forms, it becomes natural to give an alternative presentation of symmetric cosimplicial objects, and in particular to give an alternative presentation of the category of finite cardinals.  The standard presentation \cite{Gra:Symm} involves co-face maps, symmetry maps, and co-degeneracy maps.  However, for each $n$, the $n+1$ co-face maps from $n$ to $n+1$ can all be obtained by applying symmetries to a single co-face map, and thus one can show (Theorem \ref{thm:fund_cof_pres_fincard}) that the category of finite cardinals can be presented by symmetries, co-degeneracies, and a single co-face map for each $n$.  

The second point of interest relates to methodology in tangent categories in general.  The definition of the symmetry and co-degeneracy maps of sector forms involves various combinations of the \textit{lift} natural transformation $\ell: T \to TT$ and the \textit{canonical flip} transformation $c: TT \to TT$ (which are part of the definition of a tangent category).  To establish the various identities that are required of the symmetries and co-degeneracies, one then must perform various complicated calculations with these maps.  One way to handle the complexity of such calculations is to use string diagrams, as was done in previous work on tangent categories \cite{jacobiProof}.  Another way to handle the complexity is to use a recently discovered embedding theorem for tangent categories \cite{garnerEmbedding} (for more on this approach, see the discussion after \ref{tangent-category-examples}).    However, here we use a different approach.  Diagrams involving the maps $\ell$ and $c$ in a tangent category $\X$ can be viewed as the application of a certain functor from the category of finite cardinals and surjections (written as $\finCard_s$) to the category of endofunctors on $\X$ (Example \ref{exa:symm_deg_it_tang}).  Thus, to establish the commutativity of such a diagram of natural transformations, it suffices to establish the commutativity of a certain diagram in $\finCard_s$, and this is typically straightforward.  For examples of this proof technique, see Proposition \ref{thm:cod_symm_sector_forms} and Theorem \ref{thm:fund_deriv_sector_form}.  

The paper is laid out as follows.  In section \ref{sec:tanCats}, we review the definitions of tangent categories and \textit{differential objects}, which are the coefficient objects in which the forms will take their values.  Before going into the various details required of many of the proofs, in section \ref{sec:results} we give an overview of the key definitions and results of the paper, providing more detail than in the discussion above, and also providing some examples of sector forms.  In sections \ref{sec:monSemigroups}, \ref{sec:symCosimp}, and \ref{sec:presFundamentalCoface}, we study symmetric cosimplicial objects and related notions, emphasizing their relations to categories of finite cardinals and establishing equational presentations of some of these key categories.  Throughout these sections, we show how some of the structure of these categories is present in the category of endofunctors on a tangent category.  In section \ref{sec:cosimpOfSectorForms}, we look at sector forms, their \textit{fundamental derivative}, and how they have the structure of a symmetric cosimplicial object.  In section \ref{sec:complexes}, we obtain the complexes of sector forms and singular forms as simple consequences of the symmetric cosimplicial structure on sector forms.  In section \ref{sec:relFormsClSDG}, we study forms in the presence of representable tangent structure, and we show how our definitions of sector forms and singular forms in a tangent category relate to existing definitions in classical and synthetic differential geometry.  Finally, in section \ref{sec:futureWork}, we look at various ways to extend or add to the results we have presented.

\section{Tangent categories and differential objects}\label{sec:tanCats}

\subsection{Notation}\label{sec:notn}

Throughout this paper, composition in diagrammatic order is indicated with a semicolon, so that $f$, followed by $g$, is written as $f \s g$.  When $F$ and $G$ are functors, we will sometimes denote the composite $F \s G$ instead by $GF$, so that juxtaposition of functors denotes classical right-to-left composition.  Given an object $C$ of a category $\C$, we denote by $\Aut_\C(C)$ the group of automorphisms of $C$ in $\C$.  Rather than straying from convention by defining multiplication in $\Aut_\C(C)$ in terms of the diagrammatic composition order, we instead take the view that groups are certain one-object categories, and we define composition in $\Aut_\C(C)$ as in $\C$.

\subsection{Additive bundles}

If $M$ is an object in a category $\X$, an \textbf{additive bundle over $M$} consists of a map $q: E \to M$ such that (i) $q$ admits finite \textit{fibre powers}, i.e., for each $n \in \N$ there is a fibre product $E_n \rightarrow M$ of $n$ copies of $(E,q)$ over $M$ with projections $\pi_1,...,\pi_n:E_n \rightarrow E$, and (ii) $(E,q)$ is a commutative monoid in the slice category $\X/M$.  In particular, this means there is an addition operation, which will often be written as $\sigma: E_2 \to E$ and must satisfy the usual 
requirements of commutativity and associativity, and a unit for this addition, which will often be written as $\zeta:M \to E$.  A map between such bundles will, in general, just be a commutative square
$$\xymatrix{E \ar[rr]^e \ar[d]_q && E' \ar[d]^{q'} \\ B \ar[rr]_b && B'}$$
written $(e,b): q \to q'$.   If, in addition, such a map of bundles preserves the addition -- that is $e_2; \sigma' = \sigma; e$ and $b; \zeta' = \zeta; e$ -- then we shall say that $(e,b)$ is an {\bf additive bundle morphism}.

\begin{definition}\label{defnTangentCategory}
For a category $\X$, tangent structure $\T = (T,p,0,+,\ell,c)$ on $\X$ consists of the following data:
\begin{itemize} 
	\item (\textbf{tangent functor})  a functor $T: \X \to \X$ with a natural transformation $p: T \to I$ such that each 
	$p_M: T(M) \to M$ admits finite fibre powers that are preserved by each $T^n:\X \rightarrow \X$;
	\item (\textbf{additive bundle}) natural transformations $+:T_2 \to T$ (where $T_2$ is the second fibre power of $p$) and $0: I \to T$ making each $p_M: TM \to M$ an additive bundle;
	\item (\textbf{vertical lift}) a natural transformation $\ell:T \to T^2$ such that for each $M$
	$$(\ell_M,0_M): (p: TM \to M,+,0) \to (Tp: T^2M \to TM,T(+),T(0))$$ 
	is an additive bundle morphism;
        \item (\textbf{canonical flip}) a natural transformation $c:T^2 \to T^2$ such 
               that for each $M$
               $$(c_M,1): (Tp: T^2M \to TM, T(+),T(0)) \to (p_T: T^2M \to TM,+_T,0_T)$$ 
               is an additive bundle morphism;
        \item (\textbf{coherence of $\ell$ and $c$}) $c\s c = 1$ (so $c$ is an isomorphism), $\ell\s  c = \ell$, and the following diagrams commute:
\[
\xymatrix{T \ar[r]^{\ell} \ar[d]_{\ell} &T^2 \ar[d]^{T\ell} \\
                 T^2 \ar[r]_{\ell T} & T^3}
~~~
\xymatrix{T^3  \ar[r]^{T c} \ar[d]_{c T} & T^3 \ar[r]^{c T} & T^3 \ar[d]^{Tc} \\
                  T^3 \ar[r]_{Tc} & T^3 \ar[r]_{cT} & T^3}
~~~
\xymatrix{T^2 \ar[d]_{c} \ar[r]^{\ell T} & T^3 \ar[r]^{Tc} & T^3 \ar[d]^{cT} \\
                  T^2 \ar[rr]_{T\ell} & & T^3}
\]
	\item (\textbf{universality of vertical lift}) defining $v: T_2M \to T^2M$ by $v := \<\p_1\s \ell, \p_2\s 0_T\>T(+)$,
	the following diagram is a pullback that is preserved by each $T^n$: 
$$\xymatrix{T_2(M) \ar[d]_{\pi_1 \s p\:=\:\pi_2 \s p} \ar[rr]^{v} &  & T^2(M) \ar[d]^{T(p)} \\  M \ar[rr]_{0} &  & T(M)}$$	
\end{itemize}
A category with tangent structure, $(\X,\T)$, is a {\bf tangent category}.
\end{definition}

\begin{example} \label{tangent-category-examples}
Here are several important examples of tangent categories, the first four of which are drawn from \cite{sman3} and \cite{diffBundles}.
\begin{enumerate}[{\rm (i)}]
\item Finite dimensional smooth manifolds with the usual tangent bundle structure.  
\item Convenient manifolds with the kinematic tangent bundle \cite[Section 28]{convenient}.  
%\item Complex manifolds with holomorphic maps between them.
\item The infinitesimally and vertically linear objects in any model of synthetic differential geometry \cite{kock} form a tangent category:  if $D$ is the object of square-zero infinitesimals, then we take $TM := M^D$. 
\item The opposite of the category of finitely presented commutative rings (or more generally commutative rigs\footnote{That is, unital rings without additive inverses, also known as unital semirings.}) is another example of a category with representable 
tangent structure: here $D$ is the `rig of infinitesimals', $\N[\varepsilon] := \N[x]/(x^2=0)$ and again $TA := A^D$.  
\item A very different example of a tangent category arises from abstract homotopy theory, in particular in work on \textit{Abelian functor calculus} \cite{abFunCalc}.  In \cite{abFuncCalcDirDeriv}, the authors show that a certain operation in abelian functor calculus gives rise to a \textit{Cartesian differential category} \cite{cartDiff}.  As every Cartesian differential category is a tangent category \cite[Section 4.2]{sman3}, this example is also a tangent category; this insight was useful in providing a straightforward proof of the existence of certain higher-order chain rules for abelian functor calculus (see the discussion at the top of page 5 in \cite{abFuncCalcDirDeriv}).  
\item Other examples of tangent categories that arise as Cartesian differential categories include the models of the differential $\lambda$-calculus that appear in computer science (for example, see \cite{manzonetto} and \cite{cockettGallagher}).  
\end{enumerate}
\end{example}
More examples can be found in \cite{sman3} and \cite{diffBundles}.  In addition to these examples, recent work of Leung \cite{leungClassifyingTanStructures} and Garner \cite{garnerEmbedding} establishes certain equivalent formulations of tangent categories that provide new perspectives on the axioms.  Leung's work shows that tangent categories are closely related to categories of Weil algebras, while Garner's work builds on this result to show not only that tangent categories can be seen as certain types of enriched categories, but also that every tangent category can be embedded in a tangent category that is \textit{representable}, in the sense that the functor $T$ is representable (for more on representable tangent categories, see \cite[Section 5]{sman3} and also \ref{def:rep_tan} below).  

This last point allows one to work in an arbitrary tangent category as if $T$ was representable, allowing for calculations in a tangent category which closely resemble those in SDG.  One may ask, for example, whether this could simplify the proofs of some of the results in this paper.  However, we have not found that this was the case.  Certain of our initial attempts at proofs of the main results of this paper indeed used representable tangent categories, but the resulting calculations were no less lengthy than those recorded herein; indeed, by observing that the transformations $\ell$ and $c$ generate a model of a certain \textit{PROP}, in Mac Lane's sense \cite{MacL:CatAlg}, we have reduced many of these calculations to showing that certain diagrams of finite sets commute.  More importantly still, it was only by working with the relatively restrictive stuctures of tangent categories and associated PROPs that we discovered the main results of this paper, including the result that sector forms carry cosimplicial structure.

\subsection{Commutative monoids in a Cartesian tangent category}\label{sec:cmons_in_tngnt_cat}

The coefficient objects of our forms will in particular be commutative monoids, so it is useful to first make some remarks about commutative monoids in a tangent category.  

A tangent category $(\X,\T)$ is said to be \textbf{Cartesian} if $\X$ has finite products that are preserved by the tangent functor $T:\X \rightarrow \X$.  In this case we denote by $\cmon(\X)$ the category of commutative monoid objects in $\X$.  For each object $X$ of $\X$, the functor $\X(X,-):\X \rightarrow \set$ preserves limits and so sends each commutative monoid $E$ in $\X$ to a commutative monoid $\X(X,E)$ in $\set$.  When $X$ itself is a commutative monoid in $\X$, the hom-set $\cmon(\X)(X,E)$ is a submonoid of $\X(X,E)$.  Composition in $\cmon(\X)$ preserves this monoid structure in each variable separately, so we say that $\cmon(\X)$ is an \textit{additive category}, which, following previous papers on tangent categories, we take to mean a category enriched in commutative monoids.  Moreover, we have the following:

\begin{proposition}\label{prop:cmons_in_tngt_cat}
Let $(\X,\T)$ be a Cartesian tangent category.  Then the tangent endofunctor $T:\X \rightarrow \X$ lifts to an endofunctor $\cmon(T):\cmon(\X) \rightarrow \cmon(\X)$.  Moreover, the endofunctor $\cmon(T)$ is additive in the sense that it preserves the commutative monoid structure on the hom-sets of $\cmon(\X)$.  Furthermore, if $\phi:T^i \Rightarrow T^j$ is a natural transformation between iterates of $T$ ($i,j \in \N$), then for each commutative monoid $E$ in $\X$ the morphism $\phi_E:T^iE \rightarrow T^j E$ is a homomorphism of commutative monoids in $\X$.
\end{proposition}
\begin{proof}
There is a 2-category $\Cart$ whose objects are categories with finite products, wherein the 1-cells are functors preserving finite products and the 2-cells are arbitrary natural transformations.  Letting $\C$ denote the Lawvere theory of commutative monoids, there is an equivalence of categories $\Cart(\C,\D) \simeq \cmon(\D)$ for every object $\D$ of $\Cart$.  But $\Cart(\C,-):\Cart \rightarrow \Cat$ is a 2-functor valued in the 2-category of categories, and it follows that the assignment $\D \mapsto \cmon(\D)$ underlies a 2-functor $\cmon:\Cart \rightarrow \Cat$.  We can apply this 2-functor to the 1-cell $T$ and to the 2-cell $\phi$, thus proving two of the above claims.  Lastly, $\cmon(\X)$ is an additive category with finite products, which are therefore finite biproducts, and since $T$ preserves finite products it follows that $\cmon(T)$ preserves finite biproducts and hence is additive.
\end{proof}

\subsection{Differential objects}\label{sec:diff_objs}
Before we define sector forms and singular forms, we need to consider the objects in which these forms will take their values.  These will be \textit{differential objects}, which are certain objects $E$ whose tangent bundle $TE$ is simply a product $E \times E$.  This is formulated more precisely as follows.  

\begin{definition}\label{def:diff_objs}
Let $(\X,\T)$ be a Cartesian tangent category.  A \textbf{differential object} in $(\X,\T)$ consists of a commutative monoid $(E,\mu,\eta)$ in $\X$ (equivalently, an additive bundle over $1$) together with a map $\lambda: E \to TE$ (known as the \textit{lift}) such that
\begin{itemize}
	\item $(\lambda,\eta)$ is an additive bundle morphism from $(E,!_E,\mu,\eta)$ to $(TE,p,+,0)$;
	\item $\lambda$ is a homomorphism of commutative monoids from $(E,\mu,\eta)$ to $(TE,T(\mu),T(\eta))$;
	\item the equation $\lambda\s T(\lambda) = \lambda\s \ell_E$ holds;
	\item the map 
		\[ E \times E \to^{\nu \::=\: \<\pi_1\s \lambda,\:\pi_2\s 0_E\>\s T(\mu)} TE \]
	is an isomorphism.
\end{itemize}
Say that $(E,\mu,\eta,\lambda)$ is a \textbf{subtractive} differential object if the commutative monoid $(E,\mu,\eta)$ is an abelian group.  In this case we denote the inverse operation by $-: E \to E$.  
\end{definition}

\begin{example} \label{diff-object-examples} Here are some important examples of differential objects from \cite{diffBundles}.
\begin{enumerate}[(i)]
	\item In the category of smooth manifolds, each Cartesian space $\mathbb{R}^n$ is a differential object, where $\lambda:\R^n \rightarrow T\R^n = \R^n \times \R^n$ sends $x$ to $(x,0)$, construed as the tangent vector $x$ at the point $0$.
	\item Similarly, in the category of convenient manifolds, each convenient vector space is a differential object.
	\item In the category of affine schemes $\mathbf{cRing}^{\mbox{\scriptsize{op}}}$, polynomial rings $\mathbb{Z}[x_1, x_2, \ldots, x_n]$ are differential objects.  
	\item Differential objects in a tangent category associated to a model of SDG are precisely the Euclidean $R$-modules \cite[1.1.4]{lavendhomme} (see \cite[Theorem 3.9]{diffBundles} for a proof of this).  
\end{enumerate}
All of the above examples are subtractive.  
\end{example}

\subsection{Remark. }\label{sec:princ_projn}

By definition, if $E$ is a differential object, then $TE \cong E \times E$.  Through the isomorphism $\nu$, one can show that the projection from the second component is $p_E: TE \to E$.  We will write $\hat{p}:TE \rightarrow E$ for the projection to the first component, and refer to it as the \textbf{principal projection}.  Differential objects can be alternatively axiomatized in terms of the principal projection $\hat{p}$.  For example, this was how differential objects were originally presented \cite[Definition 4.8]{sman3}.  It is a relatively straightforward exercise to show the equivalence of the two definitions \cite[Proposition 3.4]{diffBundles}.  

We will make use of both the lift $\lambda: E \to TE$ and the principal projection $\hat{p}: TE \to E$ when investigating differential forms with values in $E$.  In particular, the following results about these maps will be useful.

\begin{proposition}\label{prop:diffObjs}
If $E$ is a differential object with lift $\lambda$ and principal projection $\hat{p}$, then
\begin{enumerate}[(i)]
	\item $\hat{p}$ is a homomorphism of commutative monoids, where $TE$ has the commutative monoid structure as discussed in \ref{prop:cmons_in_tngt_cat}.
	\item $\lambda \s \hat{p} = 1_E$.
	\item $0_E \s \hat{p} = \;!_E \s \zeta$.
	\item $\ell_E\s  T(\hat{p})\s \hat{p} = \hat{p}$.
	\item $c_E \s T(\hat{p})\s \hat{p} = T(\hat{p})\s \hat{p}$.
	\item $\ell_E\s  T(\hat{p}) = \hat{p}\s  \lambda$.
	\item $T(\lambda)\s c_E \s T(\hat{p}) = \hat{p}\s  \lambda$.
\end{enumerate}
\end{proposition}
\begin{proof}
The first five parts are established in \cite[Propositions 3.4 and 3.6, Definition 3.1]{diffBundles}.  For (vi), we regard $TE$ as a product $E \times E$ with projections $(\hat{p},p_E)$.  Then $\lambda$ is the morphism $\langle 1_E, 0 \rangle$ induced by the identity morphism on $E$ and the zero element $0$ of the commutative monoid $\X(E,E)$ (\ref{sec:cmons_in_tngnt_cat}).  Hence
$$\hat{p} \s \lambda \s \hat{p} = \hat{p} \s 1_E = \hat{p} = \ell_E \s T\hat{p} \s \hat{p}$$
by (iv).  Also, using the naturality of $p$, the equation $\ell \s c = \ell$, and the fact that $(\ell_E,0_E)$ and $(c_E,1_{TE})$ are bundle morphisms (\ref{defnTangentCategory}) we compute that
$$
\begin{array}{ccccccc}
\hat{p} \s \lambda \s p_E & = & \hat{p} \s 0 & = & 0 & = & p_E \s 0\\
                          & = & p_E \s 0_E \s \hat{p} & = & \ell_E \s Tp_E \s \hat{p} & = & \ell_E \s c_E \s p_{TE} \s \hat{p}\\
                          & = & \ell_E \s p_{TE} \s \hat{p} & = & \ell_E \s T\hat{p} \s p_E
\end{array}                        
$$
where each unadorned $0$ denotes the zero element of the relevant hom-set (\ref{sec:cmons_in_tngnt_cat}) while $0_E = \langle 0,1_E\rangle:E \rightarrow TE \cong E \times E$ denotes the zero section, i.e. the component of $0:I \Rightarrow T$ at $E$.  Hence $\hat{p} \s \lambda = \ell_E \s T\hat{p}$ as needed.

For (vii), we again use the fact that $TE$ is a product $E \times E$ with projections $(\hat{p},p_E)$.  Appending the first projection $\hat{p}$ to the equation in question, we compute that
$$T\lambda \s c_E \s T\hat{p} \s \hat{p} = T\lambda \s T\hat{p} \s \hat{p} = T(\lambda \s \hat{p}) \s \hat{p} = \hat{p} = \hat{p} \s \lambda \s \hat{p}$$
using (ii) and (v).  Appending the second projection $p_E$, we compute that
$$
\begin{array}{lllllllll}
T\lambda \s c_E \s T\hat{p} \s p_E & = & T\lambda \s c_E \s p_{TE} \s \hat{p} & = & T\lambda \s Tp_E \s \hat{p} & = & T(\lambda \s p_E) \s \hat{p} & = & T(0) \s \hat{p}\\
& = & 0 \s \hat{p} & = & 0 & = & \hat{p} \s 0 & = & \hat{p} \s \lambda \s p_E
\end{array}
$$
using the naturality of $p$, the additivity of $T$ (\ref{prop:cmons_in_tngt_cat}), the fact that $\lambda \s p_E = 0$, the fact that $\hat{p}$ is a homomorphism of commutative monoids, and the fact that $c_E \s p_{TE} = Tp_E$ since $(c,1_{TE})$ is a bundle morphism (\ref{defnTangentCategory}).
\end{proof}

%%%%%%%%%%%%%%%%%%%%%%%%%%%%%%%%%%%%%%%%%%%%%%%%%%%%%%%%%%%%%%%%%%%%%%%%%%%%%%%%%

\section{Overview of main results, with examples of sector forms}\label{sec:results}

To prove the main results of this paper we have used a variety of techniques and definitions.  However, we feel it is important to present many of the main results in a single place so as to be easily locatable and so as to feature them prominently.  In this section we also look at some examples of sector forms, as they are perhaps much less familiar to the general reader than differential forms.

Throughout this section, we work in a Cartesian tangent category $(\X,\T)$ with a fixed object $M$ and a fixed differential object $(E,\sigma,\zeta,\lambda)$.

We first define a number of natural transformations between powers of $T$ which will appear in the definitions of certain types of forms and their derivatives.  Much of the work of sections \ref{sec:monSemigroups} and \ref{sec:symCosimp} deals with how to interpret these natural transformations and handle them more efficiently.  
\begin{definition}
Define
\begin{equation}\ell^n_i \;:=\; T^{i - 1}\ell T^{n - i}\;\;:\;\;T^n \rightarrow T^{n+1}\;\;\;\;\;\;(n,i \in \N, 1 \leqslant i \leqslant n)\end{equation}
\begin{equation}c^n_i \;:=\; T^{i - 1} c T^{n - i - 1}\;\;:\;\;T^n \rightarrow T^n\;\;\;\;\;\;(n,i \in \N, 1 \leqslant i \leqslant n - 1)\end{equation}
\begin{equation}c^n_{(i)} \;:=\; c_{i-1}\s c_{i-2} \s...\s c_2 \s c_1\;\;:\;\;T^n \rightarrow T^n\;\;\;\;\;\;(n,i \in \N, 1 \leqslant i \leqslant n)\end{equation}
\begin{equation}a^n_{i} \;:=\; \ell_i^n \s c_{(i)}^{n+1}\;\;:\;\;T^{n} \rightarrow T^{n+1}\;\;\;\;\;\;(n,i \in \N, 1 \leqslant i \leqslant n)\end{equation}
\end{definition}

The main object of study of this paper is the following notion of form, originally due to White \cite{white}.  

\begin{definition}\label{def:sector_nform}
A \textbf{sector $n$-form} on $M$ with values in $E$ is a morphism $\omega: T^n M \rightarrow E$ such that for each $i \in \{1,...,n\}$, $\omega$ is \emph{linear in the $i$th variable}; that is, the following diagram commutes\footnote{White instead states the definition in terms of a criterion of fibrewise linearity.}:
\[
\bfig
	\square<500,300>[T^nM`E`T^{n+1}M`TE;\omega`a^n_i M`\lambda`T(\omega)]
\efig
\]
The set of sector $n$-forms on $M$ with values in $E$ will be denoted by $\Psi_n(M;E)$ (often abbreviated to $\Psi_n(M)$).  
\end{definition}

To help explore the similarities and differences between ordinary differential forms and sector forms, we will briefly look at sector 1- and 2-forms on $\mathbb{R}$ in the category of smooth manifolds.  

\begin{example}\label{ex1Form}
Let us first consider what sector 1-forms on $\mathbb{R}$ with values in $\mathbb{R}$ consist of.  By definition, such a form consists of a map $\omega: T\mathbb{R} \to \mathbb{R}$ that satisfies the single linearity equation
\[
\bfig
	\square<500,300>[T\mathbb{R}`\mathbb{R}`T^2\mathbb{R}`T\mathbb{R};\omega`\ell_{\mathbb{R}}`\lambda`T\omega]
\efig
\]

Recall that $T\mathbb{R}$ is simply $\mathbb{R} \times \mathbb{R}$, via the two projections $p_{\mathbb{R}}: T\mathbb{R} \to \mathbb{R}$ and $\hat{p}: T\mathbb{R} \to \mathbb{R}$.  Hence the commutativity of the above diagram is equivalent to its commutation when post-composed with $p_{\R}$ and $\hat{p}$, respectively.  But $\omega \s \lambda \s p_\R = \omega \s \;!_\R \s \zeta = \;!_{T\R} \s \zeta$ by \ref{def:diff_objs}, and $\ell_\R \s T\omega \s p_\R = \ell_\R \s p_{T\R} \s \omega = \ell_\R \s c_\R \s p_{T\R} \s \omega = \ell_\R \s Tp_\R \s \omega = p_\R \s 0_\R \s \omega$ by the axioms for tangent categories (\ref{defnTangentCategory}).  Hence $\omega \s \lambda \s p_\R = \ell_\R \s T\omega \s p_\R \Leftrightarrow \;!_{T\R} \s \zeta = p_\R \s 0_\R \s \omega \Leftrightarrow p_\R \s \;!_\R \s \zeta = p_\R \s 0_\R \s \omega \Leftrightarrow \;!_\R \s \zeta = 0_\R \s \omega$ since $p_\R$ is a retraction (of $0_\R$).   On the other hand, $\ell_\R \s T\omega \s \hat{p} = \ell_\R \s D\omega$ where $D \omega = T\omega \s \hat{p}:T^2\R \rightarrow \R$ is the directional derivative of $\omega$, so since $\lambda \s \hat{p} = 1$ by \ref{prop:diffObjs} we find that the above linearity equation holds if and only if the following equations hold:
$$\ell_\R \s D\omega = \omega\;\;\;\;\;\;\;\;!_\R \s \zeta = 0_\R \s \omega\;.$$
But the first equation entails the second, since if the first holds then $0_\R \s \omega = 0_\R \s \ell_\R \s D\omega = 0_\R \s \ell_\R \s T\omega \s \hat{p} = 0_\R \s 0_{T\R} \s T\omega \s \hat{p} = 0_\R \s \omega \s 0_\R \s \hat{p} = 
0_\R \s \omega \s \;!_\R \s \zeta = \;!_\R \s \zeta$
by \ref{defnTangentCategory} and \ref{prop:diffObjs}.

Hence the linearity equation is equivalent to the equation $\ell_\R \s D\omega = \omega$.  In order to reformulate this equation more concretely, let us write $\omega:\R \times \R \rightarrow \R$ as a function $\omega(x,v)$ of two variables $x,v$, so that we may write its first and second partial derivatives briefly as $\frac{\di\omega}{\di x}$ and $\frac{\di\omega}{\di v}$.  We can write $T^2\R = T(\R \times \R)$ as a product $T^2\R = (\R \times \R) \times (\R \times \R)$, whereupon $D \omega: T^2\mathbb{R}  \to \mathbb{R}$ is given by
	\[ \<x_0,v_1,v_2,d\> \mapsto \frac{\di \omega}{\di x}(x_0,v_1)\cdot v_2 + \frac{\di \omega}{\di v}(x_0,v_1) \cdot d \;,\]
and $\ell: T\mathbb{R} \to T^2\mathbb{R}$ is given by
	\[ \<x,v\> \mapsto \<x,0,0,v\>\;.\]
Hence $\ell \s D\omega$, evaluated at $\<x,v\>$, is
	\[ \frac{\di \omega}{\di x}(x,0)\cdot 0 + \frac{\di \omega}{\di v}(x,0)\cdot v  = \frac{\di \omega}{\di v}(x,0)\cdot v. \]
So the linearity equation says
	\[ \omega(x,v) = \frac{\di\omega}{\di v}(x,0)\cdot v. \]
Thus, if we set $f(x) = \frac{\di \omega}{\di v}(x,0)$, then 
	\[ \omega(x,v) = f(x) \cdot v. \]
It is easy to see that any map of this form is indeed a sector 1-form.  So, in this case, sector 1-forms are precisely the same as ordinary differential 1-forms (more generally, this is true for any smooth manifold).  
\end{example}

\begin{example}\label{ex2Form}
Despite \ref{ex1Form}, sector $n$-forms for $n \geq 2$ are in general quite different from ordinary differential $n$-forms, even for a simple smooth manifold such as $\mathbb{R}$.  For example, for $n=2$, a sector 2-form on $\mathbb{R}$ consists of a map $\omega: T^2\mathbb{R} \to \mathbb{R}$ such that
\[
\bfig
	\square<500,300>[T^2\mathbb{R}`\mathbb{R}`T^3\mathbb{R}`T\mathbb{R};\omega`\ell_{T\mathbb{R}}`\lambda`T\omega]
	\place(1000,150)[\mbox{and}]
	\square(1750,0)<500,300>[T^2\mathbb{R}`\mathbb{R}`T^3\mathbb{R}`T\mathbb{R};\omega`T\ell_{\mathbb{R}}\s c_{T\mathbb{R}}`\lambda`T\omega]
\efig
\]
Using similar reasoning to the previous example, it is straightforward to show that any map of the form
	\[ \omega(\<x,v_1,v_2,d\>) = f(x)\cdot v_1 \cdot v_2 + g(x) \cdot d \]
(where $f(x)$ and $g(x)$ are smooth functions from $\mathbb{R}$ to itself) is an example of a sector 2-form\footnote{One can also show that \emph{all} sector 2-forms on $\mathbb{R}$ are of this form.  A general description of sector $n$-forms on $\R^m$ was given by White \cite{white}.} on $\mathbb{R}$.  This is very different from the general description of ordinary differential 2-forms on $\R$: there is only one, namely the zero form.  
	
\end{example}

One of the key points of this paper is that the sector $n$-forms have a rich variety of operations that can be performed on them.  In particular, there are $n+1$ different \emph{derivative} or \emph{co-face} operations $\delta_i^n$ which take sector $n$-forms to sector $(n+1)$-forms:
	\[ \delta_i^n(\omega) := c^{n+1}_{(i)}M \s T\omega\s \hat{p}\;\;\; \mbox{ (Theorem \ref{thm:sym_cosimpl_cmon_sector_forms})}\]
there are $n-1$ different \emph{co-degeneracy} operations $\varepsilon_{i}^{n-1}$ which take sector $n$-forms to sector $(n-1)$-forms:
	\[ \varepsilon_i^{n-1}(\omega) := \ell_i^{n-1}M\s \omega\;\;\;  \mbox{ (Proposition \ref{thm:cod_symm_sector_forms})} \]
and there are $n-1$ different \emph{symmetry} operations $\sigma_i^n$ which take sector $n$-forms to sector $n$-forms:
	\[ \sigma_i^n(\omega) := c_i^nM\s \omega\;\;\; \mbox{ (Proposition \ref{thm:cod_symm_sector_forms}).} \]
%\mbox{ (see \ref{exa:cod_sets_ind_tngt_str})}	
	
Much of the work of this paper goes into proving the following result (Theorem \ref{thm:sym_cosimpl_cmon_sector_forms}):
$$
\begin{minipage}{4.5in}
\textit{The operations $\delta$, $\varepsilon$, $\sigma$ together endow the set of sector forms on $M$ with the structure of an (augmented) symmetric cosimplicial commutative monoid, i.e. a functor from the category of finite cardinals to the category of commutative monoids.}
\end{minipage}
$$
In particular, this means that for every function between finite cardinals $f: n \rightarrow m$ there is an associated monoid homomorphism $\Psi_f: \Psi_n(M) \to \Psi_m(M)$ (given as some composite of the above co-face, co-degeneracy, and symmetry operations), and this entire assignment is functorial.  

Moreover, if $E$ is subtractive, then this structure forms a symmetric cosimplicial abelian group.  In general, any cosimplicial abelian group has an associated cochain complex whose differential $\partial$ is given by taking an alternating sum of the co-face maps:
	\[ \partial(\omega) := \sum_{i=1}^{n+1} (-1)^{i-1} \delta_i^n(\omega) \]
And so there is a complex of sector forms (\ref{def:compl_sector_forms}), whose differential we call the \textit{exterior derivative}.  The fact that the sector forms constitute a cochain complex appears to be a new result in differential geometry.  

\begin{example}
Let us consider the first several groups of this complex for $M = E = \mathbb{R}$ in the category of smooth manifolds.  By definition, a sector $0$-form on $\mathbb{R}$ is simply a smooth map $\omega: \mathbb{R} \to \mathbb{R}$.   By Example \ref{ex1Form}, a sector 1-form on $\mathbb{R}$ is the same as a differential 1-form on $\mathbb{R}$; that is, a map $T\mathbb{R} \to \mathbb{R}$ of the form $\<x,v\> \mapsto f(x) \cdot v$.  By Example \ref{ex2Form}, a sector 2-form on $\mathbb{R}$ is a map $T\mathbb{R} \to \mathbb{R}$ of the form
	\[ \<x,v_1,v_2,d\> \mapsto g(x) \cdot v_1 \cdot v_2 + h(x) \cdot d. \]
	
The exterior derivative of a 0-form $\omega$ is the same as for ordinary differential forms:
	\[ \partial (\omega)(\<x,v\>) = \omega'(x) \cdot v. \]
	
For a sector 1-form $\omega: \<x,v\> \mapsto f(x) \cdot v$ we have $T\omega \s \hat{p} = D\omega$, the directional derivative of $\omega$, which in this case is
	\[ D\omega: \<x,v_1,v_2,d\> \mapsto f'(x) \cdot v_1 \cdot v_2 + f(x)\cdot d \;.\]
So then the exterior derivative of $\omega$ is $\di(\omega) = D\omega \:-\: c_\R \s D\omega$ and hence is given by
	\[ \di(\omega)(x,v_1,v_2,d) = (f'(x) \cdot v_1 \cdot v_2 + f(x)\cdot d) - (f'(x) \cdot v_2 \cdot v_1 + f(x)\cdot d) = 0, \]
since the effect of $c_\R$ is simply to switch the middle two co-ordinates ($v_1$ and $v_2$).  So, in this case, every sector 1-form has exterior derivative $0$.  (Note that this is also automatic since every 1-form is the exterior derivative of a 0-form and the sector forms constitute a complex.  However, it is useful to see how this works explicitly).  

However, the exterior derivative of a sector 2-form is not typically zero.  For a sector 2-form 
	\[ \omega: \<x,v_1,v_2,d\> \mapsto g(x) \cdot v_1 \cdot v_2 + h(x) \cdot d, \]
its directional derivative takes as input an 8-tuple $\<x,v_1,v_2,d_1,v_3,d_2,d_3,t\>$, and maps it to
	\[ g'(x)\cdot v_1 \cdot v_2 \cdot v_3 + h'(x) \cdot d_1 \cdot v_3 + g(x) \cdot v_2 \cdot d_2 + g(x) \cdot v_1 \cdot d_3 + h(x) \cdot t. \]
Now, the effect of $c_{T\mathbb{R}}$ is 
	\[ \<x,v_1,v_2,d_1,v_3,d_2,d_3,t\> \mapsto \<x,v_1,v_3,d_2,v_2,d_1,d_3,t\> \]
and the effect of $Tc_\mathbb{R}$ is
	\[ \<x,v_1,v_2,d_1,v_3,d_2,d_3,t\> \mapsto \<x,v_2,v_1,d_1,v_3,d_3,d_2,t\> \]
Hence since
	\[ \partial (\omega) \;=\; D\omega \:-\: c_{T\R} \s D\omega \:+\: Tc_\R \s c_{T\R} \s D\omega, \]
$\partial (\omega)$ maps  $\<x,v_1,v_2,d_1,v_3,d_2,d_3,t\>$ to
	\[ g'(x) \cdot v_1 \cdot v_2 \cdot v_3 + h'(x)\cdot v_3 \cdot d_1 + (2g(x) - h'(x))\cdot v_2 \cdot d_2 + h'(x)\cdot v_1 \cdot d_3 + h(x) \cdot t. \]
Note that if this is identically zero, then by setting all variables except $t$ to 0, we get $h(x) = 0$, and then by setting all variables but $v_2$ and $d_2$ to 0, we also get $g(x) = 0$.  Thus, a sector 2-form on $\R$ has exterior derivative $0$ if and only it is identically zero..

Taken together, these results tell us the first three sector cohomology groups of $\mathbb{R}$.  The 0th cohomology is the same as ordinary de Rham cohomology (namely, $\mathbb{R}$, since the constant functions are those with derivative $0$).  Similarly, the 1st cohomology is the same (namely, $0$), since every 1-form (sector or differential) is the image of a $0$-form.  Finally, the second sector form cohomology group is also zero, but for a different reason than for de Rham cohomology.  In de Rham cohomology, it is zero since there are no non-trivial 2-forms on $\mathbb{R}$.  For sector forms, it is zero since by the above, the only closed sector 2-form is the zero form.  

It is an open question whether sector form cohomology is always the same as de Rham cohomology; the basic examples given above, however, at least show that the complexes they form are quite different.  We hope to explore the relationship between sector form and de Rham cohomology in a future paper.  

\end{example}

It is also important to note that the individual `derivative' operations $\delta_i^n$ on sector forms appear to have geometric significance: for more on this, see \cite[Chapter 4]{white}.

Returning to our general setting, we shall consider the following further property possessed by some sector forms:

\begin{definition}
If $E$ is subtractive, a \textbf{singular $n$-form} on $M$ with values in $E$ is a sector $n$-form $\omega: T^n M \rightarrow E$ such that $\omega$ is \emph{alternating}; that is, for each $1 \leqslant i \leqslant n-1$,
\[
\bfig
	\square<500,300>[T^nM`E`T^{n}M`E;\omega`c^n_i M`-`\omega]
\efig
\]
\end{definition}
The exterior derivative operation $\partial$ defined above restricts to such singular forms, and so there is also a complex of singular forms (Proposition \ref{prop:complexOfSectorForms}). 

\begin{example}
Let us consider which sector 2-forms on $\mathbb{R}$ are alternating.  By the above, a sector 2-form on $\mathbb{R}$ takes the form
	\[ \omega: \<x,v_1,v_2,d\> \mapsto f(x)\cdot v_1 \cdot v_2 + g(x) \cdot d. \]
For 2-forms, the condition of being alternating amounts to a single equation
	\[ c_\R \s \omega = -\omega. \]
Since $c_\R$ swaps $v_1$ and $v_2$, $\omega$ is alternating if and only if for all $x,v_1,v_2,d$,
	\[ f(x) \cdot v_2 \cdot v_1 + g(x) \cdot d = -(f(x)\cdot v_1 \cdot v_2 + g(x) \cdot d). \]
But this implies that $f(x) = g(x) = 0$, so $\omega$ is constantly zero.  Hence the only singular 2-form on $\mathbb{R}$ is the zero form.
\end{example}

In fact, we can show much more generally that the complex of singular forms on any smooth manifold (with values in $\mathbb{R}$) is isomorphic to its de Rham complex.  We shall prove this in \ref{thm:classical_derham} after first comparing the above singular forms to those studied in synthetic differential geometry \cite{Kock:DiffFormsSDG,KRV,kock}.  Indeed, we shall show that in the tangent category determined by a model of SDG, the above complex of singular forms is isomorphic to its SDG counterpart (\ref{thm:singf_vs_sdgsingf}), and in certain models of SDG the latter complex is known to be isomorphic to the ordinary de Rham complex of differential forms when $M$ is a smooth manifold \cite[IV, Proposition 3.7]{reyes}.

%%%%%%%%%%%%%%%%%%%%%%%%%%%%%%%%%%%%%%%%%%%%%%%%%%%%%%%%%%%%%%%%%%%%%%%%%%%%%%%%%

\section{Symmetric monoids, semigroups, and finite sets}\label{sec:monSemigroups}

In working towards the symmetric cosimplicial structure on sector forms, we will make use of an algebraic structure carried by the tangent endofunctor $T$, namely the structure of a \textit{symmetric semigroup} (\ref{def:symm_semi_mon}).  Many of the results and ideas in this section are due to previous authors \cite{Burr:HdWordPr,Barr,Gra:Symm,Laf:Eq2d,Laf:Bool}, but the applications to tangent categories are new.    

\subsection{Monoids and semigroups} Given a strict monoidal category $\V$, let us denote the unit object of $\V$ by $I$ and write the monoidal product in $\V$ as juxtaposition.  By definition, a \textbf{semigroup} $(S,m)$ in $\V$ is an object $S$ of $\V$ equipped with a morphism $m:SS \rightarrow S$ (called a \textit{multiplication}) that satisfies the following \textbf{associative law}
\begin{equation}\label{eq:assoc}Sm \s m = mS \s m\;.\end{equation}
Explicitly, this means that the composites $SSS \xrightarrow{Sm} SS \xrightarrow{m} S$ and $SSS \xrightarrow{mS} SS \xrightarrow{m} S$ are equal.  A \textbf{monoid} $(S,m,e)$ is a semigroup $(S,m)$ equipped with an additional morphism $e:I \rightarrow S$ (called the \textit{unit}) that satisfies the following \textbf{unit laws}
\begin{equation}\label{eq:unit}eS \s m = 1_S\;\;\;\;\;\;Se \s m = 1_S\;,\end{equation}
noting that $IS = S = SI$ since $\V$ is strict monoidal.

If the given strict monoidal category $\V$ underlies a \textit{symmetric} monoidal category, then we say that a semigroup $(S,m)$ or monoid $(S,m,e)$ in $\V$ is \textbf{commutative} if it satisfies the following \textbf{commutative law}
\begin{equation}\label{eq:comm_law}
s \s m = m
\end{equation}
where $s:SS \rightarrow SS$ is the symmetry isomorphism carried by $\V$.

\subsection{Example: The tangent functor as (co)semigroup}\label{exa:tng_func_sgrp}
Given a tangent category $(\X,\T)$, the tangent functor $T:\X \rightarrow \X$ carries the structure of a semigroup in the monoidal category $[\X,\X]^{\op}$, i.e. the \textit{opposite} of the category $[\X,\X]$ of endofunctors on $\X$.  Indeed the vertical lift $\ell:T \rightarrow TT$ serves as an associative multiplication in $[\X,\X]^{\op}$.  (Note that in general no unit exists to make this semigroup into a monoid.)

\subsection{Monoidal categories of finite cardinals}\label{sec:cats_fin_cards}
Writing $\set$ for the category of sets, let us denote by $\finCard$ the full subcategory of $\set$ whose objects are the finite cardinals, which we identify with their corresponding ordinals and also with the natural numbers $n \in \N$.  The sum $n + m$ of a pair of finite cardinals carries the structure of a coproduct in $\finCard$, where the associated mappings $n \rightarrow n + m$ and $m \rightarrow n + m$ are order preserving and injective and send $n$ and $m$, respectively, onto initial and final segments of the ordinal $n + m$.  In general, if a category $\C$ is equipped with designated binary coproducts and a designated initial object, then $\C$ carries an associated structure of symmetric monoidal category.  In particular, $\finCard$ is therefore symmetric monoidal, with monoidal product $+$ and unit object $0$.  Further, $(\finCard,+,0)$ is a \textit{strict} monoidal category, but note that although $n + m = m + n$ as objects of $\finCard$, the symmetry isomorphism $\sigma_{nm}:n + m \rightarrow m + n$ is not the identity map.

We shall consider several non-full subcategories of $\finCard$ with the same objects as $\finCard$ itself:
\begin{enumerate}[{\rm (i)}]
\item $\finCard_s$, whose morphisms are \textit{surjections};
\item $\finCard_b$, whose morphisms are \textit{bijections}, all of which are automorphisms;
\item $\finOrd$, whose morphisms are \textit{order preserving maps};
\item $\finOrd_s$, whose morphisms are \textit{order preserving surjections}.
\end{enumerate}
Each of these subcategories is closed under the monoidal product in $\finCard$ and hence inherits the structure of a strict monoidal category.  Note that $\finCard_s$ and $\finCard_b$ contain the symmetries $\sigma_{mn}$ and so are symmetric strict monoidal categories, whereas the other subcategories are merely strict monoidal categories.

\subsection{Universal monoids and semigroups}\label{sec:univ_mon_sgrps}

The cardinal $1$ carries the structure of a commutative monoid $(1,\mu,\eta)$ in the symmetric monoidal category $\finCard$, where the associated multiplication $\mu$ and unit $\eta$ are the unique maps
$$\mu:1 + 1 \rightarrow 1\;\;\;\;\;\;\;\;\eta:0 \rightarrow 1\;.$$
Since these maps are order preserving, $(1,\mu,\eta)$ is also a monoid in $\finOrd$.  These monoids and their underlying semigroups have the following universal properties:
\begin{theorem}\label{thm:univ_mon_and_comm_mon}
Let $\V$ be a strict monoidal category.  
\begin{enumerate}[{\rm (i)}]
\item Given a monoid $(S,m,e)$ in $\V$, there is a unique strict monoidal functor $S^\sharp:\finOrd \rightarrow \V$ with $S^\sharp(1) = S$, $S^\sharp(\mu) = m$, and $S^\sharp(\eta) = e$.
\item Given a semigroup $(S,m)$ in $\V$, there is a unique strict monoidal functor $S^\sharp:\finOrd_s \rightarrow \V$ with $S^\sharp(1) = S$ and $S^\sharp(\mu) = m$.
\item \textnormal{(Burroni \cite[2.2]{Burr:HdWordPr}, Grandis \cite[4.1]{Gra:Symm})} If $\V$ is symmetric, then given a commutative monoid $(S,m,e)$ in $\V$, there is a unique symmetric strict monoidal functor $S^\sharp:\finCard \rightarrow \V$ with $S^\sharp(1) = S$, $S^\sharp(\mu) = m$, and $S^\sharp(\eta) = e$.
\item \textnormal{(Lafont \cite[2.3, p. 266]{Laf:Bool})} If $\V$ is symmetric, then given a commutative semigroup $(S,m)$ in $\V$, there is a unique symmetric strict monoidal functor $S^\sharp:\finCard_s \rightarrow \V$ with $S^\sharp(1) = S$ and $S^\sharp(\mu) = m$.
\end{enumerate}
\end{theorem}
\begin{proof}
(i) and (ii) are well-known, e.g. see \cite[VII.5, Proposition 1 and Exercise 3]{MacL}.  We will defer the proofs of (iii) and (iv) until \ref{rem:results_in_symm_mon_setting_follow} below, where we will see that they follow from more general results on the basis of the cited work of Burroni, Grandis, and Lafont.
\end{proof}

Hence, up to a bijection, monoids (resp. semigroups) in strict monoidal categories are the same as strict monoidal functors on $\finOrd$ (resp. $\finOrd_s$), and analogous statements hold for the commutative variants of these notions.  In the terminology of \cite{MacL:CatAlg}, $\finOrd$ is therefore the \textit{PRO} that defines the notion of monoid, and $\finCard$ is the \textit{PROP} that defines the notion of commutative monoid.

\subsection{Symmetric monoids and semigroups}

One of the ramifications of \ref{thm:univ_mon_and_comm_mon}(iii) is that it provides a way to generalize the notion of commutative monoid to the context of \textit{non-symmetric} monoidal categories.  Indeed, work of Burroni \cite[2.2]{Burr:HdWordPr} and of Grandis \cite[\S 2]{Gra:Symm} shows that a strict monoidal functor $\finCard \rightarrow \V$ valued in a mere strict monoidal category $\V$ is equivalently given by a monoid in $\V$ equipped with a compatible \textit{symmetry} isomorphism, per the following definition:

\begin{definition}\label{def:symm_semi_mon}
Let $\V$ be a strict monoidal category.
\begin{enumerate}[{\rm (i)}]
\item A \textbf{symmetry} on an object $S$ of $\V$ is a morphism $s:SS \rightarrow SS$ satisfying the following equations:
\begin{equation}\label{eq:symm}s \s s = 1_{SS}\;\;\;\;\;\;\;\;Ss \s sS \s Ss = sS \s Ss \s sS\;.\end{equation}
\item A \textbf{symmetric semigroup} $(S,m,s)$ in $\V$ consists of a semigroup $(S,m)$ in $\V$ together with a symmetry $s$ on the object $S$ such that the following equation is satisfied
\begin{equation}\label{eq:compat_symm_mult}Sm \s s = sS \s Ss \s mS\end{equation}
and the commutativity law \eqref{eq:comm_law} is also satisfied.
\item (Grandis \cite[\S 2]{Gra:Symm}) A \textbf{symmetric monoid} $(S,m,e,s)$ in $\V$ consists of a monoid $(S,m,e)$ in $\V$ with a symmetry $s$ on $S$ such that $(S,m,s)$ is a symmetric semigroup and the following equation is satisfied:
\begin{equation}\label{eq:comp_symm_unit}eS \s s = Se\;.\end{equation}
\end{enumerate}
One can generalize each of the above notions to the setting of an arbitrary monoidal category $\V$ by inserting associativity and unit isomorphisms as needed.
\end{definition}

\begin{remark}\label{rem:equivalent_defs_symm_sgrp_mnd}
It is readily verified that one obtains an equivalent definition of symmetric semigroup by replacing the equation \eqref{eq:compat_symm_mult} with the equation
\begin{equation}\label{eq:alt_compat_symm_mult}mS \s s = Ss \s sS \s Sm\end{equation}
which appears in \cite[2.2]{Burr:HdWordPr}, \cite[p. 265]{Laf:Bool}, and \cite[3.3]{Laf:Eq2d}.  Similarly, we obtain an equivalent definition of symmetric monoid by replacing the equation \eqref{eq:comp_symm_unit} with
\begin{equation}\label{eq:at_comp_symm_unit}Se \s s = eS\end{equation}
which appears in \cite[2.2]{Burr:HdWordPr} and \cite[3.3]{Laf:Eq2d}.
\end{remark}

\begin{remark}\label{rem:comm_sgrp_yields_symm_sgrp}
Any commutative semigroup $(S,m)$ (resp. commutative monoid $(S,m,e)$) in a symmetric strict monoidal category $\V$ carries the structure of a symmetric semigroup $(S,m,s)$ (resp. symmetric monoid $(S,m,e,s)$) in $\V$ when we take $s$ to be the relevant component of the symmetry isomorphism carried by $\V$.  In particular, the monoid $(1,\mu,\eta)$ in $\finCard$ carries the structure of a symmetric monoid $(1,\mu,\eta,\sigma)$ in $\finCard$, and its underlying symmetric semigroup $(1,\mu,\sigma)$ is also a symmetric semigroup in $\finCard_s$.
\end{remark}

\begin{theorem}\label{thm:univ_symm_mon_sgrp}
Let $\V$ be a strict monoidal category.  
\begin{enumerate}[{\rm (i)}] 
\item \textnormal{(Burroni \cite[2.2]{Burr:HdWordPr}, Grandis \cite[4.1]{Gra:Symm})} Given a symmetric monoid $(S,m,e,s)$ in $\V$, there is a unique strict monoidal functor $S^\sharp:\finCard \rightarrow \V$ with $S^\sharp(1) = S$, $S^\sharp(\mu) = m$, $S^\sharp(\eta) = e$, and $S^\sharp(\sigma) = s$.
\item \textnormal{(Lafont \cite[2.3, p. 266]{Laf:Bool})} Given a symmetric semigroup $(S,m,s)$ in $\V$, there is a unique strict monoidal functor $S^\sharp:\finCard_s \rightarrow \V$ with $S^\sharp(1) = S$, $S^\sharp(\mu) = m$, and $S^\sharp(\sigma) = s$.
\item \textnormal{(Lafont \cite[3.2]{Laf:Eq2d})} The object $1$ carries a symmetry $\sigma$ in $\finCard_b$ that is universal in the sense that if $S$ is an object of  $\V$ and $s$ is a symmetry on $S$, then there is a unique strict monoidal functor $S^\sharp:\finCard_b \rightarrow \V$ with $S^\sharp(1) = S$ and $S^\sharp(\sigma) = s$.
\end{enumerate}
\end{theorem}
\begin{proof}
(i) is explicitly proved in the cited work of Grandis and also follows immediately from the cited earlier result of Burroni.  (ii) follows immediately from the cited result of Lafont, which gives a presentation of the strict monoidal category $\finCard_s$ in terms of the generators $\mu,\sigma$ and the relations for a symmetric semigroup (\ref{def:symm_semi_mon}, \ref{rem:equivalent_defs_symm_sgrp_mnd}).  Similarly, (iii) follows from the cited result of Lafont, which presents the strict monoidal category $\finCard_b$ in terms of the generator $\sigma$ and the relations for a symmetry on an object (\ref{def:symm_semi_mon}).
\end{proof}

\begin{remark}\label{rem:results_in_symm_mon_setting_follow}
We may apply the preceding theorem to \textit{commutative} monoids in symmetric strict monoidal categories $\V$ by way of \ref{rem:comm_sgrp_yields_symm_sgrp}, yielding a proof of \ref{thm:univ_mon_and_comm_mon}(iii).  Similarly we obtain a proof of \ref{thm:univ_mon_and_comm_mon}(iv).
\end{remark}

\subsection{Example: The tangent functor as symmetric (co)semigroup}\label{sec:tng_func_symm_cosgrp}
Given a tangent category $(\X,\T)$, the tangent endofunctor $T:\X \rightarrow \X$ carries the structure of a symmetric semigroup in the opposite $[\X,\X]^{\op}$ of the endofunctor category $[\X,\X]$.  Indeed, within the definition of tangent structure (\ref{defnTangentCategory}), the axioms under the heading \textit{coherence of $\ell$ and $c$} assert precisely that $(T,\ell,c)$ is a symmetric semigroup in $[\X,\X]^{\op}$, where $\ell:T \rightarrow TT$ is the vertical lift and $c:TT \rightarrow TT$ is the canonical flip.

%%%%%%%%%%%%%%%%%%%%%%%%%%%%%%%%%%%%%%%%%%%%%%%%%%%%%%%%%%%%%%%%%%%%%%%%%%%%%%%

\section{Symmetric cosimplicial objects}\label{sec:symCosimp}

The category $\Delta$ of \textit{positive} finite ordinals and order preserving maps admits a geometric interpretation that can be illustrated by way of a well-known functor from $\Delta$ to the category of topological spaces, sending $n$ to the \textit{standard geometric $(n-1)$-simplex} $\Delta_{n - 1} \subseteq \R^n$, i.e. the convex hull of the standard basis vectors in $\R^n$.  Consequently, presheaves on $\Delta$ abound in topology and are called \textit{simplicial sets}.  A similar geometric interpretation applies to each of the categories of finite cardinals that we have considered in \ref{sec:cats_fin_cards}, leading to several corresponding variants of the notion of simplicial set:

\begin{definition}\label{def:varns_on_simpl_obj}
Let $\C$ be a category.
\begin{enumerate}[{\rm (i)}]
\item
\begin{enumerate}
\item A \textbf{degenerative object} $C$ in $\C$ is a functor $\finOrd_s^\op \rightarrow \C$.
\item A \textbf{codegenerative object} $C$ in $\C$ is a functor $C:\finOrd_s \rightarrow \C$.
\end{enumerate}
\item 
\begin{enumerate}
\item An \textbf{(augmented) simplicial object} $C$ in $\C$ is a functor $C:\finOrd^\op \rightarrow \C$.
\item An \textbf{(augmented) cosimplicial object} $C$ in $\C$ is a functor $C:\finOrd \rightarrow \C$.
\end{enumerate}
\item\begin{enumerate}
\item A \textbf{permutative object} $C$ in $\C$ is a functor $C:\finCard_b^\op \rightarrow \C$.
\item A \textbf{copermutative object} $C$ in $\C$ is a functor $C:\finCard_b \rightarrow \C$.
\end{enumerate}
\item
\begin{enumerate}
\item A \textbf{symmetric degenerative object} $C$ in $\C$ is a functor $C:\finCard_s^\op \rightarrow \C$.
\item A \textbf{symmetric codegenerative object} $C$ in $\C$ is a functor $C:\finCard_s \rightarrow \C$.
\end{enumerate}
\item (Barr \cite{Barr}, Grandis \cite{Gra:Symm})
\begin{enumerate}
\item An \textbf{(augmented) symmetric simplicial object}\footnote{The terminology is due to Grandis \cite{Gra:Symm}; Barr \cite{Barr} employed the term \textit{augmented FDP complex}.} $C$ in $\C$ is a functor $C:\finCard^\op \rightarrow \C$.
\item An \textbf{(augmented) symmetric cosimplicial object} $C$ in $\C$ is a functor $C:\finCard \rightarrow \C$.
\end{enumerate}
\end{enumerate}
For brevity, we will omit the modifier ``augmented'' when employing these terms within the present paper.  The \textit{category of degenerative objects} in $\C$ is defined as the functor category $[\finOrd^\op_s,\C]$.  Similarly, each of the listed notions determines an associated category in which the morphisms are arbitrary natural transformations.
\end{definition}

\begin{remark}
Given a category $\C$, any functor $F:\A \rightarrow \B$ determines a functor $[F,\C]:[\B,\C] \rightarrow [\A,\C]$ between the associated categories of $\C$-valued functors.  In particular, the inclusions
\begin{equation}\label{eq:incl_cats_fincards}
\xymatrix{
                           & \finOrd_s \ar@{^{(}->}[d] \ar@{^{(}->}[r] & \finOrd \ar@{^{(}->}[d]\\
\finCard_b \ar@{^{(}->}[r] & \finCard_s \ar@{^{(}->}[r]                & \finCard
}
\end{equation}
induce functors between the various functor categories defined in \ref{def:varns_on_simpl_obj}.  For example, every symmetric degenerative object carries the structure of a permutative object.
\end{remark}

\begin{remark}\label{rem:symm_grp_actions}
By definition, a \textbf{graded object} $C$ in a category $\C$ is a sequence of objects $C_n$ in $\C$ indexed by the finite cardinals $n$.  Observe that a copermutative object $C$ in $\C$ is equivalently described as a graded object $C$ in $\C$ equipped with a sequence of group homomorphisms $S_n \rightarrow \Aut_\C(C_n)$ from the \textbf{symmetric groups} $S_n = \Aut_\finCard(n)$ into the automorphism groups $\Aut_\C(C_n)$ of the objects $C_n$ of $\C$ (\S \ref{sec:notn}).  Dually, a permutative object $C$ in $\C$ is a graded object equipped with group homomorphisms $S^\op_n \rightarrow \Aut_\C(C_n)$ where $S^\op_n$ is the opposite of the symmetric group.  But every group $G$ is isomorphic to its opposite $G^\op$ via the map $(-)^{-1}:G \rightarrow G^\op$, so copermutative objects are in bijective correspondence with permutative objects.  From another perspective, this bijective correspondence is induced by an identity-on-objects isomorphism of categories
$$\finCard_b \xrightarrow{(-)^{-1}} \finCard_b^\op$$
given on arrows by $\xi \mapsto \xi^{-1}$.
\end{remark}

\subsection{Example: The degenerative object of iterated tangent functors}\label{exa:deg_obj_it_tgt}
Given a tangent category $(\X,\T)$, we saw in \ref{sec:tng_func_symm_cosgrp} that the tangent endofunctor $T:\X \rightarrow \X$ carries the structure of a symmetric semigroup $(T,\ell,c)$ in the opposite $[\X,\X]^\op$ of the category of endofunctors on $\X$.  Hence by \ref{thm:univ_symm_mon_sgrp}, this symmetric semigroup determines a corresponding strict monoidal functor $T^\sharp:\finCard_s \rightarrow [\X,\X]^\op$ sending each finite cardinal $n$ to the $n$-th iterate $T^n$ of $T$.   This functor is an example of a symmetric codegenerative object in $[\X,\X]^\op$, equivalently, a symmetric degenerative object $\finCard_s^\op \rightarrow [\X,\X]$ in the category of endofunctors on $\X$.

\subsection{Symmetric simplicial objects by generators and relations}
It is well-known that the category of finite ordinals has a convenient presentation by generators and relations, leading to a familiar equivalent way of defining simplicial sets in terms of \textit{face} and \textit{degeneracy} maps; see, e.g. \cite[VII.5]{MacL}.  Barr \cite{Barr} and Grandis \cite{Gra:Symm} gave an analogous presentation of the larger category of finite cardinals $\finCard$ in terms of the following larger collection of generators:

\begin{definition}\label{def:gens}
\emptybox
\begin{enumerate}[{\rm (i)}]
\item We denote by
\begin{equation}\label{eq:epsilon}\varepsilon^n_i:n + 1 \rightarrow n\;\;\;\;\;\;(n,i \in \N, 1 \leqslant i \leqslant n)\end{equation}
the map $(i - 1) + \mu + (n - i):n + 1 \rightarrow n$ in the notation of \ref{sec:cats_fin_cards}, where $\mu:1 + 1 \rightarrow 1$ is the multiplication carried by $1$ (\ref{sec:univ_mon_sgrps}).  We call these \textbf{codegeneracy maps}.
\item We denote by
\begin{equation}\label{eq:delta}\delta^n_i:n \rightarrow n + 1\;\;\;\;\;\;(n,i \in \N, 1 \leqslant i \leqslant n + 1)\end{equation}
the map $(i - 1) + \eta + (n - i + 1):n \rightarrow n + 1$ in the notation of \ref{sec:cats_fin_cards}, where $\eta:0 \rightarrow 1$ is the unit carried by $1$ (\ref{sec:univ_mon_sgrps}).  We call these \textbf{coface maps}.
\item We denote by
\begin{equation}\label{eq:sigma}\sigma^n_i:n \rightarrow n\;\;\;\;\;\;(n,i \in \N, 1 \leqslant i \leqslant n - 1)\end{equation}
the map $(i - 1) + \sigma + (n - i - 1):n \rightarrow n$ in the notation of \ref{sec:cats_fin_cards}, where $\sigma:1 + 1 \rightarrow 1 + 1$ is the symmetry carried by $1$ (\ref{rem:comm_sgrp_yields_symm_sgrp}).  We call these \textbf{symmetry maps}.
\end{enumerate}
We shall omit the superscripts $n$ when they are clear from the context.
\end{definition}

The following theorem is well-known; for example, a proof is given in \cite[VII.5]{MacL}.  

\begin{theorem}
The category $\finOrd$ of finite ordinals and order preserving maps can be presented by generators and relations (in the sense of \textnormal{\cite[II.8]{MacL}}) as follows:
\begin{enumerate}[{\rm (i)}]
\item Generators: The maps $\varepsilon^n_i, \delta^n_i$ of \ref{def:gens}.
\item Relations: The following \textnormal{\textbf{pure codegeneracy relations}}:
\begin{equation}\label{eq:codegen_rels}\varepsilon_i\s\varepsilon_j = \varepsilon_{j+1}\s\varepsilon_i\;\;\;\;\;\;(i \leqslant j)\end{equation}
together with the following \textnormal{\textbf{pure coface relations}}:
\begin{equation}\label{eq:coface_rels}\delta_j\s\delta_i = \delta_i\s\delta_{j+1}\;\;\;\;\;\;(i \leqslant j)\end{equation}
as well as the following \textnormal{\textbf{coface-codegeneracy relations}}:
\begin{equation}\label{eq:cofcod_rels}\delta_i\s\varepsilon_j =
\begin{cases}
\varepsilon_{j-1}\s\delta_i & (i < j)\\
1                           & (i = j,\;i = j+1)\\
\varepsilon_j\s\delta_{i-1} & (i > j+1).
\end{cases}\end{equation}
\end{enumerate}
\end{theorem}

\begin{remark}\label{rem:finords_gens_rels}
One can also present the category $\finOrd_s$ of finite ordinals and order preserving surjections by generators and relations, namely the codegeneracies $\varepsilon^n_i$ and the pure codegeneracy relations \cite[VII.5, Exercise 3]{MacL}.
\end{remark}

By adding the symmetry maps as additional generators, together with further relations, Barr and Grandis established the following variation on the preceding theorem:

\subsection{Theorem \textnormal{(Barr \cite{Barr}, Grandis \cite[4.2]{Gra:Symm})}}\label{thm:fincard_gens_rels}
\textit{The category $\finCard$ of finite cardinals and arbitrary maps can be presented by generators and relations as follows:
\begin{enumerate}[{\rm (i)}]
\item Generators: The maps $\varepsilon^n_i, \delta^n_i, \sigma^n_i$ of \ref{def:gens}.
\item Relations: The relations \eqref{eq:codegen_rels}, \eqref{eq:coface_rels}, \eqref{eq:cofcod_rels} together with the following \textnormal{\textbf{Moore relations}}
\begin{equation}\label{eq:moore_rels}\sigma_i\s\sigma_i = 1\;\;\;\;\sigma_i\s\sigma_{i+1}\s\sigma_i = \sigma_{i+1}\s\sigma_i\s\sigma_{i+1}\;\;\;\;\;\sigma_j \s \sigma_i = \sigma_i \s \sigma_j\;(i < j - 1).\end{equation}
as well as the following \textnormal{\textbf{codegeneracy-symmetry relations}}:
\begin{equation}\label{eq:codsymm_rels}
\begin{array}{llllllll}
\varepsilon_j\s\sigma_i & = & \sigma_i\s\varepsilon_j\;\;(i < j - 1) & & & \varepsilon_i\s\sigma_i & = & \sigma_{i+1}\s\sigma_i\s\varepsilon_{i+1}\\
\varepsilon_j\s\sigma_i & = & \sigma_{i+1}\s\varepsilon_j\;\;(i > j) & & & \sigma_i\s\varepsilon_i & = & \varepsilon_i\;.
\end{array}
\end{equation}
and the following \textnormal{\textbf{coface-symmetry relations}}:
\begin{equation}\label{eq:cofsymm_rels}\delta_j\s\sigma_i = \sigma_i\s\delta_j\;(i < j-1)\;\;\;\;\;\;\delta_i\s\sigma_i = \delta_{i+1}\;\;\;\;\;\;\delta_j\s\sigma_i = \sigma_{i-1}\s\delta_j\;(i > j)\;.\end{equation}
\end{enumerate}}

\begin{remark}\label{rem:moore_pres_symm_grp}
Grandis \cite[\S 3]{Gra:Symm} notes that for a fixed finite cardinal $n$, the maps $\sigma^n_i$ generate the symmetric group $S_n$, and the Moore relations \eqref{eq:moore_rels} constitute a classical presentation of this group by generators and relations.
\end{remark}

By discarding the coface maps and all the relations involving them, we shall now establish an analogous presentation of $\finCard_s$ in terms of the codegeneracy and symmetry maps:

\begin{theorem}\label{thm:fincards_gens_rels}
The category $\finCard_s$ of finite cardinals and surjections can be presented by generators and relations as follows:
\begin{enumerate}[{\rm (i)}]
\item Generators: The maps $\varepsilon^n_i,\sigma^n_i$ of \ref{def:gens}.
\item Relations: The pure codegeneracy relations \eqref{eq:codegen_rels}, the Moore relations \eqref{eq:moore_rels}, and the codegeneracy-symmetry relations \eqref{eq:codsymm_rels}.
\end{enumerate}
\end{theorem}
\begin{proof}
Let $\C$ denote the category presented by the given (formal) generators and relations (per \cite[II.8]{MacL}), with objects all finite cardinals.  We will not distinguish notationally between the morphisms $\varepsilon^n_i,\sigma^n_i$ in $\finCard_s$ and the generators in $\C$ that bear the same names.

First we show that $\C$ carries the structure of a strict monoidal category.  In order to define a functor $+:\C \times \C \rightarrow \C$, given on objects by addition, it suffices to define functors $(-) + m:\C \rightarrow \C$ and $m + (-):\C \rightarrow \C$ for all $m \in \ob\C$ and check that they satisfy the compatibility condition in \cite[II.3.1]{MacL}.  On generators, we define
$$\varepsilon^n_i + m := \varepsilon^{n+m}_i,\;\;\sigma^n_i + m := \sigma^{n+m}_i,\;\;m + \varepsilon^n_i := \varepsilon^{m+n}_{m+i},\;\;m + \sigma^n_i := \sigma^{m+n}_{m+i},$$
and it follows immediately from the relations \eqref{eq:codegen_rels}, \eqref{eq:moore_rels}, \eqref{eq:codsymm_rels} that these assignments respect these relations (in the sense of \cite[II.8.1]{MacL}) and so define functors as needed.  We must prove that the compatibility condition
$$(\alpha + m) \s\: (n' + \beta) = (n + \beta) \s\: (\alpha + m')$$
holds for arbitrary morphisms $\alpha:n \rightarrow n'$ and $\beta:m \rightarrow m'$ in $\C$, but it suffices to verify this equation in the cases where $\alpha,\beta$ are generators, and in each of these (four) cases the needed equation reduces to an instance of one of the relations in \eqref{eq:codegen_rels}, \eqref{eq:moore_rels}, \eqref{eq:codsymm_rels}.

The resulting functor $+:\C \times \C \rightarrow \C$ clearly satisfies the associativity law on objects.  The verification of the associativity law 
$$(\alpha + \beta) + \gamma = \alpha + (\beta + \gamma):n + m + k \rightarrow n' + m' + k'$$
for arrows $\alpha:n \rightarrow n'$, $\beta:m \rightarrow m'$, and $\gamma:k \rightarrow k'$ in $\C$ reduces to verification of the equations
$$(\alpha + m) + k = \alpha + (m + k),\;\;(n' + \beta) + k = n' + (\beta + k),\;\;(n' + m') + \gamma = n' + (m' + \gamma).$$
It suffices to consider the cases where $\alpha,\beta,\gamma$ are generators, and then the equations are immediate from the definition of $+$.  Verification of the unit laws for $(\C,+,0)$ reduces to showing that the functors $0 + (-),\;(-) + 0:\C \rightarrow \C$ are merely the identity functor, but this is trivially verified on generators.

We claim that the object $1$ of $\C$ carries the structure of a symmetric semigroup $S = (1,\bar{\mu},\bar{\sigma})$ with $\bar{\mu} := \varepsilon^1_1:2 \rightarrow 1$ and $\bar{\sigma} := \sigma^2_1:2 \rightarrow 2$.  The associativity law $(1 + \bar{\mu})\s \bar{\mu} = (\bar{\mu} + 1)\s \bar{\mu}$ is precisely the equation $\varepsilon^2_2 \s \varepsilon^1_1 = \varepsilon^2_1 \s \varepsilon^1_1$, which is one of the pure codegeneracy relations \eqref{eq:codegen_rels}.  The equations \eqref{eq:symm} making $\bar{\sigma}$ a symmetry on $1$ are instances of the Moore relations \eqref{eq:moore_rels}.  The equation \eqref{eq:alt_compat_symm_mult} relating $\bar{\mu}$ and $\bar{\sigma}$ is precisely the equation $\varepsilon^2_1\s\sigma^2_1 = \sigma^3_2\s\sigma^3_1\s\varepsilon^2_2$, which is one of the codegeneracy-symmetry relations \eqref{eq:codsymm_rels}.  The commutative law for $S$ is an instance of the last codegeneracy-symmetry relation \eqref{eq:codsymm_rels}.

Hence by \ref{thm:univ_symm_mon_sgrp}(ii) there is a unique strict monoidal functor $S^\sharp:\finCard_s \rightarrow \C$ with $S^\sharp(1) = 1$, $S^\sharp(\mu) = \bar{\mu}$, and $S^\sharp(\sigma) = \bar{\sigma}$.  Note that $S^\sharp$ is identity-on-objects and sends the morphisms $\varepsilon^n_i,\sigma^n_i$ in $\finCard_s$ to the similarly named generators in $\C$.  Indeed, the definition (\ref{def:gens}) of the morphisms $\varepsilon^n_i,\sigma^n_i$ in $\finCard_s$ entails that the strict monoidal functor $S^\sharp$ sends them to 
$$S^\sharp(\varepsilon^n_i) \;\;=\;\; (i - 1) + \varepsilon^1_1 + (n - i) \;\;=\;\; \varepsilon^{n}_i\;\;\;:\;\;\;n + 1 \rightarrow n$$
$$S^\sharp(\sigma^n_i) \;\;=\;\; (i - 1) + \sigma^2_1 + (n - i - 1) \;\;=\;\; \sigma^n_i\;\;\;:\;\;\;n \rightarrow n\;,$$
respectively (using the definitions of $\bar{\mu}$, $\bar{\sigma}$, and $+$).

Next we define an identity-on-objects functor $M:\C \rightarrow \finCard_s$ by sending the generators $\varepsilon^n_i,\sigma^n_i$ in $\C$ to the similarly named morphisms $\varepsilon^n_i,\sigma^n_i$ in $\finCard_s$.  This assignment respects the relations defining $\C$, simply because the morphisms $\varepsilon^n_i,\sigma^n_i$ in $\finCard_s \hookrightarrow \finCard$ satisfy these relations (by \ref{thm:fincard_gens_rels}).

The composite functor $M \s S^\sharp:\C \rightarrow \C$ preserves the generators $\varepsilon^n_i$ and $\sigma^n_i$ and so (by the universal property of $\C$) must be the identity functor.  Hence $M$ is faithful.  We claim that $M$ is also full (and hence is an isomorphism).  Firstly, every morphism in $\finCard_s$ can be expressed as a composite $\tau \s \alpha:n \rightarrow m$ where $\tau \in S_n$ is a permutation and $\alpha:n \rightarrow m$ is order preserving \cite[\S 3]{Gra:Symm}, and then $\alpha$ is necessarily surjective.  But by \ref{rem:moore_pres_symm_grp} we can express $\tau$ as a composite of symmetry maps $\sigma^n_i$, and by \ref{rem:finords_gens_rels} we can express $\alpha$ as a composite of codegeneracy maps.  Therefore the symmetries and codegeneracies $\sigma^n_i,\varepsilon^n_i$ generate $\finCard_s$, so since they lie in the image of $M$ it follows that $M$ is full.
\end{proof}

As corollaries to the above theorems, we obtain not only the classical description of cosimplicial objects in terms of coface and codegeneracy morphisms but also analogous descriptions of symmetric cosimplicial objects and symmetric codegenerative objects, as follows:

\begin{corollary}\label{thm:pres_cosimpl_obs_gens_rels}
Let $\C$ be a category.
\begin{enumerate}[{\rm (i)}]
\item A cosimplicial object $C:\finOrd \rightarrow \C$ is equivalently given by a graded object $C$ in $\C$ equipped with morphisms
\begin{equation}\label{eq:grob_cod}\varepsilon^n_i:C_{n + 1} \rightarrow C_n\;\;\;\;\;\;(n,i \in \N, 1 \leqslant i \leqslant n)\end{equation}
\begin{equation}\label{eq:grob_cof}\delta^n_i:C_n \rightarrow C_{n + 1}\;\;\;\;\;\;(n,i \in \N, 1 \leqslant i \leqslant n + 1)\end{equation}
in $\C$ that satisfy the equations \eqref{eq:codegen_rels}, \eqref{eq:coface_rels}, \eqref{eq:cofcod_rels}. 
\item \textnormal{(Barr \cite{Barr}, Grandis \cite{Gra:Symm})} A symmetric cosimplicial object $C:\finCard \rightarrow \C$ is equivalently given by a graded object $C$ in $\C$ equipped with morphisms $\varepsilon^n_i$ and $\delta^n_i$ as in \eqref{eq:grob_cod}, \eqref{eq:grob_cof} as well as morphisms
\begin{equation}\label{eq:grob_symm}\sigma^n_i:C_n \rightarrow C_n\;\;\;\;\;\;(n,i \in \N, 1 \leqslant i \leqslant n - 1)\end{equation}
in $\C$ such that these morphisms satisfy the equations \eqref{eq:codegen_rels}, \eqref{eq:coface_rels}, \eqref{eq:cofcod_rels}, \eqref{eq:moore_rels}, \eqref{eq:codsymm_rels}, \eqref{eq:cofsymm_rels}.
\item A symmetric codegenerative object $C:\finCard_s \rightarrow \C$ is equivalently given by a graded object $C$ in $\C$ equipped with morphisms $\varepsilon^n_i$, $\sigma^n_i$ as in \eqref{eq:grob_cod}, \eqref{eq:grob_symm} such that these morphisms satisfy the equations \eqref{eq:codegen_rels}, \eqref{eq:moore_rels}, \eqref{eq:codsymm_rels}.
\end{enumerate}
\end{corollary}

We call the structural morphisms $\varepsilon^n_i,\delta^n_i,\sigma^n_i$ in \ref{thm:pres_cosimpl_obs_gens_rels} \textbf{codegeneracies}, \textbf{cofaces}, and \textbf{symmetries}, respectively, just like their similarly notated counterparts in $\finCard$.  Dually, a symmetric \textit{simplicial} object $C$ carries \textbf{degeneracy morphisms} $\varepsilon^n_i:C_n \rightarrow C_{n+1}$, \textbf{face morphisms} $\delta^n_i:C_{n + 1} \rightarrow C_n$, and \textbf{symmetries} $\sigma^n_i:C_n \rightarrow C_n$.

\subsection{The codegenerative object determined by a symmetric semigroup}\label{sec:symm_cod_obj_det_symm_sgrp}

Given any symmetric semigroup $(S,m,s)$ in a strict monoidal category $\V$, the corresponding strict monoidal functor $S^\sharp:\finCard_s \rightarrow \V$ (\ref{thm:univ_symm_mon_sgrp}) is an example of a symmetric codegenerative object in $\V$.  Its underlying graded object consists of the $n$-fold monoidal powers $S^n$ of $S$.  Since $S^\sharp$ is strict monoidal, the definitions of the generators of $\finCard_s$ in \ref{def:gens} entail that the codegeneracies and symmetries carried by $S^\sharp$ can be expressed as
$$\varepsilon^n_i \;=\; S^{i - 1}m S^{n - i}\;\;:\;\;S^{n + 1} \rightarrow S^n$$
$$\sigma^n_i \;=\; S^{i - 1}sS^{n - i - 1}\;\;:\;\;S^n \rightarrow S^n\;.$$

\subsection{Example: The symmetric degenerative iterated tangent functor}\label{exa:symm_deg_it_tang}

As a special case of \ref{sec:symm_cod_obj_det_symm_sgrp}, we saw in \ref{exa:deg_obj_it_tgt} that the tangent functor $T:\X \rightarrow \X$ on a tangent category $(\X,\T)$ carries the structure of a symmetric semigroup $(T,\ell,c)$ in $[\X,\X]^\op$ and so determines a symmetric codegenerative object $T^\sharp:\finCard_s \rightarrow [\X,\X]^\op$, or equivalently, a symmetric degenerative object
\begin{equation}\label{eq:symm_deg_it_tang}T^{(-)}:\finCard_s^\op \rightarrow [\X,\X],\;\;\;\;n \mapsto T^n\;.\end{equation}
By \ref{sec:symm_cod_obj_det_symm_sgrp}, the associated degeneracy and symmetry morphisms are the natural transformations
\begin{equation}\label{eq:tng_func_deg}\ell^n_i \;=\; T^{i - 1}\ell T^{n - i}\;\;:\;\;T^n \rightarrow T^{n+1}\;\;\;\;\;\;(n,i \in \N, 1 \leqslant i \leqslant n)\end{equation}
\begin{equation}\label{eq:tng_func_symm}c^n_i \;=\; T^{i - 1} c T^{n - i - 1}\;\;:\;\;T^n \rightarrow T^n\;\;\;\;\;\;(n,i \in \N, 1 \leqslant i \leqslant n - 1).\end{equation}

\subsection{Example: Codegenerative sets induced by tangent structure}\label{exa:cod_sets_ind_tngt_str}

Let $(\X,\T)$ be a tangent category.  By transposition, the functor $T^{(-)}$ of \eqref{eq:symm_deg_it_tang} determines a functor
\begin{equation}\label{eq:func_symm_cod_objs_in_tng_cat}\X \rightarrow [\finCard_s^\op,\X],\;\;\;\;M \mapsto T^{(-)}M\end{equation}
valued in the category of symmetric degenerative objects in $\X$.  Explicitly, this functor sends each object $M$ of $\X$ to a symmetric degenerative object
$$T^{(-)}M:\finCard_s^\op \rightarrow \X,\;\;\;\;n \mapsto T^nM$$
whose underlying graded object consists of the total spaces $T^nM$ of the iterated tangent bundles of $M$.  The degeneracy and symmetry morphisms carried by $T^{(-)}M$ are just the components $\ell^n_i M$, $c^n_i M$ at $M$ of those carried by $T^{(-)}$ (\ref{eq:tng_func_deg}, \ref{eq:tng_func_symm}).  

Fixing an object $E$ of $\X$, the functor \eqref{eq:func_symm_cod_objs_in_tng_cat} induces\footnote{Explicitly, \eqref{eq:func_symm_cod_objs_in_tng_cat} determines a functor $\X^\op \rightarrow [\finCard_s,\X^\op]$ which we can then compose with the functor $[\finCard_s,\X(-,E)]:[\finCard_s,\X^\op] \rightarrow [\finCard_s,\set]$ given by composition with the hom-functor $\X(-,E):\X^\op \rightarrow \set$.} a functor
\begin{equation}\label{eq:func_tng_cat_to_symm_cod_sets}\X^\op \rightarrow [\finCard_s,\set]\;,\;\;\;\;M \mapsto \X(T^{(-)}M,E)\end{equation}
valued in the category of symmetric codegenerative sets.  Explicitly, this functor sends each object $M$ of $\X$ to a symmetric codegenerative set $\X(T^{(-)}M,E)$ whose underlying graded set consists of the hom-sets $\X(T^nM,E)$.  The associated codegeneracies and symmetries are the mappings
$$\varepsilon^n_i = \X(\ell^n_i M,E)\;\;:\;\;\X(T^{n+1}M,E) \rightarrow \X(T^nM,E),\;\;\;\;\omega \mapsto \ell^n_iM \s \omega$$
$$\sigma^n_i = \X(c^n_i M,E)\;\;:\;\;\X(T^nM,E) \rightarrow \X(T^nM,E),\;\;\;\;\omega \mapsto c^n_i M \s \omega$$
given by precomposing with the degeneracies $\ell^n_i M:T^n M \rightarrow T^{n+1}M$ and the symmetries $c^n_i M:T^n M \rightarrow T^n M$.  

Note that if $\X$ is a Cartesian tangent catgory and $E$ carries the structure of a commutative monoid (resp. abelian group) object in $\X$, then the representable presheaf $\X(-,E):\X^\op \rightarrow \set$ lifts to a presheaf valued in the category $\cmon$ of commutative monoids (resp. the category $\ab$ of abelian groups).  Hence the functor \eqref{eq:func_tng_cat_to_symm_cod_sets} lifts to a functor
$$\X^\op \rightarrow [\finCard_s,\cmon]\;,\;\;\;\;M \mapsto \X(T^{(-)}M,E)$$
valued in the category of symmetric codegenerative objects in $\cmon$ (resp. $\ab$).

%%%%%%%%%%%%%%%%%%%%%%%%%%%%%%%%%%%%%%%%%%%%%%%%%%%%%%%%

\section{Presenting symmetric cosimplicial objects by fundamental cofaces}\label{sec:presFundamentalCoface}

The coface morphisms $\delta^n_i:C_n \rightarrow C_{n + 1}$ carried by a symmetric cosimplicial object $C$ can be expressed in terms of the \textbf{fundamental cofaces} $\delta^n_1$ by repeated application of the equation $\delta_{i+1} = \delta_i\s\sigma_i$ of \eqref{eq:cofsymm_rels}.  This leads to the following new succinct equational presentation of symmetric cosimplicial objects, which will be useful in establishing the symmetric cosimplicial structure that engenders the de Rham complex:

\begin{theorem}\label{thm:fund_cof_pres_symm_cosimpl_obj}
A symmetric cosimplicial object $C:\finCard \rightarrow \C$ in a category $\C$ is equivalently given by a graded object $C$ in $\C$ equipped with morphisms $\varepsilon^n_i$, $\sigma^n_i$ as in \eqref{eq:grob_cod}, \eqref{eq:grob_symm} together with a sequence of morphisms
\begin{equation}\label{eq:fund_cof}\delta^n_1:C_n \rightarrow C_{n+1}\;\;\;\;\;\;(n \in \N)\end{equation}
such that the equations \eqref{eq:codegen_rels}, \eqref{eq:moore_rels}, \eqref{eq:codsymm_rels} are satisfied along with the following further equations:
\begin{equation}\label{eq:fund_cofcod_eqs}\delta_1\s\varepsilon_1 = 1\;\;\;\;\;\;\;\;\delta_1\s\varepsilon_{j+1} = \varepsilon_j\s\delta_1\end{equation}
\begin{equation}\label{eq:fund_cofsymm_eqs}\delta_1\s\delta_1\s\sigma_1 = \delta_1\s\delta_1\;\;\;\;\;\;\;\;\delta_1\s\sigma_{i + 1} = \sigma_i\s\delta_1\;.\end{equation}
Therefore, in view of \ref{thm:pres_cosimpl_obs_gens_rels}(iii), a symmetric cosimplicial object $C$ in $\C$ is equivalently given by a symmetric codegenerative object $C$ equipped with a sequence of morphisms \eqref{eq:fund_cof} satisfying the equations \eqref{eq:fund_cofcod_eqs}, \eqref{eq:fund_cofsymm_eqs}.
\end{theorem}

Before proving this, let us adopt the following notational conventions.

\begin{notation}\label{def:notn_perms}
\emptybox
\begin{enumerate}[{\rm (i)}]
\item By abuse of notation\footnote{rather than a breach of the well-foundedness axiom of set theory} we write the elements of each finite ordinal $n \in \N$ in ascending order as $1,2,3,...,n$ (rather than $0,1,...,n-1$).
\item For each $n \in \N$ and each $i \in n$, let $\sigma^n_{(i)}$ denote the permutation
$$\sigma^n_{(i)} = \bigl(i (i-1) (i-2) ... 3 2 1\bigr) \in S_n$$
written in cycle notation (i.e., $i \mapsto i-1$, $i-1 \mapsto i-2$, etc.) on the elements $1,...,n$ of the ordinal $n$.  We sometimes omit the superscript $n$, writing just $\sigma_{(i)}$.
\end{enumerate}
Observe that $\sigma_{(i)}$ can be written as a composite
$$\sigma_{(i)} = \sigma_1\s\sigma_2 \s ... \s \sigma_{i-1}$$
of the transpositions $\sigma_j = (j(j+1))$ defined in \ref{def:gens}.  When $i = 1$ we interpret the resulting empty composite as the identity map on $n$, so that $\sigma_{(1)} = 1$.
\end{notation}

\begin{lemma}\label{thm:main_fund_coface_lemma}
Let $C:\finCard_s \rightarrow \C$ be a symmetric codegenerative object (\ref{def:varns_on_simpl_obj}) equipped with a sequence of morphisms $\delta^n_1:C_n \rightarrow C_{n+1}$ $(n \in \N)$ satisfying the equations \eqref{eq:fund_cofcod_eqs}, \eqref{eq:fund_cofsymm_eqs}.  Then $C$ extends uniquely to a symmetric cosimplicial object $C':\finCard \rightarrow \C$ whose fundamental cofaces are the given morphisms $\delta^n_1$.  Explicitly, the cofaces of $C'$ can be expressed in terms of the fundamental cofaces as
\begin{equation}\label{eq:cof_via_fund_cof}\delta^n_i = \delta^n_1 \s \sigma^{n+1}_{(i)}\;\;\;\;\;\;(n \in \N, 1 \leqslant i \leqslant n + 1)\end{equation}
where we write $\sigma^{n+1}_{(i)}:C_{n+1} \rightarrow C_{n+1}$ to denote the image of the automorphism $\sigma^{n+1}_{(i)}:n+1 \rightarrow n+1$ (\ref{def:notn_perms}) under the functor $C:\finCard_s \rightarrow \C$.
\end{lemma}
\begin{proof}
In view of \ref{thm:fincard_gens_rels}, repeated application of the equation $\delta^n_{i+1} = \delta^n_i\s\sigma^{n+1}_i$ of \eqref{eq:cofsymm_rels} shows that the maps $\delta^n_1,\varepsilon^n_i,\sigma^n_i$ generate $\finCard$.  Hence, in view of \ref{thm:fincards_gens_rels}, the uniqueness of $C'$ is immediate if $C'$ exists.  Defining the coface morphisms $\delta^n_i$ for $C'$ by way of the equation \eqref{eq:cof_via_fund_cof}, it suffices (by \ref{thm:pres_cosimpl_obs_gens_rels}) to show that these satisfy the relations \eqref{eq:coface_rels}, \eqref{eq:cofcod_rels}, \eqref{eq:cofsymm_rels} when taken together with the morphisms $\varepsilon^n_i,\sigma^n_i$ carried by $C$.

In order to verify the first coface-codegeneracy relation $\delta_i\s\varepsilon_j = \varepsilon_{j-1}\s\delta_i$ $(i < j)$, we compute as follows, applying the second equation in \eqref{eq:fund_cofcod_eqs} and then repeatedly applying the first equation in \eqref{eq:codsymm_rels}:
$$
\begin{array}{lllll}
\varepsilon_{j-1}\s\delta_i & = & \varepsilon_{j-1}\s\delta_1\s\sigma_1\s...\s\sigma_{i-1} & = & \delta_1\s\varepsilon_j\s\sigma_1\s...\s\sigma_{i-1}\\
                            & = & \delta_1\s\sigma_1\s...\s\sigma_{i-1}\s\varepsilon_j & = & \delta_i\s\varepsilon_j\;.
\end{array}
$$

Next, we prove the second coface-codegeneracy relation $\delta^n_i\s\varepsilon^n_i = 1$ by induction on $i$ (with $n$ fixed).  In the base case where $i = 1$, this holds by assumption.  For the inductive step, suppose that $\delta^n_i\s\varepsilon^n_i = 1$ holds for a given index $i$.  Then by applying one of the codegeneracy-symmetry relations \eqref{eq:codsymm_rels} and the fact that $(\sigma_{i+1})^{-1} = \sigma_{i+1}$ we compute that
$$
\begin{array}{lllll}
\delta_{i+1}\s\varepsilon_{i+1} & = &\delta_1\s\sigma_1\s...\s\sigma_{i-1}\s\sigma_i\s\varepsilon_{i+1} & & \\
                                & = & \delta_1\s\sigma_{(i)}\s\sigma_i\s\varepsilon_{i+1} & = & \delta_1\s\sigma_{(i)}\s\sigma_{i+1}\s\varepsilon_i\s\sigma_i\;.\\
\end{array}
$$
in $\C$.  But $\sigma_{(i)}\s\sigma_{i+1} = \sigma_{i+1}\s\sigma_{(i)}$ in $\C$ since the the permutations $\sigma_{(i)} = (i(i-1)...321)$ and $\sigma_{i+1} = ((i+1)(i+2))$ are disjoint cycles.  Hence we compute as follows, applying the inductive hypothesis and one of the relations in \eqref{eq:fund_cofsymm_eqs}:
$$
\begin{array}{lllllll}
\delta_{i+1}\s\varepsilon_{i+1} & = &\delta_1\s\sigma_{i+1}\s\sigma_{(i)}\s\varepsilon_i\s\sigma_i & = & \sigma_i\s\delta_1\s\sigma_{(i)}\s\varepsilon_i\s\sigma_i & &\\
                                & = & \sigma_i\s\delta_i\s\varepsilon_i\s\sigma_i & = & \sigma_i\s\sigma_i & = & 1\;.\\
\end{array}
$$

The third coface-codegeneracy relation $\delta_{j+1}\s\varepsilon_j = 1$ now follows, using the last co\-de\-ge\-ne\-ra\-cy-symmetry relation \eqref{eq:codsymm_rels}:  $\delta_{j+1}\s\varepsilon_j = \delta_1\s\sigma_{(j+1)}\s \varepsilon_j = \delta_1\s \sigma_{(j)}\s\sigma_j\s\varepsilon_j = \delta_1 \s\sigma_{(j)}\s \varepsilon_j = \delta_j \s \varepsilon_j = 1$.

We next prove the last coface-codegeneracy relation $\delta^{n+1}_i\s\varepsilon^{n+1}_j = \varepsilon^n_j\s\delta^n_{i-1}$ $(i > j+1)$.  By definition, the left-hand side is $\delta_1\s\sigma_{(i)}\s\varepsilon_j:C_{n+1} \rightarrow C_{n+1}$, whereas by applying the second equation in \eqref{eq:fund_cofcod_eqs} we can express the right-hand side as $\varepsilon_j\s\delta_1\s\sigma_{(i-1)} = \delta_1\s\varepsilon_{j+1}\s\sigma_{(i-1)}$.  Hence it suffices to show that
$$\sigma_{(i)}\s\varepsilon_j = \varepsilon_{j+1}\s\sigma_{(i-1)}\;\;\;:\;\;\;n+2 \rightarrow n+1$$
in $\finCard$, but it is straightforward to verify that the left- and right-hand sides of this equation both denote the map $\phi$ given by
$$
\phi(x) = \begin{cases}
  i-1 & (x = 1)\\
  x-1 & (1 < x \leqslant j+1)\\
  x-2 & (j+1 < x \leqslant i)\\
  x-1 & (i < x).
\end{cases}
$$

Next we verify the first coface-symmetry relation $\delta_j^n\s\sigma^{n+1}_i = \sigma^n_i\s\delta^n_j$ $(i < j-1)$.  By definition, the left-hand side is $\delta_1\s\sigma_{(j)}\s\sigma_i$, whereas by applying the second equation in \eqref{eq:fund_cofsymm_eqs} we can express the right-hand side as $\sigma_i\s\delta_1\s\sigma_{(j)} = \delta_1\s\sigma_{i+1}\s\sigma_{(j)}$.  Hence it suffices to show that
$$\sigma_{(j)}\s\sigma_i = \sigma_{i+1}\s\sigma_{(j)}\;\;\;:\;\;\;n+1 \rightarrow n+1$$
in $\finCard$, but it is straightforward to verify that the left- and right-hand sides of this equation both denote the map $\phi$ given by
$$
\phi(x) = \begin{cases}
j & (x = 1)\\
i+1 & (x = i+1)\\
i & (x = i+2)\\
x-1 & (1 < x \leqslant j\;\text{and}\;x \not\in \{i+1,i+2\})\\
x   & (x > j)
\end{cases}
$$

The second coface-symmetry relation $\delta_i\s\sigma_i = \delta_{i+1}$ is almost immediate from the way that we have defined the cofaces, since $\delta_i\s\sigma_i = \delta_1\s\sigma_{(i)}\s\sigma_i = \delta_1\s\sigma_{(i+1)} = \delta_{i+1}$.

Next we establish the third coface-symmetry relation $\delta^n_j\s\sigma^{n+1}_i = \sigma^n_{i-1}\s\delta^n_j$ $(i > j)$.  By definition $\delta_j\s\sigma_i = \delta_1\s\sigma_{(j)}\s\sigma_i$ in $\C$, but since $i > j$ the permutations $\sigma_{(j)} = (j(j-1)...321)$ and $\sigma_i = (i(i+1))$ are disjoint cycles and hence commute.  Thus $\sigma_{(j)}\s\sigma_i = \sigma_i\s\sigma_{(j)}$ in $\C$.  Hence we can compute as follows, applying the second equation in \eqref{eq:fund_cofsymm_eqs} with the knowledge that $i > 1$ (since $i > j \geqslant 1$):
$$\delta_j\s\sigma_i = \delta_1\s\sigma_i\s\sigma_{(j)} = \sigma_{i-1}\s\delta_1\s\sigma_{(j)} = \sigma_{i-1}\s\delta_j\;.$$

Finally, let us verify the pure coface relations $\delta^n_j\s\delta^{n+1}_i = \delta^n_i \s\delta^{n+1}_{j+1}$, where $i \leqslant j$.  By repeatedly applying the equations in \eqref{eq:fund_cofsymm_eqs} we deduce that
$$
\begin{array}{lllll}
\delta_j\s\delta_i & = & \delta_1\s\sigma_1\s ... \s\sigma_{j - 1}\s\delta_1\s\sigma_1\s ...\s\sigma_{i-1} & & \\
& = & \delta_1\s\delta_1\s\sigma_2 \s ... \s \sigma_j\s\sigma_1 \s ...\s \sigma_{i-1} & = & \delta_1\s\delta_1\s\sigma_2\s...\s\sigma_j\s\sigma_{(i)}\\
\delta_i\s\delta_{j+1} & = & \delta_1\s\sigma_1\s...\s\sigma_{i-1}\s\delta_1\s\sigma_1\s...\s\sigma_j & = & \delta_1\s\delta_1\s\sigma_2\s...\s\sigma_i\s\sigma_1\s...\s\sigma_j\\
& = & \delta_1\s\delta_1\s\sigma_1\s\sigma_2\s...\s\sigma_i\s\sigma_1\s...\s\sigma_j & = & \delta_1\s\delta_1\s\sigma_{(i+1)}\s\sigma_{(j+1)}
\end{array}
$$
in $\C$, so it suffices to show that 
\begin{equation}\label{eq:pure_cod_arg}\sigma_2\s...\s\sigma_j\s\sigma_{(i)} = \sigma_{(i+1)}\s\sigma_{(j+1)}\;\;\;:\;\;\;n+2 \rightarrow n+2\end{equation}
in $\finCard$.  Using the fact that $\sigma_2\s...\s\sigma_j = ((j+1)j...32)$ in cycle notation on the elements $1,2,...n+2$ of the set $n+2$, it is straightforward to verify that the left- and right-hand sides of \eqref{eq:pure_cod_arg} both denote the map $\phi$ given by
$$\phi(x) = \begin{cases}
               i   & (x = 1)\\
               j+1 & (x = 2)\\
               x-2 & (3 \leqslant x \leqslant i + 1)\\
               x-1 & (i+1 < x \leqslant j+1)\\
               x   & (x > j+1).
\end{cases}$$

\end{proof}

Using the preceding lemma, we now establish the following equational presentation of the category of finite cardinals, from which Theorem \ref{thm:fund_cof_pres_symm_cosimpl_obj} then immediately follows:

\begin{theorem}\label{thm:fund_cof_pres_fincard}
The category $\finCard$ of finite cardinals and arbitrary maps can be presented by generators and relations as follows:
\begin{enumerate}[{\rm (i)}]
\item Generators: The codegeneracy and symmetry maps $\varepsilon^n_i, \sigma^n_i$ of \ref{def:gens} together with the fundamental coface maps $\delta^n_1$ $(n \in \N)$ of \ref{def:gens}.
\item Relations: \eqref{eq:codegen_rels}, \eqref{eq:moore_rels}, \eqref{eq:codsymm_rels}, \eqref{eq:fund_cofcod_eqs}, \eqref{eq:fund_cofsymm_eqs}.
\end{enumerate}
\end{theorem}
\begin{proof}
By \ref{thm:fincard_gens_rels}, the morphisms $\varepsilon^n_i,\sigma^n_i$ in $\finCard$ satisfy the relations \eqref{eq:codegen_rels}, \eqref{eq:moore_rels}, \eqref{eq:codsymm_rels}, and it is straightforward to verify that the morphisms $\delta^n_1,\varepsilon^n_i,\sigma^n_i$ in $\finCard$ also satisfy the relations \eqref{eq:fund_cofcod_eqs}, \eqref{eq:fund_cofsymm_eqs}.  Indeed, it is easy to check that these relations follow from the coface-codegeneracy relations \eqref{eq:cofcod_rels}, the coface-symmetry relations \eqref{eq:cofsymm_rels}, and the pure coface relations \eqref{eq:coface_rels}, all of which hold in $\finCard$ (\ref{thm:fincard_gens_rels}).

Hence it suffices to show that $\finCard$ has the relevant universal property \cite[\S 8]{MacL}.  But in view of \ref{thm:pres_cosimpl_obs_gens_rels}(iii) this universal property is equivalent to the extension property established in Lemma \ref{thm:main_fund_coface_lemma}.
\end{proof}

%%%%%%%%%%%%%%%%%%%%%%%%%%%%%%%%%%%%%%%%%%%%%%%%%%%%%%%%%%%%%

\section{The symmetric cosimplicial set of sector forms}\label{sec:cosimpOfSectorForms}

Let $E$ be a differential object in a Cartesian tangent category $(\X,\T)$, and let $M$ be an object of $\X$.  Recall that $\Psi_n(M)$ $(n \in \N)$ denotes the set of all sector $n$-forms on $M$ with values in $E$.  In the present section we show that the graded set $(\Psi_n(M))_{n \in \N}$ carries the structure of a symmetric cosimplicial commutative monoid.

\subsection{Remarks on the definition of sector form}\label{sec:rms_def_sector_form}
Recall from Definition \ref{def:sector_nform} that a sector $n$-form on $M$ is a morphism $\omega:T^n M \rightarrow E$ such that for each $j \in \{1,...,n\}$ the equation $a^n_j M\s T\omega = \omega \s \lambda$ holds, where $\lambda:E \rightarrow TE$ is the lift morphism carried by $E$.  Using the results of the previous sections, we can get a better understanding of this equation.  In particular, $a^n_j$ is the composite transformation
$$a^n_j = \left(T^n \overset{\ell_j}{\longrightarrow} T^{n+1} \overset{c_{(j)}}{\longrightarrow} T^{n+1}\right)$$
where $\ell_j = T^{j-1}\ell T^{n-j}$ is the degeneracy morphism carried by the symmetric degenerative object $T^{(-)}:\finCard_s^\op \rightarrow [\X,\X]$ (\ref{exa:symm_deg_it_tang}).  The morphism $c_{(j)}$ is the composite $c_{j-1}\s c_{j-2} \s...\s c_2 \s c_1$, where $c_i = T^{i-1}cT^{n - i}$ denotes the symmetry carried by $T^{(-)}$ (\ref{exa:symm_deg_it_tang}).  Equivalently, $c_{(j)}$ is obtained by applying the functor $T^{(-)}:\finCard_s^\op \rightarrow [\X,\X]$ to the permutation $\sigma^{n+1}_{(j)} = (j(j-1)...321):n+1 \rightarrow n+1$ (\ref{def:notn_perms}) . 

Hence if we define $\alpha^n_j$ as the composite morphism
$$\alpha^n_j := \left(n+1 \overset{\sigma^{n+1}_{(j)}}{\longrightarrow} n+1 \overset{\varepsilon^n_j}{\longrightarrow} n\right)$$
in $\finCard_s$, then
$$\begin{minipage}{3.5in}\textit{$a^n_j:T^n \rightarrow T^{n+1}$ is the image of $\alpha^n_j:n+1 \rightarrow n$ under the functor $T^{(-)}:\finCard_s^\op \rightarrow [\X,\X]$.}\end{minipage}$$
Concretely, one can readily verify that $\alpha^n_j$ is the mapping given by
$$
\alpha^n_j(x) = \begin{cases}
j   & (x = 1)\\
x-1 & (x \neq 1).\\
\end{cases}
$$

\begin{proposition}\label{thm:sector_nforms_submonoid}
For each $n \in \N$, the set $\Psi_n(M)$ of sector $n$-forms is a submonoid
$$\Psi_n(M) \hookrightarrow \X(T^n M, E)$$
of the commutative monoid of all morphisms $\omega:T^n M \rightarrow E$ in $\X$ (\ref{sec:cmons_in_tngnt_cat}).  If $E$ is subtractive, then $\Psi_n(M)$ is a subgroup of the abelian group $\X(T^nM,E)$.
\end{proposition}
\begin{proof}
Given $\omega,\tau \in \Psi_n(M)$, the sum $\omega + \tau$ in $\X(T^n M,E)$ is a sector $n$-form, since for each $j = 1,...,n$ we can compute as follows, using the fact that $\cmon(T):\cmon(\X) \rightarrow \cmon(\X)$ is an additive functor (\ref{prop:cmons_in_tngt_cat}) and the fact that $\lambda:E \rightarrow TE$ is a homomorphism of commutative monoids (\ref{def:diff_objs}):
$$
\begin{array}{lllll}
a_j M \s T(\omega + \tau) & = & a_j M \s (T(\omega) + T(\tau)) & = & (a_j M \s T(\omega)) + (a_j M \s T(\tau))\\
                          & = & (\omega \s \lambda) + (\tau \s \lambda) & = & (\omega + \tau) \s \lambda\;.
\end{array}
$$
Also, the zero element $0$ of the commutative monoid $\X(T^n M,E)$ is a sector $n$-form since we compute that $a_jM \s T(0) = a_jM \s 0 = 0 = 0 \s \lambda$, where each occurrence of $0$ denotes the zero element of the relevant hom-set, again using the additivity of $T$ and the fact that $\lambda$ is a monoid homomorphism.  The remaining claim is verified similarly.
\end{proof}

Recall that the graded commutative monoid $(\X(T^n M,E))_{n \in \N}$ carries the structure of a symmetric codegenerative commutative monoid (\ref{exa:cod_sets_ind_tngt_str}).  We now show that this structure restricts to sector forms:

\begin{proposition}\label{thm:cod_symm_sector_forms}
The codegeneracy and symmetry maps
$$\varepsilon^n_i = \X(\ell^n_i M,E)\;\;:\;\;\X(T^{n+1}M,E) \rightarrow \X(T^nM,E),\;\;\;\;\omega \mapsto \ell^n_iM \s \omega$$
$$\sigma^n_i = \X(c^n_i M,E)\;\;:\;\;\X(T^nM,E) \rightarrow \X(T^nM,E),\;\;\;\;\omega \mapsto c^n_i M \s \omega$$
carried by the symmetric codegenerative commutative monoid $\X(T^{(-)}M,E)$ (\ref{exa:cod_sets_ind_tngt_str}) restrict to yield homomorphisms
$$\varepsilon^n_i\;\;:\;\;\Psi_{n+1}(M) \rightarrow \Psi_n(M)\;\;\;\;\;\;\;\;\;\;\;\;\sigma^n_i\;\;:\;\;\Psi_n(M) \rightarrow \Psi_n(M)$$
between the commutative monoids of sector forms.
\end{proposition}
\begin{proof}
Letting $\omega \in \Psi_{n+1}(M)$, we shall show first that $\varepsilon^n_i(\omega) = \ell^n_i M \s \omega$ is a sector $n$-form on $M$.  It suffices to show that for each $j \in \{1,...,n\}$ there is some $k \in \{1,...,n+1\}$ such that the following diagram commutes.
$$
\xymatrix{
T^n M \ar[d]_{a_j M} \ar[r]^{\ell_i M} & T^{n+1} M \ar[d]_{a_k M} \ar[r]^\omega & E \ar[d]^\lambda\\
T^{n+1}M \ar[r]_{T\ell_i M} & T^{n+2}M \ar[r]_{T\omega} & TE
}
$$
But the rightmost square commutes since $\omega$ is a sector $(n+1)$-form, so it suffices to obtain the commutativity of the following diagram.
$$
\xymatrix{
T^n \ar[r]^{\ell_i} \ar[d]_{a_j} & T^{n+1} \ar[d]^{a_k}\\
T^{n+1} \ar[r]_{T\ell_i} & T^{n+2}
}
$$
In view of \ref{sec:rms_def_sector_form}, \ref{exa:deg_obj_it_tgt}, and \ref{exa:symm_deg_it_tang}, this diagram is obtained by applying the strict monoidal functor $T^{(-)}:\finCard_s^\op \rightarrow [\X,\X]$ to the following diagram in $\finCard_s$
\begin{equation}\label{eq:diag_in_fincards_for_cod_sforms}
\xymatrix{
n & & n+1 \ar[ll]_{\varepsilon^n_i}\\
n+1 \ar[u]^{\alpha_j} & & n+2 \ar[u]_{\alpha_k} \ar[ll]^(.45){1+\varepsilon_i\:=\:\varepsilon_{i+1}}
}
\end{equation}
so it suffices to find $k$ such that this diagram commutes.  In the case where $i < j$ we can take $k = j+1$, whereas in the case where $i \geqslant j$ we can take $k = j$, for in each case it is straightforward to verify that both composites in \eqref{eq:diag_in_fincards_for_cod_sforms} are then equal to the map $\phi$ given by
$$
\phi(x) = \begin{cases}
j   & (x = 1)\\
x-1 & (1 < x \leqslant i + 1)\\
x-2 & (x > i+1).
\end{cases}
$$

Next we prove that if $\omega$ is a sector $n$-form on $M$ then $\sigma^n_i(\omega) = c_i^n M \s \omega$ is a sector $n$-form on $M$.  Letting $j \in \{1,...,n\}$, it suffices to show that the following diagram commutes:
$$
\xymatrix{
T^nM \ar[d]_{a_j M} \ar[r]^{c_i M} & T^nM \ar[r]^\omega & E \ar[d]^\lambda\\
T^{n+1}M \ar[r]_{Tc_i M} & T^{n+1}M \ar[r]_{T\omega} & TE
}
$$
Again since $\omega$ is a sector $n$-form it suffices to show that there is some $k$ such that the following diagram commutes:
$$
\xymatrix{
T^n \ar[d]_{a_j} \ar[r]^{c_i} & T^n \ar[d]^{a_k}\\
T^{n+1} \ar[r]_{Tc_i} & T^{n+1}
}
$$
But as above, we reason that this diagram is obtained by applying $T^{(-)}:\finCard_s^\op \rightarrow [\X,\X]$ to the following diagram in $\finCard_s$
\begin{equation}\label{eq:diag_in_fincards_for_symm_sforms}
\xymatrix{
n & & n \ar[ll]_{\sigma_i}\\
n+1 \ar[u]^{\alpha_j} & & n+1 \ar[ll]^{1 + \sigma_i = \sigma_{i+1}} \ar[u]_{\alpha_k}
}
\end{equation}
and so it suffices to show that this diagram commutes for some $k$.

In the case where $j \notin \{i,i+1\}$ we can take $k = j$, whereas in the case where $j = i$ we can take $k = i + 1$, while in the case where $j = i+1$ we can take $k = i$, for in each of these three cases it is straightforward to verify that both composites in \eqref{eq:diag_in_fincards_for_symm_sforms} are then equal to the mapping $\phi$ given by
$$\phi(x) = \begin{cases}
j   & (x = 1)\\
i+1 & (x = i+1)\\
i   & (x = i+2)\\
x-1 & (x \notin \{1,i+1,i+2\}).
\end{cases}
$$
\end{proof}

\begin{corollary}\label{thm:symm_cod_struct_on_sector_forms}
There is a symmetric codegenerative commutative monoid
$$\Psi(M)\;:\;\finCard_s \rightarrow \cmon,\;\;\;\;n \mapsto \Psi_n(M)$$
where $\Psi_n(M)$ is the commutative monoid of sector $n$-forms on $M$.
\end{corollary}
\begin{proof}
By \ref{thm:cod_symm_sector_forms}, the graded commutative monoid $(\Psi_n(M))$ is equipped with codegeneracy and symmetry homomorphisms $\varepsilon^n_i$ and $\sigma^n_i$.  These are restrictions of the codegeneracy and symmetry maps carried by the symmetric codegenerative set $\X(T^{(-)}M,E)$, so they satisfy the equations listed in \ref{thm:pres_cosimpl_obs_gens_rels}(iii).  Hence an application of \ref{thm:pres_cosimpl_obs_gens_rels}(iii) yields the needed result.
\end{proof}

By the results of section \ref{sec:presFundamentalCoface}, in order to show that the symmetric codegenerative structure on sector forms is part of a symmetric cosimplicial structure, it suffices to define the \textit{fundamental coface} maps $\delta^n_1:\Psi_n(M) \rightarrow \Psi_{n+1}(M)$ and check that they satisfy certain equations (\ref{thm:main_fund_coface_lemma}).  We now proceed to define these maps.

\subsection{The fundamental derivative of a sector form}\label{sec:fund_deriv_sf}

Given a sector $n$-form $\omega:T^n M \rightarrow E$ on $M$, we define the \textbf{fundamental derivative} $\delta_1(\omega)$ of $\omega$ as the following composite morphism
$$\delta_1(\omega) = \left(T^{n+1}M \xrightarrow{T\omega} TE \xrightarrow{\hat{p}} E\right)$$
where $\hat{p}$ denotes the principal projection associated to the differential object $E$ (\ref{sec:princ_projn}).

\begin{theorem}\label{thm:fund_deriv_sector_form}
The fundamental derivative of a sector $n$-form $\omega$ on $M$ is a sector $(n+1)$-form $\delta_1(\omega)$ on $M$.
\end{theorem}
\begin{proof}
Letting $j \in \{1,...,n+1\}$, it suffices to show that the following diagram commutes.
\begin{equation}\label{eq:diag_fund_cof_sf_proof}
\xymatrix{
T^{n+1}M \ar[d]_{a_j M} \ar[r]^{T\omega} & TE \ar[r]^{\hat{p}} & E \ar[d]^{\lambda}\\
T^{n+2}M \ar[r]_{TT\omega} & TTE \ar[r]_{T\hat{p}} & TE
}
\end{equation}
To this end, we proceed in two cases.  First consider the case where $j = 1$.  Then $a^{n+1}_j = a^{n+1}_1 = \ell^{n+1}_1 \s c^{n+2}_{{(1)}} = \ell T^n$, so it suffices to show that the following diagram commutes.
$$
\xymatrix{
T^{n+1}M \ar[d]_{a_1 M\:=\:\ell T^nM} \ar[r]^{T\omega} & TE \ar[d]_{\ell E} \ar[r]^{\hat{p}} & E \ar[d]^\lambda\\
T^{n+2}M \ar[r]_{TT\omega} & TTE \ar[r]_{T\hat{p}} & TE
}
$$
But the leftmost square commutes by the naturality of $\ell$, and the rightmost square commutes by properties of differential objects (Proposition \ref{prop:diffObjs}(v)).  

Now consider the case where $j > 1$.  We want to show that the periphery of the following diagram commutes.
\begin{equation}\label{eq:diag_case_i_neq_j}
\xymatrix{
T^{n+1}M \ar[d]_{a_jM} \ar[r]^{T\omega} & TE \ar[d]_{T\lambda \s cE} \ar[r]^{\hat{p}} & E \ar[d]^\lambda\\
T^{n+2}M \ar[r]_{TT\omega} & TTE \ar[r]_{T\hat{p}} & TE
}
\end{equation}
The rightmost square again commutes by properties of differential objects, namely \ref{prop:diffObjs}(vi).

In order to show that the leftmost square in \eqref{eq:diag_case_i_neq_j} also commutes, it suffices to show that the following diagram commutes.
$$
\xymatrix{
T^{n+1}M \ar[rr]^{T\omega} \ar[dd]_{a_j M} \ar[dr]^{Ta_{j-1} M} & & TE \ar[d]^{T\lambda}\\
& T^{n+2}M \ar[dl]_{cT^n M} \ar[r]^{TT\omega} & TTE \ar[d]^{cE}\\
T^{n+2}M \ar[rr]_{TT\omega} & & TTE
}
$$
The upper-right cell commutes since $\omega$ is a sector $n$-form, and the lower-right cell commutes by the naturality of $c$.  Hence it suffices to show that the following diagram commutes.
$$
\xymatrix{
T^{n+1} \ar[rr]^{Ta_{j-1}} \ar[drr]_{a_j} & & T^{n+2} \ar[d]^{cT^n\:=\:c^{n+2}_1}\\
& & T^{n+2}
}
$$
In view of \ref{sec:rms_def_sector_form}, \ref{exa:deg_obj_it_tgt}, \ref{exa:symm_deg_it_tang}, this diagram is obtained by applying the strict monoidal functor $T^{(-)}:\finCard_s^\op \rightarrow [\X,\X]$ to the following diagram in $\finCard_s$
$$
\xymatrix{
n+1 \ar@{<-}[rr]^{1+\alpha_{j-1}} \ar@{<-}[drr]_{\alpha_j} & & n+2 \ar@{<-}[d]^{\sigma_1}\\
& & n+2
}
$$
which commutes, as one readily verifies directly, using the fact that $1+\alpha_{j-1}$ is given by $1 \mapsto 1$ and $1 + x \mapsto 1 + \alpha_{j-1}(x)$ for all $x \in n+1$.
\end{proof}

\begin{theorem}\label{thm:sym_cosimpl_cmon_sector_forms}
There is a symmetric cosimplicial commutative monoid
$$\Psi(M)\;:\;\finCard \rightarrow \cmon,\;\;\;\;n \mapsto \Psi_n(M)$$
where $\Psi_n(M)$ is the commutative monoid of sector $n$-forms on $M$.  The codegeneracies and symmetries carried by $\Psi(M)$ are those obtained in \ref{thm:cod_symm_sector_forms}, and the fundamental cofaces of $\Psi(M)$ are the maps
$$\delta^n_1\;:\;\Psi_n(M) \rightarrow \Psi_{n+1}(M)\;\;\;\;\;\;(n \in \N)$$
that send a sector $n$-form $\omega$ to its fundamental derivative $\delta_1(\omega) = T\omega \s \hat{p}$ (\ref{thm:fund_deriv_sector_form}).
\end{theorem}
\begin{proof}
By \ref{thm:fund_deriv_sector_form}, the maps $\delta^n_1$ are well-defined, and they are homomorphisms of commutative monoids since $T$ is additive (\ref{prop:cmons_in_tngt_cat}) and $\hat{p}:TE \rightarrow E$ is a homomorphism of commutative monoids (\ref{prop:diffObjs}).  By \ref{thm:symm_cod_struct_on_sector_forms} we have already defined a symmetric codegenerative object $\Psi(M):\finCard_s \rightarrow \cmon$, so by \ref{thm:main_fund_coface_lemma} it suffices to verify the equations \eqref{eq:fund_cofcod_eqs} and \eqref{eq:fund_cofsymm_eqs}, which govern the interaction of the fundamental coface maps $\delta^n_1$ with the codegeneracies and symmetries.

The first equation $\delta_1 \s \varepsilon_1 = 1$ in \eqref{eq:fund_cofcod_eqs} holds, since for each $\omega \in \Psi_n(M)$ we compute that
$$\varepsilon^n_1(\delta^n_1(\omega)) = \ell^n_1 M \s T\omega \s \hat{p} = a^n_1 M \s T\omega \s \hat{p} = \omega \s \lambda \s \hat{p} = \omega$$
using the fact that $\omega$ is a sector form as well as the equations $a^n_1 = \ell^n_1$, $\lambda \s \hat{p}  = 1_E$.

The second equation $\delta_1 \s \varepsilon_{j+1} = \varepsilon_j \s \delta_1$ in \eqref{eq:fund_cofcod_eqs} holds, since for each $\omega \in \Psi_{n+1}(M)$ we compute that
$$\delta^n_1(\varepsilon^n_j(\omega)) = T\ell^n_j M \s T\omega \s \hat{p} = \ell^{n+1}_{j+1}M \s T\omega \s \hat{p} = \varepsilon^{n+1}_{j+1}(\delta^{n+1}_1(\omega))$$
using the fact that $T\ell^n_j = T^j\ell T^{n-j} = \ell^{n+1}_{j+1}$.

The first equation $\delta_1 \s \delta_1 \s \sigma_1 = \delta_1 \s \delta_1$ in \eqref{eq:fund_cofsymm_eqs} holds, since for each $\omega \in \Psi_n(M)$ we compute that
$$
\begin{array}{lllllll}
\sigma^{n+2}_1(\delta^{n+1}_1(\delta^n_1(\omega))) & = & c_1^{n+2} M \s TT\omega \s T\hat{p} \s \hat{p} & = & cT^n M \s TT\omega \s T\hat{p} \s \hat{p} & & \\
& = & TT\omega \s cE \s T\hat{p} \s \hat{p} & = & TT\omega \s T\hat{p} \s \hat{p} & = & \delta^{n+1}_1(\delta^n_1(\omega))
\end{array}
$$
using the naturality of $c$ and the fact that $cE \s T\hat{p} \s \hat{p} = T\hat{p} \s \hat{p}$ (\ref{prop:diffObjs}).

The second equation $\delta_1 \s \sigma_{i+1} = \sigma_i \s \delta_1$ in \eqref{eq:fund_cofsymm_eqs} holds, since for each $\omega \in \Psi_n(M)$ we compute that
$$\delta^n_1(\sigma^n_i(\omega)) = Tc^n_i M \s T\omega \s \hat{p} = c^{n+1}_{i+1} M \s T\omega \s \hat{p} = \sigma^{n+1}_{i+1}(\delta^n_1(\omega))$$
using the fact that $Tc^n_i = T^ic T^{n-i-1} = c^{n+1}_{i+1}$.
\end{proof}

\begin{remark}\label{rmk:der_posi}
By \ref{thm:main_fund_coface_lemma}, the cofaces carried by the symmetric cosimplicial set of sector forms $\Psi(M)$ are the maps
$$\delta^n_i = \delta^n_1 \s \sigma^{n+1}_{(i)}\;\;:\;\;\Psi_n(M) \rightarrow \Psi_{n+1}(M)\;\;\;\;(n \in \N, 1 \leqslant i \leqslant n+1)$$
that send a sector $n$-form to the sector $(n+1)$-form
$$\delta_i(\omega) = \left(T^{n+1}M \xrightarrow{c_{(i)}} T^{n+1}M \xrightarrow{T\omega} TE \xrightarrow{\hat{p}} E\right)$$
which we call the \textbf{derivative of $\omega$ in position $i$}.
\end{remark}

\begin{corollary}\label{thm:symm_cosimpl_ab_grp_sector_forms}
If $E$ is a subtractive differential object in $\X$, then there is a symmetric cosimplicial abelian group
$$\Psi(M)\;:\;\finCard \rightarrow \ab,\;\;\;\;n \mapsto \Psi_n(M)$$
where $\Psi_n(M)$ is the abelian group of sector $n$-forms on $M$ with values in $E$.
\end{corollary}
\begin{proof}
This follows from the preceding theorem and \ref{thm:sector_nforms_submonoid}.
\end{proof}

\begin{corollary}\label{thm:func_sector_forms}
Let $E$ be a differential object in a Cartesian tangent category $(\X,\T)$.
\begin{enumerate}[{\rm (i)}]
\item There is a functor
$$\Psi\;:\;\X^\op \rightarrow [\finCard,\cmon],\;\;\;\;\;\;M \mapsto \Psi(M)$$
that sends each object $M$ of $\X$ to the symmetric cosimplicial commutative monoid $\Psi(M)$ of sector forms on $M$ with values in $E$.
\item If $E$ is a subtractive differential object, then the functor $\Psi$ lifts to a functor valued in the category $[\finCard,\ab]$ of symmetric cosimplicial abelian groups.
\end{enumerate}
\end{corollary}
\begin{proof}
Recall from \ref{exa:cod_sets_ind_tngt_str} that we have a functor $\X^\op \rightarrow [\finCard_s,\cmon]$ that sends each morphism $f:M \rightarrow N$ in $\X$ to the natural transformation $\X(T^{(-)}f,E):\X(T^{(-)}N,E) \Rightarrow \X(T^{(-)}M,E)$ whose components $\X(T^n f,E):\X(T^n N,E) \rightarrow \X(T^n M,E)$ are given by precomposition with $T^nf:T^nM \rightarrow T^nN$.  It follows immediately from the naturality of the transformations $a^n_j:T^n \Rightarrow T^{n+1}$ that $\X(T^n f,E)$ restricts to yield a homomorphism $\Psi_n(f):\Psi_n(N) \rightarrow \Psi_n(M)$ between the submonoids consisting of sector $n$-forms.  We claim that the homomorphisms $\Psi_n(f)$ constitute a natural transformation $\Psi(f):\Psi(N) \Rightarrow \Psi(M)$.  It suffices to verify the naturality condition on the generators $\varepsilon^n_i,\sigma^n_i,\delta^n_1$ of $\finCard$ (\ref{thm:fund_cof_pres_fincard}).  But for the generators $\varepsilon^n_i$, $\sigma^n_i$ this naturality condition follows from the naturality of $\X(T^{(-)}f,E)$, so it suffices to show that
$$
\xymatrix{
\Psi_n(N) \ar[d]_{\delta^n_1} \ar[r]^{\Psi_n(f)} & \Psi_n(M) \ar[d]^{\delta^n_1}\\
\Psi_{n+1}(N) \ar[r]_{\Psi_{n+1}(f)} & \Psi_{n+1}(M)
}
$$
commutes, but this follows immediately from the definitions. 
\end{proof}

%%%%%%%%%%%%%%%%%%%%%%%%%%%%%%%%%%%%%%%%%%%%%%%%%%%%%%%%%%%

\section{Complexes of forms and the exterior derivative}\label{sec:complexes}

Given an (augmented) cosimplicial abelian group $C:\finOrd \rightarrow \ab$, it is well-known \cite[Definition 8.2.1]{weibel} that the underlying graded abelian group $(C_n)_{n \in \N}$ carries the structure of a (non-negatively graded) cochain complex $C_\bullet$ when we define the differential $\partial_n:C_n \rightarrow C_{n+1}$ by
$$\partial_n(c) = \sum_{i = 1}^{n+1} (-1)^{i-1} \delta^n_i(c)\;\;\;\;\;\;\;\;(c \in C_n).$$
In particular, we therefore have that $\partial_n \s \partial_{n+1} = 0$.  We call $C_\bullet$ the \textbf{cochain complex associated to} $C$.

In the present section, we show that when $C$ is a \textit{symmetric} cosimplicial abelian group one also obtains a subcomplex $C_\bullet^\alt \hookrightarrow C_\bullet$ consisting of the \textit{alternating elements} of $C$.  Applied to the symmetric cosimpicial object of sector forms (\ref{thm:symm_cosimpl_ab_grp_sector_forms}), we obtain complexes of sector forms ($C_\bullet$) and singular forms ($C_\bullet^\alt$) in tangent categories.

\subsection{Sector forms and the exterior derivative}

Let $E$ be a subtractive differential object in a Cartesian tangent category $(\X,\T)$, and let $M$ be an object of $\X$.

\begin{definition}\label{def:compl_sector_forms}
\begin{enumerate}[{\rm (i)}]
\item The \textbf{complex of sector forms} on $M$ is defined as the cochain complex $\Psi_\bullet(M)$ associated to the cosimplicial abelian group $\Psi(M)$ of sector forms on $M$ with values in $E$ (\ref{thm:symm_cosimpl_ab_grp_sector_forms}).
\item Given a sector $n$-form $\omega:T^n M \rightarrow E$ on $M$, the \textbf{exterior derivative} of $\omega$ is defined as the sector $(n+1)$-form
$$\partial\omega = \partial_n(\omega) = \sum_{i = 1}^{n+1} (-1)^{i-1} \delta^n_i(\omega)$$
where $\partial_n:\Psi_n(M) \rightarrow \Psi_{n+1}(M)$ is the differential carried by the complex $\Psi_\bullet(M)$, recalling that the sector $(n+1)$-form $\delta^n_i(\omega) = c^{n+1}_{(i)} \s T\omega \s \hat{p}$ is the derivative of $\omega$ in position $i$ (\ref{rmk:der_posi}).
\end{enumerate}
\end{definition}

The following theorem is now immediate, but it would be difficult to prove if we had just defined the exterior derivative directly without first proving Theorem \ref{thm:sym_cosimpl_cmon_sector_forms}:

\begin{theorem}
Let $\omega$ be a sector $n$-form on $M$.  Then $\partial_{n+1}(\partial_n(\omega)) = 0$.
\end{theorem}

\begin{remark}\label{rmk:func_alt_cof_map_complex}
It is well known that the the assignment $C \mapsto C_\bullet$ extends to a functor $(-)_\bullet$ from the category of cosimplicial abelian groups to the category $\cochain_+$ of non-negatively graded cochain complexes\footnote{Although it is well known, the authors are unable to find an explicit statement of precisely this fact in the literature.  However, it can be verified almost immediately, and it can be seen also as a corollary to \cite[\href{http://stacks.math.columbia.edu/tag/0194}{Tag 0194}]{stacks-project}, in view of \cite[\href{http://stacks.math.columbia.edu/tag/018F}{Tag 018F}]{stacks-project}.}.  Hence by \ref{thm:func_sector_forms} we can form the composite functor
$$\Psi_\bullet = \left(\X^\op \xrightarrow{\Psi} [\finCard,\ab] \rightarrow [\finOrd,\ab] \xrightarrow{(-)_\bullet} \cochain_+\right)$$
whose middle factor is the evident forgetful functor.  This functor $\Psi_\bullet$ sends each object $M$ of $\X$ to the complex of sector forms on $M$.
\end{remark}

\subsection{The complex of alternating elements}

Given a symmetric cosimplicial abelian group $C$, we now define a certain subcomplex $C^\alt_\bullet$ of $C_\bullet$.

\begin{definition}\label{def:alt_elt}
Let $n \in \N$.
\begin{enumerate}[{\rm (i)}]
\item We say that an element $c$ of $C_n$ is an \textbf{alternating element} of $C$ if $\sigma^n_i(c) = -c$ for all $i \in \{1,...,n-1\}$, recalling that $\sigma^n_i:C_n \rightarrow C_n$ is the symmetry map carried by $C$.
\item We denote by $C^\alt_n \subseteq C_n$ the subset consisting of all alternating elements.
\end{enumerate}
\end{definition}

\begin{theorem}
Given a symmetric cosimplicial abelian group $C$, the alternating elements of $C$ constitute a subcomplex $C^\alt_\bullet$ of the cochain complex $C_\bullet$ associated to $C$.
\end{theorem}
\begin{proof}
$C_n^\alt \hookrightarrow C_n$ is an intersection of equalizers in $\ab$ and hence is a subgroup inclusion.  Letting $c \in C^\alt_n \subseteq C_n$, it suffices to show that the associated element $\partial(c) = \partial_n(c)$ of $C_{n+1}$ is alternating.  Letting $i \in \{1,...,n\}$ we must show that $\sigma^{n+1}_i(\partial(c)) = -\partial(c)$.  Since $\sigma^{n+1}_i$ is a homomorphism of abelian groups we compute that
$$\sigma^{n+1}_i(\partial(c)) = \sum_{j = 1}^{n+1}(-1)^{j-1}\sigma^{n+1}_i(\delta^n_j(c))\;.$$
Using the coface-symmetry relations \eqref{eq:cofsymm_rels} and the fact that $\sigma^{n+1}_i$ is self-inverse, we compute that
$$
\sigma^{n+1}_i(\delta^n_j(c)) = \begin{cases}
\delta^n_j(\sigma^n_{i-1}(c)) = \delta^n_j(-c) = -\delta^n_j(c) & (j < i)\\
\delta^n_{i+1}(c)                                               & (j = i)\\
\delta^n_i(c)                                                   & (j = i+1)\\
\delta^n_j(\sigma^n_i(c)) = \delta^n_j(-c) = -\delta^n_j(c)    & (j > i+1)\\
\end{cases}
$$
since $c$ is alternating.  Hence, recalling that $\partial(c) = \sum_{j = 1}^{n+1}t_j$ where $t_j = (-1)^{j-1}\delta^n_j(c)$, we compute that
$$\sigma^{n+1}_i(\partial(c)) = \left(\sum_{j < i}-t_j\right) - t_{i+1} - t_i + \left(\sum_{j > i + 1}-t_j\right) = -\partial(c)$$
since $i$ and $i+1$ are of opposite parity.
\end{proof}

\begin{definition}\label{def:compl_alt_elts}
Given a symmetric cosimplicial abelian group $C$, we call the subcomplex $C^\alt_\bullet$ of $C_\bullet$ the \textbf{complex of alternating elements of $C$}.
\end{definition}

\begin{proposition}\label{thm:func_compl_alt_elts}
There is a functor
$$(-)_\bullet^\alt\;:\;[\finCard,\ab] \rightarrow \cochain_+\;\;\;\;\;\;\;\;C \mapsto C_\bullet^\alt$$
from the category of symmetric cosimplicial abelian groups to the category of (non-negatively graded) cochain complexes, sending a symmetric cosimplicial abelian group $C$ to the complex of alternating elements $C^\alt_\bullet$ of $C$.
\end{proposition}
\begin{proof}
Given a morphism of symmetric cosimplicial abelian groups $f:C \rightarrow D$, we claim that the associated morphism of chain complexes $f_\bullet:C_\bullet \rightarrow D_\bullet$ restricts to a morphism $f^\alt_\bullet:C^\alt_\bullet \rightarrow D^\alt_\bullet$ between the subcomplexes of alternating elements.  Indeed, given $c \in C^\alt_n \subseteq C_n$, the associated element $f_n(c)$ of $D_n$ is alternating, since for each $i \in \{1,...,n-1\}$, we compute that $\sigma^n_i(f_n(c)) = f_n(\sigma^n_i(c)) = f_n(-c) = -f_n(c)$ since $f$ is natural and $c$ is alternating.  The result now follows from \ref{rmk:func_alt_cof_map_complex}.
\end{proof}

\subsection{The complex of singular forms}

Again let us fix a subtractive differential object $E$ in a Cartesian tangent category $(\X,\T)$.  Letting $M$ be an object of $\X$, recall that $\Psi(M)$ denotes the symmetric cosimplicial abelian group of sector forms on $M$ (\ref{thm:symm_cosimpl_ab_grp_sector_forms}).

\begin{definition}\label{def:sing_nform_compl}\emptybox
\begin{enumerate}
\item We call a sector $n$-form $\omega$ on $M$ a \textbf{singular $n$-form} if $\omega$ is an alternating element of $\Psi(M)$.  Equivalently, a morphism $\omega:T^nM \rightarrow E$ is a singular $n$-form iff $\omega$ is a sector $n$-form and $c^n_iM \s \omega = -\omega$ for every $i \in \{1,...,n-1\}$, recalling that $c^n_i M = T^{i-1}cT^{n-i-1}M:T^n M \rightarrow T^n M$ is the symmetry carried by $T^{(-)}M$ (\ref{exa:cod_sets_ind_tngt_str}).
\item We denote the complex of alternating elements of $\Psi(M)$ (\ref{def:compl_alt_elts}) by
$$\Omega_\bullet(M) := (\Psi(M))_\bullet^\alt$$
and call it the \textbf{complex of singular forms} on $M$ with values in $E$.
\end{enumerate}
\end{definition}

Since $\Omega_\bullet(M)$ is a subcomplex of $\Psi_\bullet(M)$, the following theorem is now immediate:

\begin{theorem}
The exterior derivative $\partial\omega$ of a singular $n$-form on $M$ is a singular $(n+1)$-form.
\end{theorem}

\begin{proposition}\label{prop:complexOfSectorForms}
There is a functor
$$\Omega_\bullet\;:\;\X^\op \rightarrow \cochain_+\;\;\;\;\;\;\;\;M \mapsto \Omega_\bullet(M)$$
sending each object $M$ of $\X$ to the complex of singular forms on $M$ with values in $E$.
\end{proposition}
\begin{proof}
This follows from \ref{thm:func_sector_forms} and \ref{thm:func_compl_alt_elts}.
\end{proof}

%%%%%%%%%%%%%%%%%%%%%%%%%%%%%%%%%%%%%%%%%

\section{Relationship to de Rham in synthetic and classical differential geometry}\label{sec:relFormsClSDG}

\textit{Synthetic differential geometry} (SDG) is an approach to differential geometry in terms of infinitesimals that was initiated in a lecture of Lawvere in 1967 and developed by several authors, starting with work of Wraith and of Kock \cite{Kock:SimplAxDiff} in the 1970s.  The reader is referred to the books \cite{kock} and \cite{lavendhomme} for a comprehensive introduction to SDG.  An approach to differential forms in SDG was developed in \cite{Kock:DiffFormsSDG,KRV,MoeRey:CohThSDG} (see \cite{kock}, \cite{lavendhomme}), and in the present section we compare this work to the development of differential forms given above, recovering the classical de Rham complex of a smooth manifold as a corollary (\ref{thm:classical_derham}).  This comparison involves specializing our treatment of sector forms to the case in which the tangent structure is \textit{representable} (\ref{def:rep_tan}), an exercise that is illuminating in its own right.

In the most prevalent formulation of SDG, one begins with a topos $\E$ and a commutative ring object $R$ in $\E$, and then one defines $D$ to be the subobject of $R$ described by the equation $x^2 = 0$, so that $D$ is the part of $R$ that consists of square-zero `infinitesimal elements'.  Writing $[M,N]$ for the internal hom between objects $M$ and $N$ of $\E$, one construes the object $[D,M]$ as the space $TM$ of all tangent vectors on $M$.  One postulates that $R$ should satisfy the \textit{Kock-Lawvere} axiom (see section \ref{sec:compn_forms_sdg}), and sometimes further axioms, on the basis of which one can develop much differential geometry in $\E$.  One can define specific toposes $\E$ into which the category $\mf$ of smooth manifolds embeds, via an embedding $\mf \hookrightarrow \E$ that sends the real numbers $\R$ to $R$; see \cite{reyes} and \S\ref{sec:classical_derham} below.

The approach of defining $D$ as the square-zero part of $R$ was put forward in Kock's 1977 paper \cite{Kock:SimplAxDiff}, wherein it is indicated that Lawvere's 1967 lecture did not define $D$ in this way but rather postulated that an object $D$ of infinitesimals should exist and that $[D,M]$ for each object $M$ should have properties expected of the tangent bundle of $M$.  Evidently tangent categories provide an axiomatics of such properties, and indeed Rosick\'y's 1984 paper \cite{rosicky} considers in particular those tangent categories $(\X,\T)$ for which there is an exponentiable object $D$ with $T \cong [D,-]:\X \rightarrow \X$.  This leads to an axiomatics for structure and properties that should be possessed by an object of infinitesimals $D$ \cite[\S 4]{rosicky}, \cite[5.6]{sman3}.  The following definition was given in \cite{sman3} and is a variation on a similar definition given in \cite{rosicky}.  Herein, we say that an endofunctor $F$ on a category $\X$ with finite products is \textbf{representable} if it is isomorphic to an endofunctor of the form $[X,-]:\X \rightarrow \X$ for some exponentiable object $X$ of $\X$, and we then say that $F$ is represented by $X$.

\begin{definition}\label{def:rep_tan}
A category $\X$ carries \textbf{representable tangent structure} if $\X$ has finite products and carries a tangent structure $\T$ in which the endofunctors $T^n$ and $T_n$ $(n \in \N)$ are representable.
\end{definition}

It is proved in \cite[Prop. 5.7]{sman3} that a category $\X$ with finite products carries representable tangent structure if and only if there is an exponentiable object $D$ of $\X$ that carries the structure of an \textbf{infinitesimal object} in the sense of \cite[Def. 5.6]{sman3}.  In this case the tangent endofunctor $T$ is represented by $D$, and its iterates $T^n$ are represented by the $n$-th powers $D^n$ of $D$.  In particular, representable tangent structure is necessarily Cartesian in the sense of \ref{sec:cmons_in_tngnt_cat}.

Let us now fix a category $\X$ with representable tangent structure $\T$, represented by an infinitesimal object $D$ in $\X$.  Throughout, we shall assume without loss of generality that $T = [D,-]$ on the nose.  We shall not assume however that $T^n = [D^n,-]$ for $n > 1$, but rather we now define specific isomorphisms $T^n \cong [D^n,-]$ for use in the sequel.

\begin{definition}\label{def:isos_rep_it_tngnt}
Let us define isomorphisms
$$\psi_n:[D^n,-] \overset{\sim}{\longrightarrow} T^n\;\;\;\;\;\;(n \in \N)$$
by recursion on $n$, as follows.  Firstly, $\psi_0$ is defined as the canonical isomorphism from $[1,-]$ to the identity functor on $\X$.  Next, the components of $\psi_{n+1}$ are defined as the composites
$$\psi_{n+1}M := \left([D^{n+1},M] \xrightarrow{\phi_{n+1}M} [D,[D^n,M]] \xrightarrow{[D,\psi_nM]} [D,T^nM] = T^{n+1}M\right)\;\;\;(M \in \ob\X)$$
where $\phi_{n+1}M$ is the isomorphism whose transpose $[D^{n+1},M] \times D \rightarrow [D^n,M]$ is in turn defined as the transpose of the evaluation morphism 
$$[D^{n+1},M] \times D \times D^n = [D^{n+1},M] \times D^{n+1} \longrightarrow M.$$
Note that $\psi_1$ is therefore the identity transformation on $[D,-] = T$.
\end{definition}

\subsection{Simple type theory and lambda calculus}

In synthetic differential geometry one often makes use of the \textit{internal language} of a given topos $\E$ in order to define morphisms in $\E$ by means of `elementwise' formulae, to show that diagrams commute just by chasing elements, and so on; see, e.g., \cite[Part II]{kock}.  After all, the internal language of $\E$ is a restricted form of set theory, or rather \textit{higher-order intuitionistic type theory} \cite{LS}.

Even though our given tangent category $\X$ is not assumed to be a topos, it still possesses an internal language, albeit a rather restricted one, namely the \textbf{simple type theory} of $\X$ \cite[Chapter 2]{Jac}, considered as a category with finite products.  We now informally review some basic elements of this language and one of its extensions, the \textit{simply typed lambda calculus}; readers who are familiar with the latter may safely skip this section.  

Given any morphism $f:X_1 \times ... \times X_n \rightarrow Y$ in $\X$, we can form a \textit{typing judgment} or \textit{term-in-context}
$$x_1:X_1,...,x_n:X_n \;\;\vdash\;\; f(x_1,...,x_n)\;:\;Y$$
in which each expression $x_i:X_i$ indicates that $x_i$ is a formal variable of type $X_i$.  The typing judgment asserts that the expression $f(x_1,...,x_n)$ is a \textit{term} of type $Y$.  The part of the typing judgement to the left of the turnstile $\vdash$ is called the \textit{context}.  The simple type theory of $\X$ includes various \textit{term formation rules} which allow us to construct new terms-in-context from others \cite[2.1]{Jac}.  For example given terms-in-context $x:X \vdash f(x):Y$ and $y:Y \vdash g(y):Z$ associated to morphisms $f:X \rightarrow Y$ and $g:Y \rightarrow Z$ in $\X$, we can form a term-in-context $x:X \vdash g(f(x)):Z$.  Every term-in-context denotes an associated morphism in $\X$, and in particular, the latter term-in-context denotes the composite morphism $f \s g:X \rightarrow Z$.  The simple type theory of $\X$ carries also a calculus of \textit{equations}
$$x_1:X_1,...,x_n:X_n \;\;\vdash\;\; t_1 = t_2\;\;:\;\;Y$$
where $t_1, t_2$ are terms in the same context $\Gamma$, namely $x_1:X_1,...,x_n:X_n$, and we say that such an equation \textit{holds} in $\X$ if the morphisms in $\X$ denoted by $\Gamma \vdash t_1:Y$ and $\Gamma \vdash t_2:Y$ are equal.  %Several expected deductive rules can thus be soundly applied in reasoning about equality of morphisms in $\X$; see, e.g., \cite[\S 3.2]{Jac}.
We shall often omit typing indications ``$y:Y$'' within terms-in-context and equations when the intended typing is clear.

Having assumed that $\X$ has an infinitesimal object $D$, which is exponentiable, we would also like to employ an internal language in reasoning about exponential transposition of morphisms.  In the case where $\X$ is Cartesian closed, we can employ the \textbf{simply typed lambda calculus} of $\X$ \cite[2.3]{Jac}\cite{LS}, which extends the simply type theory of $\X$ by adding term-formation rules corresponding to exponential transposition, together with corresponding rules governing equality.  For example given a morphism $f:X \times Y \rightarrow Z$ in $\X$, the associated transpose $X \rightarrow [Y,Z]$ is denoted by the term-in-context 
$$x:X \;\;\vdash\;\; \slambda y:Y.f(x,y)\;\;:\;\;[Y,Z]$$ 
where the construct ``$\slambda y:Y.$'' serves to bind the variable $y$ within the scope of the expression $\slambda y:Y.f(x,y)$.  We will often write just $\slambda y.f(x,y)$.

Although the given tangent category $\X$ is not assumed Cartesian closed, we can clearly\footnote{One way of justifying this claim is to embed $\X$ into the Cartesian closed category of presheaves on $\X$ taking values in some universe with respect to which $\X$ is small.  But clearly more elementary approaches are possible.} still employ simply typed lambda calculus and its interpretation in $\X$ as long as the instances of exponential transposition and evaluation employed are those permitted by the exponentiable objects $D^n$.  Given objects $X,Y,Z$ of $\X$ with $X,Y$ exponentiable, we shall write $f:[X,Y],\; g:[Y,Z] \vdash f \s g$ to denote the composition morphism $[X,Y] \times [Y,Z] \rightarrow [X,Z]$ in $\X$.

As a first application of this type-theoretic notation, we record the following:

\begin{proposition}\label{thm:isos_rep_it_tngnt}
For each $n \geqslant 1$ and each object $M$ of $\X$, the isomorphism $\psi_n:[D^n,M] \rightarrow T^nM = [D,[D,...]]$ defined in \ref{def:isos_rep_it_tngnt} is characterized by the following equation:
$$\tau:[D^n,M] \;\;\vdash\;\; \psi_n(\tau) = \slambda d_1.\slambda d_2.\:...\:\slambda d_n.\tau(d_1,d_2,...,d_n)\;\;:\;\;T^nM$$
\end{proposition}

\subsection{Symmetric degenerative structure induced by $D$}

Again fixing a category $\X$ with representable tangent structure, the representing infinitesimal object $D$ carries a structural morphism $\odot:D \times D \rightarrow D$ that makes $D$ a commutative semigroup in $\X$ \cite[Def. 5.6]{sman3}.  Hence $D$ carries the structure of a symmetric semigroup $(D,\odot,s)$ in $\X$ (\ref{def:symm_semi_mon}, \ref{rem:comm_sgrp_yields_symm_sgrp}), where $s:D^2 \rightarrow D^2$ is the symmetry in $\X$.  Within the simple type theory of $\X$, we shall denote the multiplication $\odot$ by $d_1:D, d_2:D \vdash d_1d_2\::\:D$.

Recalling that the vertical lift $\ell:T \rightarrow T^2$ and the canonical flip $c:T^2 \rightarrow T^2$ together equip $T$ with the structure of a symmetric semigroup $(T,\ell,c)$ in $[\X,\X]^\op$ (\ref{sec:tng_func_symm_cosgrp}), it is clear from the discussion in \cite[\S 5.2]{sman3} that this symmetric semigroup structure is induced by the semigroup structure $(D,\odot,s)$ on $D$ in the sense that the following diagrams commute.
\begin{equation}\label{eq:sgrp_str_on_t_ind_by_that_on_d}
\xymatrix{
[D,-] \ar@{=}[d] \ar[r]^{[\odot,-]} & [D^2,-] \ar[d]^{\psi_2}_\wr & & [D^2,-] \ar[d]_{\psi_2}^\wr \ar[r]^{[s,-]} & [D^2,-] \ar[d]^{\psi_2}_\wr\\
T \ar[r]_\ell                       & T^2                         & & T^2     \ar[r]_{c}                   & T^2
}
\end{equation}

Whereas $\X$ is not a strict monoidal category, we can construe $(D,\odot,s)$ as a symmetric semigroup in a Cartesian strict monoidal category $\X_D$ that is defined as follows.  Define $\ob\X_D = \N$ and $\X_D(n,m) = \X(D^n,D^m)$, with composition as in $\X$.  Informally, we will write $D^n$ for the object $n$ of $\X_D$, noting that $\X_D$ is equivalent to the full subcategory of $\X$ on the objects $D^n$.  One encounters no complication in defining a Cartesian strict monoidal structure on $\X_D$, and it is for this reason that we work with $\X_D$ rather than the latter full subcategory of $\X$.

By \ref{thm:univ_symm_mon_sgrp}, the symmetric semigroup $(D,\odot,s)$ in $\X_D$ induces a strict monoidal functor
$$D^\sharp:\finCard_s \rightarrow \X_D,\;\;\;\;\;\;n \mapsto D^n,$$
and we will also write $D^\sharp$ to denote the functor $D^\sharp:\finCard_s \rightarrow \X$ obtained by composing with the canonical fully faithful functor $\X_D \rightarrow \X$.

\begin{proposition}\label{thm:dsharp}
\begin{enumerate}[{\rm (i)}]
\item The codegeneracies and symmetries carried by the symmetric codegenerative object $D^\sharp$ in $\X$ are the following morphisms, respectively:
$$\odot^n_i := D^{i-1}\times \odot \times D^{n-i}:D^{n+1} \rightarrow D^n,\;\;\;\;\;\;s^n_i := D^{i-1} \times s \times D^{n-i-1}:D^n \rightarrow D^n.$$
\item Writing $D^f:D^m \rightarrow D^n$ for the morphism in $\X$ induced by a mapping $f:n \rightarrow m$ between finite cardinals $n,m$, the functor $D^\sharp:\finCard_s \rightarrow \X$ sends each permutation $\xi:n \xrightarrow{\sim} n$ to the automorphism $D^{\xi^{-1}}:D^n \rightarrow D^n$ in $\X$.
\end{enumerate}
\end{proposition}
\begin{proof}
(i) is immediate from \ref{sec:symm_cod_obj_det_symm_sgrp}.  For (ii) it suffices to show that the copermutative object $\finCard_b \rightarrow \X_D$ underlying $D^\sharp$ is equal to the composite
$$\finCard_b \xrightarrow{(-)^{-1}} \finCard_b^\op \xrightarrow{D^{(-)}} \X_D$$
whose first factor is the identity-on-objects isomorphism given in \ref{rem:symm_grp_actions}.  But these two functors $\finCard_b \rightarrow \X_D$ are both strict monoidal, and they both send $1$ to $D$ and send the symmetry $\sigma:2 \rightarrow 2$ to the symmetry $s:D^2 \rightarrow D^2$ in $\X_D$, so by \ref{thm:univ_symm_mon_sgrp}(iii) they are equal.
\end{proof}

\begin{proposition}\label{thm:symm_deg_ind_by_d}
For each object $M$ of $\X$, the isomorphisms $\psi_n:[D^n,M] \rightarrow T^nM$ constitute an isomorphism of symmetric degenerative objects
$$[D^\sharp(-),M] \cong T^{(-)}M\;\;:\;\;\finCard_s^\op \rightarrow \X\;.$$
\end{proposition}
\begin{proof}
Since the codegeneracies $\varepsilon^n_i$ and symmetries $\sigma^n_i$ generate the category $\finCard_s$ (\ref{thm:fincards_gens_rels}), it suffices to show that the isomorphisms $\psi_n$ commute with the degeneracies and symmetries carried by $[D^\sharp(-),M]$ and $T^{(-)}M$.  By \ref{thm:dsharp}, the degeneracies and symmetries carried by $[D^\sharp(-),M]$ are the morphisms $[\odot^n_i,M]:[D^n,M] \rightarrow [D^{n+1},M]$ and $[s^n_i,M]:[D^n,M] \rightarrow [D^n,M]$, and by \ref{exa:symm_deg_it_tang}, \ref{exa:cod_sets_ind_tngt_str} those carried by $T^{(-)}M$ are (the components at $M$ of) the transformations $\ell^n_i = T^{i-1}\ell T^{n-i}$ and $c^n_i = T^{i-1}c T^{n-i-1}$.  We can now use \eqref{eq:sgrp_str_on_t_ind_by_that_on_d} and \ref{thm:isos_rep_it_tngnt} to compute that the following equations hold in $\X$:
$$\tau:[D^n,M] \vdash \ell^n_i(\psi_n(\tau)) = \slambda d_1.\:...\:\slambda d_{n+1}.\tau(d_1,...,d_{i-1},d_id_{i+1},d_{i+2},...,d_{n+1}) = \psi_{n+1}(\odot^n_i \s \tau)$$
$$\tau:[D^n,M] \vdash c^n_i(\psi_n(\tau)) = \slambda d_1.\slambda d_2.\:...\:\slambda d_n.\tau(d_1,...,d_{i-1},d_{i+1},d_i,d_{i+2},...,d_n) = \psi_n(s^n_i \s \tau)$$
\end{proof}

\subsection{Some infinitesimal left actions}\label{sec:infl_actions}

For each $n \in \N$ and each $j = 1,...,n$ we have a morphism $\alpha^n_j = \sigma^{n+1}_{(j)} \s \varepsilon^n_j:n+1 \rightarrow n$ in $\finCard_s$ (\ref{sec:rms_def_sector_form}).  Using \ref{thm:dsharp} and the definition of $\sigma^{n+1}_{(j)} = (j(j-1)...321)$ (\ref{def:notn_perms}), we compute that the morphism $D^\sharp(\sigma^{n+1}_{(j)}):D^{n+1} \rightarrow D^{n+1}$ carried by the symmetric codegenerative object $D^\sharp$ is characterized by
\begin{equation}\label{eq:dsharp_sigmaparensj}(d_0,d_1,...,d_n):D^{n+1} \;\vdash\; (D^\sharp(\sigma^{n+1}_{(j)}))(d_0,...,d_n) = (d_1,...,d_{j-1},d_0,d_j,d_{j+1},...,d_n)\;.\end{equation}
Again applying \ref{thm:dsharp} we therefore compute that the morphism
$$D^\sharp(\alpha^n_j):D \times D^n = D^{n+1} \rightarrow D^n$$
is characterized as follows:
$$(d_0,d_1,...,d_n):D^{n+1} \;\vdash\; (D^\sharp(\alpha^n_j))(d_0,...,d_n) = (d_1,...,d_{j-1},d_0d_j,d_{j+1},...,d_n)\;:\;D^n$$
In effect, $D^\sharp(\alpha^n_j)$ is the left action of $D$ on the $j$-th factor of $D^n$.

Now fixing a differential object $E$ in $\X$, the third axiom in \ref{def:diff_objs} entails that $E$ carries an associative action of $D$, namely the transpose $\bullet:D \times E \rightarrow E$ of the lift morphism $\lambda:E \rightarrow [D,E] = TE$ carried by $E$.  As with the multiplication carried by $D$, we will denote this left action on $E$ by juxtaposition within the lambda calculus.

\begin{definition}\label{def:synth_sf}
Let $M$ be an object of $\X$, and let $n \in \N$.
\begin{enumerate}[{\rm (i)}]
\item For each $j = 1,...,n$, let $\bullet^n_j:D \times [D^n,M] \rightarrow [D^n,M]$ denote the transpose of the composite
$$[D^n,M] \xrightarrow{[D^\sharp(\alpha^n_j),M]} [D^{n+1},M] \xrightarrow{\phi_{n+1}} [D,[D^n,M]]$$
where $\phi_{n+1}$ is the isomorphism defined in \ref{def:isos_rep_it_tngnt}.  Writing $\bullet^n_j$ in infix notation in the lambda calculus, this morphism $\bullet^n_j$ is characterized by the following equation:
$$d:D, \tau:[D^n,M]\;\vdash\; d \bullet^n_j \tau = \slambda (d_1,...,d_n).\tau(d_1,...,d_{j-1},dd_j,d_{j+1},...,d_n)$$
\item We say that a morphism $\nu:[D^n,M] \rightarrow E$ in $\X$ is a \textbf{synthetic sector $n$-form} on $M$ with values in $E$ if
\begin{equation}\label{eq:synth_sf_multilin_ax}d:D, \tau:[D^n,M] \;\vdash\; \nu(d \bullet^n_j \tau) = d\nu(\tau)\;\;:\;\;E\end{equation}
holds in $\X$ for each $j = 1,...,n$, i.e. if the following diagram commutes:
\begin{equation}\label{eq:diag_synth_sf}
\xymatrix{
D \times [D^n,M] \ar[d]_{\bullet^n_j} \ar[r]^(.55){D \times \nu} & D \times E \ar[d]^\bullet\\
[D^n,M] \ar[r]_\nu                                          & E
}
\end{equation}
\end{enumerate}
\end{definition}

\begin{theorem}\label{thm:symm_cosimpl_cmon_syn_sf}
For each object $M$ of $\X$, there is a symmetric cosimplicial commutative monoid $\Psi^\syn(M)$ in which $\Psi^\syn_n(M)$ is the set of all synthetic sector $n$-forms on $M$ with values in $E$ $(n \in \N)$.  Further, there is an isomorphism of symmetric cosimplicial commutative monoids
$$\Psi(M) \cong \Psi^\syn(M)$$
between sector forms on $M$ and synthetic sector forms on $M$.  The codegeneracies and symmetries carried by $\Psi^\syn(M)$ are given by pre-composition with the degeneracies and symmetries carried by $[D^\sharp(-),M]$ (\ref{thm:symm_deg_ind_by_d}).
\end{theorem}
\begin{proof}
For each $n \in \N$, the isomorphism $\psi_n:[D^n,M] \rightarrow T^nM$ induces a bijection between morphisms $\nu:[D^n,M] \rightarrow E$ and morphisms $\omega:T^nM \rightarrow E$.  Given a pair of morphisms $\nu,\omega$ that correspond under this bijection, so that $\nu = \psi_n \s \omega$, we claim that $\nu$ is a synthetic sector $n$-form if and only if $\omega$ is a sector $n$-form.  To prove this, first observe that the following diagram commutes, by the inductive definition of the isomorphisms $\psi_n$ (\ref{def:isos_rep_it_tngnt}).
$$
\xymatrix{
[D^{n+1},M] \ar[dr]_{\psi_{n+1}} \ar[r]^{\phi_{n+1}}_\sim & [D,[D^n,M]] \ar[d]^{T\psi_n} \ar[dr]^{T\nu \:=\: [D,\nu]} & \\
                                                          & T^{n+1}M \ar[r]_(.45){T\omega}                & {[D,E] = TE}
}
$$
Next recall that $\omega$ is a sector $n$-form if and only if $a^n_j \s T\omega = \omega \s \lambda$ for all $j = 1,...,n$ (\ref{def:sector_nform}).  Here $a^n_j:T^nM \rightarrow T^{n+1}M$ is obtained by applying the functor $T^{(-)}M:\finCard_s^\op \rightarrow \X$ to the morphism $\alpha^n_j:n+1 \rightarrow n$ in $\finCard_s$ (\ref{sec:rms_def_sector_form}), so the naturality of the isomorphism $\psi:[D^\sharp(-),M] \xrightarrow{\sim} T^{(-)}M$ (\ref{thm:symm_deg_ind_by_d}) entails that $\psi_n \s a^n_j = [D^\sharp(\alpha^n_j),M] \s \psi_{n+1}:[D^n,M] \rightarrow T^{n+1}M$.  Using these facts we readily deduce that $\omega$ is a sector $n$-form if and only if the following diagram commutes for each $j = 1,...,n$.
$$
\xymatrix{
[D^n,M] \ar[d]_{[D^\sharp(\alpha^n_j),M]} \ar[rr]^\nu & & E \ar[d]^\lambda\\
[D^{n+1},M] \ar[r]_{\phi_{n+1}}^\sim & [D,[D^n,M]] \ar[r]_(.55){[D,\nu]} & [D,E]
}
$$
But the two composites in this diagram are the exponential transposes of the two composites in the diagram \eqref{eq:diag_synth_sf} whose commutativity characterizes synthetic sector $n$-forms.

Hence we have a bijection $\Psi_n(M) \xrightarrow{\sim} \Psi^\syn_n(M)$ given by $\omega \mapsto \psi_n \s \omega$, and since $\Psi_n(M)$ is a submonoid of $\X(T^nM,E)$ it follows that $\Psi^\syn_n(M)$ is a submonoid of $\X([D^n,M],E)$ and the given bijection is an isomorphism of commutative monoids.  Hence in view of \ref{thm:sym_cosimpl_cmon_sector_forms} there is a unique functor $\Psi^\syn(M):\finCard \rightarrow \cmon$ given on objects by $n \mapsto \Psi^\syn_n(M)$ such that the given isomorphisms $\Psi_n(M) \cong \Psi^\syn_n(M)$ are natural in $n \in \finCard$.  The naturality of these isomorphisms together with the naturality of the isomorphism $\psi:[D^\sharp(-),M] \xrightarrow{\sim} T^{(-)}M$ (\ref{thm:symm_deg_ind_by_d}) entails the remaining claim.
\end{proof}

\begin{proposition}\label{thm:eltwise_formula_ext_deriv}
Given a synthetic sector $n$-form $\nu:[D^n,M] \rightarrow E$ and any $i = 1,...,n+1$, the $i$-th coface of $\nu$ is the synthetic sector $(n+1)$-form $\delta^n_i(\nu):[D^{n+1},M] \rightarrow E$ characterized by the following equation:
$$\tau:[D^{n+1},M] \;\vdash\; (\delta^n_i(\nu))(\tau) = \hat{p}\bigl(\slambda d_0\colon D.\nu(\slambda(d_1,...,d_n)\colon D^n.\tau(d_1,...,d_{i-1},d_0,d_i,...,d_n))\bigr)$$
where $\hat{p}:[D,E] = TE \rightarrow E$ is the principal projection (\ref{sec:princ_projn}).
\end{proposition}
\begin{proof}
By \ref{thm:main_fund_coface_lemma} we know that $\delta^n_i(\nu) = \sigma^{n+1}_{(i)}(\delta^n_1(\nu))$ where $\sigma^{n+1}_{(i)}:\Psi^\syn_{n+1}(M) \rightarrow \Psi^\syn_{n+1}(M)$ is the automorphism induced by the permutation $\sigma^{n+1}_{(i)}:n+1 \rightarrow n+1$ in $\finCard$.

Given any synthetic sector $(n+1)$-form $\gamma:[D^{n+1},M] \rightarrow E$, we deduce by \ref{thm:symm_cosimpl_cmon_syn_sf} that $\sigma^{n+1}_{(i)}(\gamma)$ is the composite 
$$[D^{n+1},M] \xrightarrow{[D^\sharp(\sigma^{n+1}_{(i)}),M]} [D^{n+1},M] \overset{\gamma}{\longrightarrow} E,$$
so by using \eqref{eq:dsharp_sigmaparensj} we compute as follows:
\begin{equation}\label{eq:deriv_syn_sf1}\tau:[D^{n+1},M] \;\vdash\; (\sigma^{n+1}_{(i)}(\gamma))(\tau) = \gamma(\slambda (d_0,...,d_n).\tau(d_1,...,d_{i-1},d_0,d_i,...,d_n))\end{equation}

Let $\omega:T^nM \rightarrow E$ be the sector $n$-form corresponding to $\nu$, so that
$$\nu = \left([D^n,M] \xrightarrow{\psi_n} T^nM \xrightarrow{\omega} E\right).$$
In view of the proof of \ref{thm:symm_cosimpl_cmon_syn_sf}, $\delta^n_1(\nu):[D^{n+1},M] \rightarrow E$ is the composite 
$$[D^{n+1},M] \xrightarrow{\psi_{n+1}} T^{n+1}M \xrightarrow{\delta^n_1(\omega)} E$$
where $\delta^n_1(\omega) = T\omega \s \hat{p}$ is the fundamental derivative of $\omega$ (\ref{sec:fund_deriv_sf}, \ref{thm:sym_cosimpl_cmon_sector_forms}).  Hence since $\omega = \psi^{-1}_n \s \nu$ we compute that
$$\delta^n_1(\nu) \;=\; \psi_{n+1} \s T(\psi_n^{-1}) \s T\nu \s \hat{p} \;=\; \phi_{n+1} \s T\nu \s \hat{p} \;=\; \phi_{n+1} \s [D,\nu] \s \hat{p}$$
by the inductive definition of $\psi$ (\ref{def:isos_rep_it_tngnt}).  Hence we compute as follows:
$$\label{eq:formula_for_fund_deriv_synsf}\tau:[D^{n+1},M] \;\vdash\; (\delta^n_1(\nu))(\tau) = \hat{p}(\phi_{n+1}(\tau) \s \nu) = \hat{p}(\slambda d_0.\nu(\slambda (d_1,...,d_n).\tau(d_0,d_1,...,d_n)))$$
Applying this together with \eqref{eq:deriv_syn_sf1} in the case where $\gamma = \delta^n_1(\nu)$, we obtain the needed result.
\end{proof}

Now let us assume that $E$ is a subtractive differential object in $\X$, and again let $M$ be an object of $\X$.

\begin{definition}\label{def:synth_singf}
\begin{enumerate}
\item We call a synthetic sector $n$-form $\nu$ on $M$ a \textbf{synthetic singular $n$-form} if $\nu$ is an alternating element of the symmetric cosimplicial abelian group $\Psi^\syn(M)$ (\ref{def:alt_elt}).
\item We define the cochain complex $\Omega^\syn_\bullet(M)$ as the complex of alternating elements (\ref{def:compl_alt_elts}) of $\Psi^\syn(M)$, and we call $\Omega^\syn_\bullet(M)$ the \textbf{complex of synthetic singular forms} on $M$ with values in $E$.
\end{enumerate}
\end{definition}

By \ref{thm:symm_cosimpl_cmon_syn_sf} and \ref{thm:dsharp}, the symmetry $\sigma^n_i:\Psi^\syn_n(M) \rightarrow \Psi^\syn_n(M)$ $(i = 1,...,n-1)$ sends each synthetic sector $n$-form $\nu$ to the composite 
$$[D^n,M] \xrightarrow{[s^n_i,M]} [D^n,M] \overset{\nu}{\longrightarrow} E$$
in the notation of \ref{thm:dsharp}.  Hence we obtain the following:

\begin{proposition}
A morphism $\nu:[D^n,M] \rightarrow E$ is a synthetic singular $n$-form iff $\nu$ is a synthetic sector $n$-form and
\begin{equation}\label{eq:synth_alt_ax}\tau:[D^n,M] \;\vdash\; \nu\bigl(\lambda(d_1,...,d_n).\tau(d_1,...,d_{i-1},d_{i+1},d_i,d_{i+2},...,d_n)\bigr) \:=\: -\nu(\tau)\;\;:\;\;E\end{equation}
holds for each $i = 1,...,n-1$, where $-:E \rightarrow E$ is the negation morphism.
\end{proposition}

\begin{theorem}\label{thm:compl_singf_iso_synsingf}
There is an isomorphism of cochain complexes
$$\Omega_\bullet(M) \cong \Omega^\syn_\bullet(M)$$
between the complex of singular forms and the complex of synthetic singular forms.
\end{theorem}
\begin{proof}
This follows immediately from \ref{thm:symm_cosimpl_cmon_syn_sf}, \ref{thm:func_compl_alt_elts}, and \ref{def:sing_nform_compl}.
\end{proof}

\subsection{Relationship to de Rham in SDG}\label{sec:compn_forms_sdg}

Let $\E$ be a topos equipped with a commutative ring object $R$.  Writing the multiplication in $R$ as juxtaposition in the lambda calculus, let $(-)^2,0:R \rightarrow R$ denote the morphisms in $\E$ denoted by the terms-in-context $x:R \vdash xx:R$ and $x:R \vdash 0:R$, respectively.  Let $D \hookrightarrow R$ denote the equalizer of these morphisms $(-)^2,0$.

An $R$-module object $E$ in $\E$ is said to be a \textbf{Kock-Lawvere $R$-module}, or satisfy the \textbf{Kock-Lawvere axiom}, if the morphism
\begin{equation}\label{eq:kl_compn_morph}E \times E \xrightarrow{\kappa} [D,E]\end{equation}
denoted by
$$(e_1,e_0):E \times E \;\vdash\; \slambda d:D\:.\:de_1 + e_0$$
is an isomorphism, where we have written the action of $R$ on $E$ as juxtaposition.  $R$ is said to be a \textbf{ring of line type} if $R$ itself satisfies the Kock-Lawvere axiom when considered as an $R$-module.

Let us now assume that $R$ is a ring of line type.  The books \cite[I.14]{kock}, \cite[IV]{reyes}, \cite[Ch. 4]{lavendhomme} treat differential forms valued in a Kock-Lawvere $R$-module $E$, following \cite{Kock:DiffFormsSDG,KRV,MoeRey:CohThSDG}.  We now recall the definition of the notion of differential form employed in the cited sections of these books.  We shall soon show that these are exactly the same as the synthetic singular forms defined in \ref{def:synth_singf} above.

\begin{definition}\label{def:sdg_sing_form}
Let $M$ be an object of $\E$, and let $n \in \N$.  For each $j = 1,...,n$, define a morphism $*^n_j:R \times [D^n,M] \rightarrow [D^n,M]$ by
$$r:R, \tau:[D^n,M]\;\vdash\; r *^n_j \tau = \slambda (d_1,...,d_n).\tau(d_1,...,d_{j-1},rd_j,d_{j+1},...,d_n)$$
where we have written $*^n_j$ in infix notation.  We say that a morphism $\nu:[D^n,M] \rightarrow E$ is an \textbf{SDG singular $n$-form} on $M$ with values in $E$ if the equation \eqref{eq:synth_alt_ax} holds for all $i = 1,...,n-1$ and the following equation holds for all $j = 1,...,n$:
\begin{equation}\label{eq:sdg_sing_form_multilin_ax}r:R, \tau:[D^n,M] \;\vdash\; \nu(r *^n_j \tau) = r\nu(\tau)\;\;:\;\;E\end{equation}
\end{definition}

\begin{remark}
Observe that the axioms for an SDG singular form are almost exactly the same as those for a synthetic singular form as defined in \ref{def:synth_singf} above, except that for SDG singular forms the axiom \eqref{eq:sdg_sing_form_multilin_ax} applies to arbitrary scalars $r:R$ rather than just $d:D$.  We shall show that these notions of form are identical in a suitable setting.  The key idea is as follows.
\end{remark}

\begin{proposition}\label{thm:pres_dact_implies_pres_ract}
If a morphism $h:E \rightarrow F$ in $\E$ between Kock-Lawvere $R$-modules $E,F$ satisfies the equation $d:D, e:E \vdash h(de) = dh(e)$, then $h$ satisfies the equation $r:R, e:E \vdash h(re) = rh(e)$.
\end{proposition}
\begin{proof}
Let $\rho,\tau:R \times E \rightarrow F$ be the morphisms denoted by $r:R, e:E \vdash h(re)$ and $r:R, e:E \vdash rh(e)$, respectively.  Note that the multiplicative action $D \times R \rightarrow R$ factors through the inclusion $D \hookrightarrow R$ since $d:D, r:R \vdash (dr)^2 = d^2r^2 = 0r^2 = 0$ holds.  Denoting the resulting morphism $D \times R \rightarrow D$ by $d:D, r:R \vdash dr:D$, we compute that
$$d:D, r:R, e:E \vdash dh(re) = h(dre) = drh(e)\;:\;F$$
since $h$ preserves the actions by $D$.  Letting $\lambda:F \rightarrow [D,F]$ be the morphism denoted by $f:F \vdash \slambda d.df$, we compute that
$$r:R, e:E \vdash \lambda(\rho(r,e)) = (\slambda d.dh(re)) = (\slambda d.drh(e)) = \lambda(\tau(r,e))\;:\;[D,F]$$
so $\rho \s \lambda = \tau \s \lambda$.  But $\lambda$ is the composite $F \xrightarrow{(1_F,0)} F \times F \xrightarrow{\kappa} [D,F]$ in the notation of \eqref{eq:kl_compn_morph}, so $\lambda$ is a split monomorphism and hence $\rho = \tau$.
\end{proof}

\subsection{The tangent category of microlinear objects}\label{sec:tngt_cat_microlin_objs}

In order to be able to use certain results given in \cite{lavendhomme}, we shall now assume that $R$ satisfies the \textbf{Kock-Weil axiom (K-W)} of \cite[2.1.3]{lavendhomme}.  This axiom entails the above Kock-Lawvere axiom, and it also entails that the object $R$ of $\E$ is \textbf{microlinear} \cite[2.3.1]{lavendhomme} (that is, $R$ perceives finite quasi-colimits of infinitesimal objects as colimits).  

Let
$$\X \hookrightarrow \E$$
denote the full subcategory of $\E$ consisting of the microlinear objects, and let $\E_{iv} \hookrightarrow \E$ denote the full subcategory consisting of those objects that are both \textit{infinitesimally linear} and \textit{vertically linear} in the sense used in \cite[5.2, 5.3]{sman3}.  Note that $\X$ is contained in $\E_{iv}$.  By \cite[5.4]{sman3} $\E_{iv}$ carries representable tangent structure, with representing object $D$.  Both $\X$ and $\E_{iv}$ are closed under finite limits and exponentials in $\E$ (\cite[2.3.1]{lavendhomme}, \cite[5.4]{sman3}).  Hence the tangent structure on $\E_{iv}$ restricts to $\X$, and
$$\begin{minipage}{3.6in}\textit{$\X$ is a Cartesian closed category with representable tangent structure, represented by $D$.}\end{minipage}$$

Now let us fix an object $M$ of $\X$ and a Kock-Lawvere $R$-module $E$ that lies in $\X$.  By \cite[3.9]{diffBundles}, $E$ carries the structure of a subtractive differential object in $\X$, where the associated lift morphism $\lambda:E \rightarrow TE = [D,E]$ is the transpose of the restricted action $D \times E \rightarrow E$, so that the latter is the morphism written as $\bullet$ in \ref{sec:infl_actions}.

Given any object $X$ of $\E$, let $X^*:\E = \E \slash 1 \rightarrow \E \slash X$ denote the functor given by pullback along $!:X \rightarrow 1$.  Since $X^*$ is a logical functor between toposes \cite[1.42]{Joh:TT}, $X^*$ sends $R$ to a ring of line type $X^*(R) = (\pi_2:R \times X \rightarrow X)$ in $\E \slash X$, and $X^*$ sends $E$ to a Kock-Lawvere $X^*(R)$-module $X^*(E)$ in $\E \slash X$.

\begin{proposition}\label{thm:kl_rmods_structs_it_tan}
\emptybox
\begin{enumerate}[{\rm (i)}]
\item Given any object $X$ of $\X$, the tangent bundle $(TX,p_X)$ carries the structure of a Kock-Lawvere $X^*(R)$-module in $\E \slash X$.
\item For each $n \in \N$ and each $j = 1,...,n+1$, the morphism 
$$p_j:[D^{n+1},M] \rightarrow [D^n,M]$$
given by
$$\tau:[D^{n+1},M] \;\vdash\; p_j(\tau) = \slambda (d_1,...,d_n).\tau(d_1,...,d_{j-1},0,d_j,...,d_n)$$
carries the structure of a Kock-Lawvere $([D^n,M])^*(R)$-module in $\E \slash [D^n,M]$.  The associated action of $([D^n,M])^*(R)$ is given by the morphism
$$*^{n+1}_j:R \times [D^{n+1},M] \rightarrow [D^{n+1},M]$$
of \ref{def:sdg_sing_form}.
\end{enumerate}
\end{proposition}
\begin{proof}
(i) is established in \cite[3.1.2, Prop. 4]{lavendhomme}, noting that Lavendhomme's universal quantification over points of the base space is interpreted in an internal sense in $\E$, cf. \cite[II.6]{kock}.  Letting $X = [D^n,M]$ and $Y = [D^{n+1},M]$, we have an isomorphism $\phi_{n+1}:Y \xrightarrow{\sim} [D,X] = TX$ (\ref{def:isos_rep_it_tngnt}) that commutes with the projections $p_1$ and $p_X$ to $X$.  Hence by (i), $(Y,p_1)$ carries the structure of a Kock-Lawvere $X^*(R)$-module in $\E \slash X$, and in view of \cite[3.1.1]{lavendhomme} the associated action of $X^*(R)$ is $*^{n+1}_1$.  By \eqref{eq:dsharp_sigmaparensj} we have an automorphism $\Sigma_{(j)} = [D^\sharp(\sigma^{n+1}_{(j)}),M]:Y \xrightarrow{\sim} Y$ in $\X \hookrightarrow \E$ characterized by
$$\tau:[D^{n+1},M] \;\vdash\; \Sigma_{(j)}(\tau) = \slambda (d_1,...,d_{n+1}).\tau(d_2,...,d_j,d_1,d_{j+1},...,d_{n+1})\;.$$
Observe that $\Sigma_{(j)}$ is in fact an isomorphism $\Sigma_{(j)}:(Y,p_j) \xrightarrow{\sim} (Y,p_1)$ in $\E \slash X$.  Hence there is a unique $X^*(R)$-module structure on $(Y,p_j)$ such that $\Sigma_{(j)}$ is an isomorphism of $X^*(R)$-modules $(Y,p_j) \xrightarrow{\sim} (Y,p_1)$.  Hence $(Y,p_j)$ is a Kock-Lawvere $X^*(R)$-module, and the result follows.
\end{proof}

\begin{theorem}\label{thm:sdgsing_vs_synthsing}
A morphism $\nu:[D^n,M] \rightarrow E$ is an SDG singular $n$-form if and only if $\nu$ is a synthetic singular $n$-form in the sense of \ref{def:synth_singf}.
\end{theorem}
\begin{proof}
If $n = 0$ then the result is immediate, so we may assume that $n \geqslant 1$.  Since the morphisms $\bullet^n_j$ of \ref{def:synth_sf} are restrictions of the morphisms $*^n_j$ of \ref{def:sdg_sing_form}, it suffices to assume that $\nu$ is a synthetic singular $n$-form and show that \eqref{eq:sdg_sing_form_multilin_ax} holds for each $j \in \{1,...,n\}$.  Letting $Y = [D^n,M]$ and $X = [D^{n-1},M]$, we deduce by \ref{thm:kl_rmods_structs_it_tan} that $(Y,p_j)$ is a Kock-Lawvere $X^*(R)$-module with action morphism $*^n_j:R \times Y \rightarrow Y$.  Also $X^*(E) = (E \times X,\pi_2)$ is a Kock-Lawvere $X^*(R)$-module, and $\nu$ induces a morphism $\bar{\nu} := (\nu,p_j):(Y,p_j) \rightarrow X^*(E)$ in $\E \slash X$.  Note that since $X^*$ preserves limits, $X^*$ sends $D \hookrightarrow R$ to the square zero part $X^*(D) \hookrightarrow X^*(R)$ of $X^*(R)$.  But the axiom \eqref{eq:synth_sf_multilin_ax} entails that $\bar{\nu}$ preserves the $X^*(D)$-actions carried by the Kock-Lawvere $X^*(R)$-modules $(Y,p_j)$ and $X^*(E)$, so $\bar{\nu}$ preserves the $X^*(R)$-actions by \ref{thm:pres_dact_implies_pres_ract}.  Hence \eqref{eq:sdg_sing_form_multilin_ax} holds.
\end{proof}

\begin{theorem}\label{thm:singf_vs_sdgsingf}
The complex of singular forms $\Omega_\bullet(M)$ is isomorphic to the complex $\Omega^{\textnormal{sdg}}_\bullet(M)$ of SDG singular forms on $M$ with values in $E$, as defined in \cite{Kock:DiffFormsSDG,KRV}, \cite[Ch. 4]{lavendhomme}.  Further, $\Omega^{\textnormal{sdg}}_\bullet(M)$ is identical to the complex of synthetic singular forms $\Omega^\syn_\bullet(M)$.
\end{theorem}
\begin{proof}
By \ref{thm:compl_singf_iso_synsingf}, it suffices to prove the second claim.  By \ref{thm:sdgsing_vs_synthsing}, the complexes $\Omega^{\textnormal{sdg}}_\bullet(M)$ and $\Omega^\syn_\bullet(M)$ have the same underlying graded abelian group.  The differential $\partial_n:\Omega^\syn_n(M) \rightarrow \Omega^\syn_{n+1}(M)$ is given by $\partial_n(\nu) = \sum_{i=1}^{n+1}(-1)^{n-1}\delta^n_i(\nu)$, and by consulting the formula for $\delta^n_i(\nu)$ in \ref{thm:eltwise_formula_ext_deriv} we find that $\partial_n(\nu)$ is precisely the exterior derivative of $\nu$ described in \cite[\S 1, (1.3)]{KRV} and \cite[4.2.3, Prop. 4]{lavendhomme}.
\end{proof}

\subsection{Relationship to de Rham for smooth manifolds}\label{sec:classical_derham}

Let $\mf$ denote the category of (Hausdorff, second-countable) smooth manifolds, and let $M$ be an object of $\mf$.  Since $\mf$ is a Cartesian tangent category and $\R$ is a subtractive differential object in $\mf$, we can consider the complex of singular forms $\Omega_\bullet(M)$ on $M$ with values in $\R$ (\ref{def:sing_nform_compl}).  In order to show that $\Omega_\bullet(M)$ is isomorphic to the classical de Rham complex of $M$\footnote{The relationship between singular forms and classical differential forms is also briefly discussed in an exercise in White's book; see \cite[pg. 116, exercise 7]{white}.}, we shall consider an embedding of $\mf$ into a topos $\E$ modelling SDG, and then we shall invoke \ref{thm:singf_vs_sdgsingf} and a result on differential forms within this specific topos \cite[IV.3.7]{reyes}.

In particular, we shall take $\E$ to be \textit{the Dubuc topos}, i.e., the topos denoted by $\mathcal{G}$ in \cite{reyes} and by $\widetilde{\mathcal{B}^\op}$ in \cite{kock}.  Explicitly, $\E$ is the topos of sheaves on the opposite of the category of germ-determined, finitely generated $C^\infty$-rings, with respect to the open cover topology.  But more to the point, $\E$ is a topos equipped with an embedding
$$\iota:\mf \hookrightarrow \E$$
such that $R = \iota(\R)$ is a ring of line type satisfying the Kock-Weil axiom \cite[8.3.3]{lavendhomme}, and $\iota$ has several further pleasant properties making $(\E,\iota)$ a \textit{well-adapted model} (see \cite[III.3, III.4, III.8.4]{kock}).  We now show that any well-adapted model gives rise to an embedding of tangent categories:

\begin{proposition}\label{thm:well_ad_model_yields_str_mor_tngt_cats}
Let $\iota:\mf \hookrightarrow \E$ be a well-adapted model of SDG.
\begin{enumerate}[{\rm (i)}]
\item The embedding $\iota$ factors as $\mf \overset{\iota'}{\hookrightarrow} \X \hookrightarrow \E$, where $\X$ is the full subcategory of $\E$ consisting of microlinear objects.
\item The embedding $\iota':\mf \hookrightarrow \X$ carries the structure of a strong morphism of Cartesian tangent categories (in the sense of \cite[2.7, 2.8]{sman3}).
\end{enumerate}
\end{proposition}
\begin{proof}
For any manifold $M \in \ob\mf$, the object $\iota(M)$ of $\E$ is a \textit{formal manifold} \cite[I.17]{kock} (by \cite[III.3.4, III.4.C]{kock}), so by \cite[I, Notes 2006, Footnote 27; Appendix D]{kock}), $\iota(M)$ is microlinear and (i) is proved.  Letting $T$ denote the tangent endofunctor on $\mf$, it is proved in \cite[III.4.1]{kock} that there is an isomorphism $\alpha_M:\iota(TM) \xrightarrow{\sim} [D,\iota(M)]$ natural in $M \in \mf$, recalling that $[D,-]:\X \rightarrow \X$ is the tangent endofunctor on $\X$ (\ref{sec:tngt_cat_microlin_objs}).  Also, by \cite[III.3.A]{kock}, $\iota$ preserves finite products as well as the (iterated) pullbacks that are given as part of the tangent structure on $\mf$ (\ref{defnTangentCategory}), since these are all \textit{transversal pullbacks} (as each can be described as a pullback of a submersion; see \cite{diffBundles}).  But $\X$ is closed under finite limits in $\E$ (\ref{sec:tngt_cat_microlin_objs}), so it remains only to show that $\alpha$ satisfies the diagrammatic axioms\footnote{Note that composition of functors is written in diagrammatic order in the cited source.} of \cite[2.7]{sman3}.

Letting $\cart \hookrightarrow \mf$ denote the full subcategory consisting of the \textit{Cartesian spaces} $\R^n$, we know that the Cartesian tangent structure on $\mf$ restricts to a Cartesian tangent structure on $\cart$, wherein $T(\R^n) = \R^n \times \R^n$.  Indeed, $\cart$ is a particularly simple kind of tangent category, namely (the tangent category associated to) a \textit{Cartesian differential category} \cite[4.1, 4.2]{sman3}.  The embedding $\iota$ sends the Cartesian spaces $\R^n$ to Kock-Lawvere $R$-modules $\iota(\R^n) \cong R^n$ in $\E$, where $R = \iota(\R)$ is the line object in $\E$.  The microlinear Kock-Lawvere $R$-modules in $\E$ are precisely the differential objects of $\X$ \cite[3.9]{diffBundles} and so constitute a Cartesian differential category $\Diff(\X)$ \cite[4.11]{sman3}, such that the inclusion $\Diff(\X) \hookrightarrow \X$ carries the structure of a strong morphism of Cartesian tangent categories \cite[4.12]{sman3}.  The restriction $\iota'':\cart \hookrightarrow \Diff(\X)$ of $\iota'$ preserves finite products, and $\iota''$ preserves the Jacobian derivative \cite[III.3.3]{kock} and therefore preserves Cartesian differential structure.  Therefore $\iota''$ is a strong morphism of Cartesian tangent structure, so the composite $\cart \hookrightarrow \Diff(\X) \hookrightarrow \X$ is a strong morphism of Cartesian tangent structure whose structural isomorphisms 
$$\iota(\R^n \times \R^n) = \iota(T\R^n) \xrightarrow{\sim} [D,\iota(\R^n)]$$
are the isomorphisms $\alpha_M$ with $M = \R^n$ as constructed in \cite[III.4.1]{kock}.

Hence the restriction of $\iota':\mf \hookrightarrow \X$ to $\cart$ is a strong morphism of Cartesian tangent structure.  Now for an arbitrary manifold $M \in \ob\mf$, we can choose a covering by open embeddings $e_j:U_j \hookrightarrow M$ $(j \in J)$ where the $U_j$ are Cartesian spaces.  It follows that the families $(Te_j)_{j \in J}$, $(T^2e_j)_{j \in J}$, and $(T_2e_i)_{j \in J}$ (in the notation of \ref{defnTangentCategory}) are coverings by open embeddings.  For each of the structural transformations $t \in \{p,0,+,\ell,c\}$ carried by $\mf$, with corresponding transformation $t'$ in $\X$, we can now use the fact that $\alpha_{U_j}$ commutes with $t_{U_j},t'_{U_j}$ for each $j \in J$ to show that $\alpha_M$ commutes with $t_M,t'_M$.
\end{proof}

\begin{theorem}\label{thm:classical_derham}
The classical de Rham complex $\Omega^{\textnormal{dR}}_\bullet(M)$ of a smooth manifold $M$ is isomorphic to the complex $\Omega_\bullet(M)$ of singular forms on $M$ with values in $\R$ (\ref{def:sing_nform_compl}), where $M$ is considered as an object of the tangent category $\mf$.  In symbols,
$$\Omega^{\textnormal{dR}}_\bullet(M) \;\cong\; \Omega_\bullet(M)\;.$$
\end{theorem}
\begin{proof}
Taking $\E$ to be the Dubuc topos, we deduce by \ref{thm:well_ad_model_yields_str_mor_tngt_cats} that the associated embedding $\iota':\mf \hookrightarrow \X$ is a strong morphism of Cartesian tangent structure, and we can invoke \ref{thm:singf_vs_sdgsingf} and \cite[IV.3.7]{reyes} in order to compute that
$$\Omega_\bullet(M) \cong \Omega_\bullet(\iota'(M)) \cong \Omega^{\textnormal{sdg}}_\bullet(\iota(M)) \cong \Omega^{\textnormal{dR}}_\bullet(M)$$
as complexes, where $\Omega_\bullet(\iota'(M))$ is the complex of singular forms on $\iota'(M)$ in $\X$ with values in $R = \iota'(\R)$.
\end{proof}

\section{Conclusions and future work}\label{sec:futureWork}

In this paper, we have shown that not only do tangent categories support a generalization of de Rham cohomology, but that they support a second cohomology, the cohomology of sector forms; furthermore, sector forms have a rich algebraic structure that goes beyond this cohomology.  There are many possible extensions of this work.
\begin{itemize}
	\item We have shown that tangent categories possess a cohomology of sector forms.  Even in the canonical case of smooth manifolds, this may be distinct from the ordinary de Rham cohomology of classical differential forms; further investigation is required to compare these cohomologies.
	\item The relationship between classical differential forms and singular forms in an arbitrary tangent category needs to be better understood.  In general, one would expect that any object $M$ which is ``locally a differential object'' would have the property that classical differential forms on $M$ and singular forms on $M$  would be in bijective correspondence, but this requires detailed work to check.  Another possibility is that differential forms and singular forms may correspond if $M$ possesses a ``symmetric $n$-connection'' \cite{diffFormsLavendhommeNishimura}, suitably defined in a tangent category.  
	\item An important operation on differential forms is the wedge product.  Since this involves multiplication in $\mathbb{R}$, in the setting of tangent categories, one would need the coefficient object $E$ to have ring structure.  Once such a generalized wedge product is defined, one could consider how such an operation interacts with the co-face, symmetry, and co-degeneracy maps.
	\item It is a well-known result that the exterior derivative is the unique map from $n$ forms to $n+1$ forms satisfying certain algebraic properties \cite[Proposition 7.11]{spivakVol1}.   It would be interesting to determine for which tangent categories this uniqueness result holds.  
	\item It is not clear to what cohomology theories the cohomologies found here correspond in algebraic geometry (for example, in the category of schemes).  The cohomologies may recover an existing cohomology theory or represent a new one; further investigation is required.   
\end{itemize}
Finally, as mentioned in the introduction, sector forms generalize covariant tensors, and because of this, White writes that ``the \emph{calculus} of [sector forms] can serve as a unified framework for the presentation of classical local Riemannian geometry, and that it can lead to new methods of analysis in modern differential geometry'' \cite[pg. x]{white}.  The results presented here on sector forms contribute to this calculus by means of a methodology which is applicable more generally.  
%and thus have intrinsic value in differential geometry and its generalizations.  

\bibliography{simplicialDeRham}

\end{document}